\documentclass[11pt]{article}
\usepackage[a4paper,margin=27mm]{geometry}
\usepackage[T1]{fontenc}
\usepackage{lmodern,amsmath,amssymb,amsthm,mathtools,microtype}
\usepackage[dvipsnames]{xcolor}
\usepackage[colorlinks=true,linkcolor=MidnightBlue,citecolor=MidnightBlue,urlcolor=MidnightBlue]{hyperref}
\usepackage{booktabs,array}
\usepackage{enumitem}
\setlist{itemsep=3pt,topsep=5pt}
\numberwithin{equation}{section}
\newtheorem{theorem}{Theorem}[section]
\newtheorem{proposition}[theorem]{Proposition}
\newtheorem{lemma}[theorem]{Lemma}
\newtheorem{corollary}[theorem]{Corollary}
\theoremstyle{remark}
\newtheorem{remark}[theorem]{Remark}
\newcommand{\R}{\mathbb R}
\newcommand{\T}{\mathbb T}
\newcommand{\E}{\mathcal E}
\newcommand{\Dcal}{\mathcal F}
\newcommand{\Ccyl}{C_{\mathrm{cyl}}}
\newcommand{\HWP}{H_{\mathrm{WP}}}
\newcommand{\WP}{\mathrm{WP}}
\newcommand{\VK}{\mathrm{VK}}
\newcommand{\CapOne}{\operatorname{Cap}_1}

\newcommand{\Dom}{\operatorname{Dom}}

\newcommand{\id}{\operatorname{id}}
\newcommand{\prX}{\operatorname{pr}_X}
\allowdisplaybreaks[2]
\title{The canonical SLE-reversible\\Weil--Petersson response diffusion}
\author{Chunhao Cai\\
{\small School of Mathematics (Zhuhai), Sun Yat-Sen University}\\
{\small \texttt{caichh9@mail.sysu.edu.cn}}}
\date{}
\begin{document}
\maketitle
\begin{abstract}
The rough response calculus for the rotationally marked SLE welding
law gives a closed Weil--Petersson energy but does not by itself give a
process on rough weldings.  For $0<\kappa\leq4$, we prove that this form
is quasi-regular in the natural topology of circle homeomorphisms and
therefore generates a canonical diffusion reversible with respect to
the SLE welding probability.  The proof constructs a compact form nest
from inverse arc masses.  We identify the generator as $-D^*D$ and prove
absolute $L^2$ convergence of its Fourier expansion on the response
cylinder core, together with the Fukushima decomposition.  A common
Fourier noise and planar estimates preventing boundary collapse yield
a second quasi-regular realization by normalized Jordan welding
triples.  The diffusion is irreducible and converges strongly in $L^2$
to equilibrium.  For $0<\kappa<4$, the complement of the conformally
removable welding locus has zero capacity.
\end{abstract}



\section{Introduction}
\label{sec:introduction}

We seek a continuous evolution of complete SLE weldings that is
reversible with respect to their law and governed by Weil--Petersson
response directions. In \cite{Response}, the finite-mode integration-by-parts
identity is established directly in the actual right-composition convention,
and the resulting Weil--Petersson response differential is closed. This yields a
symmetric Markovian form on the rough welding probability space, but not yet a
process on the welding space. Appendix~A of \cite{Response} compares this
convention with the analytic Lie derivatives of Baverez--Jego \cite{BJwelding};
related loop-space integration-by-parts and Virasoro structures appear in
\cite{BJloops}.

The missing step is quasi-regularity on the intended state space.  A
closed Dirichlet form on an abstract $L^2$-space need not produce a
sample-continuous process on the configuration space one wants to
study.  For SLE weldings, compactness in the natural topology requires
simultaneous control of a circle homeomorphism and its inverse, whereas
the response energy controls only Weil--Petersson directions.  We
therefore begin by constructing a compact form nest in the welding
space itself.

Canonical Brownian motions associated with the $H^{3/2}$ Fourier
geometry of circle homeomorphisms, and stochastic differential
equations obtained by adding a prescribed drift, were studied in
\cite{Fang02,Fang04}.  Their transport through Beltrami equations and
stochastic conformal welding to univalent functions and Jordan curves
was developed in \cite{AMT04,AMT}.  Related static SLE inputs include
quasi-invariance of welding laws \cite{FS} and infinitesimal conformal
restriction and Virasoro structures for SLE loop measures \cite{GQW}.
Here the prescribed reversible measure is the SLE welding probability
itself: its closed response differential supplies the score drift, and
the resulting symmetric form is realized as a diffusion on the rough
welding space.

\subsection{From rough response to a diffusion}

Fix $0<\kappa\leq4$. Let $S^1$ be the unit circle and
$\T=\R/(2\pi\mathbb Z)$ the angle circle, identified with $S^1$ by
$\theta\mapsto e^{i\theta}$. Let $X$ be the space of orientation-preserving
homeomorphisms of $S^1$, equipped with
\[
 d(h_1,h_2)
 =\sup_{z\in S^1}|h_1(z)-h_2(z)|
  +\sup_{z\in S^1}|h_1^{-1}(z)-h_2^{-1}(z)|.
\]
Thus convergence in $X$ means uniform convergence of both the welding
and its inverse.  Let $\nu$ be the rotationally marked SLE welding
probability with parameter $\kappa$ on $X$, and let $\HWP$ be the real
Weil--Petersson response space. The $L^2$ spaces used for the forms are real;
complex-valued observables are treated by complexification.
The construction in \cite{Response} gives a closed gradient
\begin{equation}\label{eq:input-form}
 D:\Dcal\subset L^2(\nu)\longrightarrow L^2(\nu;\HWP),
 \qquad
 \E(F,G)=\int_X(DF,DG)_{\WP}\,d\nu,
 \qquad \Dcal=\Dom(D).
\end{equation}
For $v\in\HWP$ write $D_vF=(DF,v)_{\WP}$.  The form
$(\E,\Dcal)$ is closed, symmetric and Markovian.  We use the standard
Dirichlet-form terminology of \cite{FOT}.  Write $\CapOne$ for the
$1$-capacity of $(\E,\Dcal)$; quasi-everywhere means outside a set of
zero $\CapOne$-capacity.  We write $\mathbf 1$ for the constant-one
function and $\mathbf 1_A$ for the indicator of a set $A$.  The capacity
and nest conventions used below are recalled in Section~\ref{sec:nest}.

\begin{theorem}\label{thm:main}
There is a compact $\E$-nest $(K_m)$ in $X$ such that
\begin{equation}\label{eq:intro-compact-nest}
 \CapOne(X\setminus K_m)\longrightarrow0.
\end{equation}
Hence $(\E,\Dcal)$ is quasi-regular and determines a $\nu$-symmetric
diffusion $(X_t)_{t\geq0}$ with continuous $X$-valued paths for
quasi-every initial state. Its semigroup $(T_t)_{t\geq0}$ satisfies,
for $F,G\in L^2(\nu)$ and $t\geq0$,
\begin{equation}\label{eq:intro-stationary-reversible}
 T_t\mathbf 1=\mathbf 1,
 \qquad
 \langle T_tF,G\rangle_{L^2(\nu)}
 =\langle F,T_tG\rangle_{L^2(\nu)}.
\end{equation}
\end{theorem}

\subsection{Generator and stochastic calculus}

Let $L$ be the nonpositive self-adjoint generator of $(T_t)_{t\geq0}$
on $L^2(\nu)$, called the \emph{response generator}. Equivalently,
\begin{equation}\label{eq:gen-definition}
 \E(F,G)=-\langle LF,G\rangle_{L^2(\nu)},
 \qquad F\in\Dom(L),\quad G\in\Dcal.
\end{equation}
Let $D^*$ be the Hilbert-space adjoint of $D$. We identify $v\in\HWP$
with the constant function $h\mapsto v$ in $L^2(\nu;\HWP)$.
For the explicit formula, use the normalized real Fourier basis of $\HWP$,
\begin{align}
 e_{n,c}(\theta)&=\alpha_n\cos(n\theta),\qquad
 e_{n,s}(\theta)=\alpha_n\sin(n\theta),\qquad
 \alpha_n=\sqrt{\frac2{n^3-n}},\quad n\geq2,
 \label{eq:onb}\\
 \shortintertext{and write}
 D_{n,\tau}&=D_{e_{n,\tau}},\qquad
 \rho_{n,\tau}=D^*e_{n,\tau},\qquad \tau\in\{c,s\}.\notag
\end{align}
By \cite[Theorem~1.2]{Response}, $e_{n,\tau}\in\Dom(D^*)$.  Let $\Ccyl$ denote the response cylinder
core, defined in Section~\ref{sec:nest}.

\begin{theorem}\label{thm:generator-main}
For the diffusion of Theorem~\ref{thm:main},
\begin{equation}\label{eq:gen-DstarD}
 L=-D^*D,
 \qquad
 \Dom(L)=\{F\in\Dcal:DF\in\Dom(D^*)\}.
\end{equation}
For every $F\in\Ccyl$,
\begin{align}
 LF&=\sum_{n=2}^{\infty}\sum_{\tau\in\{c,s\}}
        \bigl(D_{n,\tau}^2F-\rho_{n,\tau}D_{n,\tau}F\bigr),
 \label{eq:gen-infinite-Fourier}\\
 \shortintertext{with}
 \sum_{n,\tau}\|D_{n,\tau}^2F\|_\infty&<\infty,
 \qquad
 \sum_{n,\tau}\|\rho_{n,\tau}D_{n,\tau}F\|_{L^2(\nu)}<\infty.
 \label{eq:main-generator-summability}
\end{align}
Hence \eqref{eq:gen-infinite-Fourier} converges absolutely in $L^2(\nu)$,
the second-order part converges uniformly absolutely, and $LF\in
L^\infty(\nu)$.
\end{theorem}

Writing $L_{\rm cyl}:=L|_{\Ccyl}$, the nonnegative operator $-L$ is the
Friedrichs extension of $-L_{\rm cyl}$.

Put $\Gamma(F,G)=(DF,DG)_{\WP}$ and $\Gamma(F)=\Gamma(F,F)$. The same
closed gradient also determines the martingale covariance of the
sample-path dynamics. Domain functions evaluated along paths use
quasi-continuous representatives.

\begin{corollary}\label{cor:fukushima-main}
For every $F\in\Dom(L)$ there is a continuous martingale additive
functional $M^F$ of finite energy such that, for quasi-every initial
state,
\begin{equation}\label{eq:main-Fukushima}
 F(X_t)-F(X_0)
  =M_t^F+\int_0^t LF(X_s)\,ds,
 \qquad
 \langle M^F\rangle_t=2\int_0^t\Gamma(F)(X_s)\,ds,
\end{equation}
for all $t\geq0$, almost surely.
\end{corollary}

\subsection{Geometric realization}

The process $X_t$ records the boundary welding, but its continuity
does not by itself give a continuous planar curve. A common Fourier
noise recovered from the boundary martingales lets us control smooth
welding deformations and pass to their rough limit. To retain the
Jordan curve itself, let $Y$ denote the normalized welding triples
$(h,f,g)$ for which $h\in X$, the maps $f$ and $g$ uniformize the two
complementary components,
\[
 f=g\circ h\quad\text{on }S^1,
 \qquad
 g(z)=z+O(z^{-1})\quad(z\to\infty).
\]
Section~\ref{sec:welding-support} gives the precise function-space
realization and topology of $Y$.  Choose the measurable normalized
welding lift $\iota$ supplied by \cite{Response}, extending it
measurably across its $\nu$-null exceptional set, and put
\begin{equation}\label{eq:intro-Jordan-transport}
 \begin{gathered}
 \prX(h,f,g)=h,
 \qquad UF=F\circ\prX,
 \qquad \nu^Y=\iota_*\nu,\\
 \Dcal^Y=U\Dcal,
 \qquad
 \E^Y(UF,UG)=\E(F,G).
 \end{gathered}
\end{equation}

\begin{theorem}\label{thm:full-main}
The response diffusion has a continuous Jordan lift
$(Y_t)_{t\geq0}$ associated with the transported form
$(\E^Y,\Dcal^Y)$.  Outside a properly exceptional set in $Y$,
\begin{equation}\label{eq:intro-Jordan-projection}
 Y_\cdot\in C([0,\infty),Y),
 \qquad
 \bigl(\prX(Y_t)\bigr)_{t\geq0}
 \stackrel{\mathrm{law}}{=}
 (X_t)_{t\geq0}
 \quad\text{started from }\prX(Y_0).
\end{equation}
If $(T_t^Y)_{t\geq0}$ denotes the semigroup of $Y_t$, then for $t\geq0$
and $F,G\in L^2(\nu^Y)$,
\begin{equation}\label{eq:intro-Jordan-reversibility}
 T_t^Y\mathbf 1=\mathbf 1,
 \qquad
 \langle T_t^YF,G\rangle_{L^2(\nu^Y)}
 =\langle F,T_t^YG\rangle_{L^2(\nu^Y)}.
\end{equation}
\end{theorem}

\subsection{Removability}

Let $W_{\rm rem}\subset X$ be the weldings of conformally removable
Jordan curves.  For the static SLE welding law,
$\nu(W_{\rm rem})=1$ for $0<\kappa\leq4$
\cite{RS,JS,KMS}.  The pathwise question is stronger: is the complement
exceptional for the Dirichlet form?  In the subcritical range it is.

\begin{theorem}\label{thm:removability-main}
If $0<\kappa<4$, then
\begin{equation}\label{eq:subcritical-removable-capacity}
 \CapOne(X\setminus W_{\rm rem})=0.
\end{equation}
\end{theorem}

At $\kappa=4$ the Jordan realization of Theorem~\ref{thm:full-main}
still holds, but the capacity conclusion of
Theorem~\ref{thm:removability-main} is not proved here.

\subsection{Irreducibility and equilibrium}

Returning to the response form, we ask whether the diffusion has
distinct invariant components. The argument below applies throughout
$0<\kappa\leq4$, independently of the preceding subcritical
removability result.

\begin{theorem}\label{thm:ergodicity}
The zero-energy space and the kernel of the generator satisfy
\begin{align}
 \{F\in\Dcal:DF=0\}
 &=\ker L
 =\{F\in L^2(\nu):F=c\ \nu\text{-a.e. for some }c\in\mathbb R\},
 \label{eq:intro-zero-energy-kernel}\\
 \shortintertext{and, for every $F\in L^2(\nu)$,}
 \lim_{t\to\infty}
 \left\|T_tF-\int_X F\,d\nu\right\|_{L^2(\nu)}&=0.
 \label{eq:ergodic-L2}
\end{align}
Hence the response form is irreducible and the stationary diffusion is
ergodic.
\end{theorem}

The rest of the paper is organized as follows.
Section~\ref{sec:nest} constructs the compact nest and proves the
quasi-regular realization. Section~\ref{sec:generator} identifies the
generator and establishes the Fukushima decomposition.
Section~\ref{sec:boundary-dynamics} derives the common-noise boundary
dynamics. Section~\ref{sec:welding-support} establishes the planar
characteristics and two-point estimates, and
Section~\ref{sec:Jordan-realization} constructs the continuous Jordan
realization. Section~\ref{sec:removable-support} proves the subcritical
capacity statement. Section~\ref{sec:ergodicity} proves irreducibility
and convergence to equilibrium.

\section{The compact nest}
\label{sec:nest}

This section proves Theorem~\ref{thm:main}.  We first identify a countable
continuous response core, then control inverse arc masses, construct two
vanishing-energy cutoffs, and assemble them into a compact form nest.

A finite real response field satisfies
\begin{align}
 v(\theta)&=\sum_{|n|\geq2}v_ne^{in\theta},
 \qquad v_{-n}=\overline{v_n},\notag\\
 \shortintertext{and}
 \|v\|_{\WP}^2&=2\sum_{n\geq2}(n^3-n)|v_n|^2.
 \label{eq:wp-norm}
\end{align}
For scalar or $\HWP$-valued functions, $\|\cdot\|_2$ denotes the
corresponding $L^2(\nu)$ norm.  For the capacity estimates below, set
\[
 \E_1(F,G)=\E(F,G)+\int_XFG\,d\nu,
 \qquad \|F\|_{\E_1}^2=\E_1(F,F).
\]
For open $O\subset X$ define
\[
 \CapOne(O)=\inf\{\|F\|_{\E_1}^2:F\in\Dcal,
                         \ F\geq1\ \nu\text{-a.e. on }O\},
\]
and use the outer extension for arbitrary sets.  An increasing
sequence of closed sets $(K_m)$ is an $\E$-nest when
\[
 \overline{\bigcup_m\{F\in\Dcal:
            F=0\ \nu\text{-a.e. on }X\setminus K_m\}}
 ^{\|\cdot\|_{\E_1}}=\Dcal.
\]
A statement holds quasi-everywhere if its exceptional set has zero
$\CapOne$-capacity.

\subsection{The response rules and separating coordinates}

For $a\in C^\infty(\T;\R)$ and $b\in C^\infty(S^1;\R)$, set
\begin{equation}\label{eq:J}
 J_{a,b}(h)=\frac1{2\pi}\int_0^{2\pi}
                   a(\theta)b(h(e^{i\theta}))\,d\theta.
\end{equation}
Let $\Ccyl$ be the algebra of functions
$\Phi(J_{a_1,b_1},\ldots,J_{a_m,b_m})$ with
$\Phi\in C_b^\infty(\mathbb R^m)$.  If $\psi_t^v$ is the circle flow
generated by a finite field $v$, then
\begin{equation}\label{eq:actual-right-derivative}
 \partial_v^{\rm act}F(h)
 =\left.\frac{d}{dt}\right|_{t=0}F(h\circ\psi_t^v),
 \qquad F\in\Ccyl,
\end{equation}
and \cite[Theorem~1.1]{Response} gives
$D_vF=\partial_v^{\rm act}F$ on $\Ccyl$.

The actual-right response formula \cite[(3.4)]{Response} is
\begin{equation}\label{eq:right-response}
 D_vJ_{a,b}=-J_{(av)',b}.
\end{equation}
For smooth $\phi$ with bounded derivative and a normal contraction
$\vartheta$ satisfying $\vartheta(0)=0$ and
$|\vartheta(s)-\vartheta(t)|\leq|s-t|$, the closed chain and
contraction rules \cite[(5.18), (5.26)]{Response} give
\begin{equation}\label{eq:closed-calculus-rules}
 D[\phi(F)]=\phi'(F)\,DF,\qquad
 \|\vartheta(F)\|_{\E_1}\leq\|F\|_{\E_1}
 \quad(F\in\Dcal).
\end{equation}
For bounded $F,G\in\Dcal$, choose a smooth function $\chi_{\rm sq}$ with
bounded derivative agreeing with $t^2$ on the ranges of $F+G$ and
$F-G$. The same chain rule then gives the product rule:
\begin{align*}
 FG&=\tfrac14\{\chi_{\rm sq}(F+G)-\chi_{\rm sq}(F-G)\}\in\Dcal,\\
 \shortintertext{and}
 D(FG)&=\tfrac12\{(F+G)(DF+DG)-(F-G)(DF-DG)\}
       =F\,DG+G\,DF.
\end{align*}

\begin{lemma}\label{lem:state-space-core}
The metric $d$ is complete and $X$ contains a countable dense subset.  Moreover,
\begin{align}
 \Ccyl&\subset\Dcal\cap C_b(X),
 \qquad
 \overline{\Ccyl}^{\|\cdot\|_{\E_1}}=\Dcal,
 \label{eq:continuous-core}\\
 \shortintertext{and there exists a sequence $(F_j)_{j\geq1}\subset\Ccyl$ such that}
 \overline{\{F_j:j\geq1\}}^{\|\cdot\|_{\E_1}}
 &=\Dcal,
 \qquad
 h\ne g\Longrightarrow F_j(h)\ne F_j(g)\text{ for some }j.
 \label{eq:countable-separation}
\end{align}
\end{lemma}
\begin{proof}
Since $S^1$ is compact metric, the uniform metric makes
$C(S^1,S^1)$ complete and separable.  Embed $X$ into
$C(S^1,S^1)^2$ by $h\mapsto(h,h^{-1})$.  If
$(h_j,q_j)\to(h,q)$ uniformly and
$h_j\circ q_j=q_j\circ h_j=\id$, then, with $\omega_f$ denoting the
modulus of continuity of $f$,
\begin{align*}
 \|h\circ q-\id\|_\infty
 &\leq\omega_h(\|q-q_j\|_\infty)+\|h-h_j\|_\infty
 \longrightarrow0,\\
 \shortintertext{and symmetrically}
 \|q\circ h-\id\|_\infty&\longrightarrow0.
\end{align*}
For all sufficiently large $j$,
$\sup_{z\in S^1}|h_j(z)-h(z)|<2$, so
\begin{align*}
 H_j(s,z)&=
 \frac{(1-s)h_j(z)+s h(z)}{|(1-s)h_j(z)+s h(z)|},
 \qquad 0\leq s\leq1,\\
 \shortintertext{is a homotopy from $h_j$ to $h$, and hence}
 \deg h&=\deg h_j=1.
\end{align*}
Thus the image of $X$ is closed in the complete separable product
$C(S^1,S^1)^2$.  The induced metric $d$ is complete.  Since a separable
metric space has a countable base and subspaces inherit such a base,
$X$ contains a countable dense subset.

For the averaged coordinates,
\begin{align*}
 |J_{a,b}(h)|&\leq\|a\|_\infty\|b\|_\infty,\\
 \shortintertext{and}
 |J_{a,b}(h)-J_{a,b}(g)|
 &\leq\|a\|_\infty\|b\circ h-b\circ g\|_\infty\\
 &\leq\|a\|_\infty\,
       \omega_b(\|h-g\|_\infty)
 \longrightarrow0.
\end{align*}
Hence $J_{a,b}\in C_b(X)$.  Set
\[
 b_1(z)=\operatorname{Re}z,
 \qquad b_2(z)=\operatorname{Im}z,
 \qquad
 a\in\{1,\cos n\theta,\sin n\theta:n\geq1\}.
\]
If all corresponding coordinates agree for $h$ and $g$, then
\begin{align*}
 \int_0^{2\pi}e^{in\theta}
 (b_\ell\circ h-b_\ell\circ g)(e^{i\theta})\,d\theta&=0
 \qquad(n\in\mathbb Z,\ \ell\in\{1,2\}),\\
 \shortintertext{hence, by Fourier completeness and continuity,}
 b_\ell\circ h&=b_\ell\circ g\quad(\ell=1,2),
 \qquad h=g.
\end{align*}
These coordinates form a countable point-separating family.  Each lies
in $\Ccyl$: for
$A=\|a\|_\infty\|b\|_\infty$, choose
$\chi_A\in C_b^\infty(\mathbb R)$ with
$\chi_A(t)=t$ on $[-A,A]$.

By construction of the closed response gradient in \cite[Theorem~1.1]{Response},
$D$ is the closure of its restriction to $\Ccyl$.  Therefore
\begin{equation}\label{eq:core-graph-norm}
 \|F-G\|_{\E_1}^2
 =\|F-G\|_2^2+\|DF-DG\|_2^2,
 \qquad
 \overline{\Ccyl}^{\|\cdot\|_{\E_1}}=\Dcal.
\end{equation}
Because the Borel $\sigma$-algebra of the separable metric space $X$
is countably generated and $\nu$ is finite, $L^2(\nu)$ is separable;
the completion of $\nu$ does not change this $L^2$ space.  The Fourier
basis \eqref{eq:onb} makes $\HWP$ separable, hence
$L^2(\nu;\HWP)$ is separable.  Thus
\[
 F\longmapsto(F,DF)
\]
embeds $(\Dcal,\|\cdot\|_{\E_1})$ isometrically into the separable
space $L^2(\nu)\oplus L^2(\nu;\HWP)$.  Choose a countable
$\|\cdot\|_{\E_1}$-dense subset of $\Ccyl$, adjoin the countable
separating coordinates above, and enumerate the union as $(F_j)$.  This gives \eqref{eq:countable-separation}.
\end{proof}

Enlarging the sequence $(F_j)$ within $\Ccyl$ if necessary, we henceforth
assume that it contains $\mathbf 1$ and enumerates the countable algebra over
$\mathbb Q$ generated by the preceding dense separating family. It
remains form dense and point separating. Since $X$ is a complete
separable metric space, the sigma-algebra generated by $(F_j)$ is its
Borel sigma-algebra; the functional monotone-class theorem therefore
shows that $(F_j)$ determines Borel probability measures on $X$.

\subsection{Inverse arc masses and the critical kernel}

Inverse arc masses control both equicontinuity and strict increase of
subsequential limits.

For a proper positively oriented arc $I=(e^{i\alpha},e^{i\beta})$, define
\begin{equation}\label{eq:arc-mass}
 p_I(h)=\frac1{2\pi}\operatorname{Leb}(h^{-1}(I)).
\end{equation}
Choose compatible inverse lifts $x=x(h)=h^{-1}(\alpha)$ and
$y=y(h)=h^{-1}(\beta)$ with $0<y-x<2\pi$. Then
$p_I=(y-x)/(2\pi)$.  In particular $0<p_I<1$.  If $h_j\to h$ in
$d$, then $h_j^{-1}\to h^{-1}$ uniformly, so the two inverse endpoint
images converge and
\[
 p_I(h_j)\longrightarrow p_I(h).
\]

\begin{lemma}\label{lem:arc-gradient}
For every proper oriented arc $I$,
\[
 p_I\in\Dcal\cap C_b(X).
\]
There is a continuous map $G_I:X\to\HWP$ such that
$Dp_I=G_I$ in $L^2(\nu;\HWP)$ and, for every finite response field $v$,
\begin{align}
 (G_I(h),v)_{\WP}&=\frac{v(x)-v(y)}{2\pi},
 \label{eq:arc-derivative}\\
 \shortintertext{and}
 \Gamma(p_I)(h)&=\frac1{\pi^2}
 \sum_{n\geq2}\frac{1-\cos(n(y-x))}{n^3-n}
 \quad\text{for $\nu$-a.e. }h.
 \label{eq:arc-gamma}
\end{align}
\end{lemma}
\begin{proof}
Choose smooth functions $0\leq b_\varepsilon\leq1$ on $S^1$ such
that $b_\varepsilon\to\mathbf 1_I$ away from the two endpoints of
$I$, and put $p_\varepsilon=J_{1,b_\varepsilon}$.  The preimages under $h$ of the two endpoints of $I$ are singletons and therefore have zero Lebesgue measure.  Hence
\begin{align*}
 b_\varepsilon\circ h&\longrightarrow\mathbf 1_I\circ h
 \quad\text{in }L^2(\T,d\theta/(2\pi))
 \qquad(h\in X),\\
 \shortintertext{while, for every finite response field $v$,}
 \|v'\|_{L^2(d\theta/(2\pi))}^2
 &=2\sum_{n\geq2}n^2|v_n|^2
 \leq\frac23\|v\|_{\WP}^2.
\end{align*}
Using \eqref{eq:right-response} and duality in $\HWP$ therefore gives,
for every $h\in X$,
\begin{equation}\label{eq:arc-gradient-Cauchy}
 \|Dp_\varepsilon(h)-Dp_\delta(h)\|_{\WP}
 \leq\sqrt{\frac23}
 \|(b_\varepsilon-b_\delta)\circ h
       \|_{L^2(d\theta/(2\pi))}.
\end{equation}
Bounded convergence first on $\T$ and then on $(X,\nu)$ yields
\begin{align*}
 p_\varepsilon&\longrightarrow p_I\quad\text{in }L^2(\nu),\\
 \shortintertext{and, for some $V_I$,}
 Dp_\varepsilon&\longrightarrow V_I
 \quad\text{in }L^2(\nu;\HWP).
\end{align*}  Closedness of $D$ gives
$p_I\in\Dcal$ and $Dp_I=V_I$.

Let $e_{n,\tau}$ be the basis from \eqref{eq:onb}.  Define
\begin{equation}\label{eq:arc-gradient-field}
 G_I(h)=\sum_{n\geq2}\sum_{\tau\in\{c,s\}}
 \frac{e_{n,\tau}(x)-e_{n,\tau}(y)}{2\pi}
 e_{n,\tau}.
\end{equation}
The summands depend continuously on $h$: they are periodic evaluations
at the two inverse endpoint images, so the choice of angular lifts is
irrelevant.  Moreover,
\[
 \sup_{h\in X}
 \sum_{n>N}\sum_{\tau\in\{c,s\}}
 \left|\frac{e_{n,\tau}(x)-e_{n,\tau}(y)}{2\pi}\right|^2
 \leq\frac1{\pi^2}\sum_{n>N}\alpha_n^2
 \longrightarrow0.
\]
Hence the series in \eqref{eq:arc-gradient-field} converges uniformly
in $h$ with respect to the $\HWP$ norm, and $G_I:X\to\HWP$ is
continuous.

For every finite response field $v$, integration by parts in the
angular variable gives
\begin{align*}
 (Dp_I,v)_{\WP}
 &=\lim_{\varepsilon\downarrow0}D_vp_\varepsilon\\
 &=-\frac1{2\pi}\int_x^y v'(\theta)\,d\theta
 =\frac{v(x)-v(y)}{2\pi}
 =(G_I,v)_{\WP}
 \quad\text{in }L^2(\nu).
\end{align*}
Applying this identity to the countable basis gives
$Dp_I=G_I$ in $L^2(\nu;\HWP)$.  Finally,
\[
 \sum_{\tau\in\{c,s\}}
 |e_{n,\tau}(x)-e_{n,\tau}(y)|^2
 =2\alpha_n^2\{1-\cos(n(y-x))\}.
\]
Parseval's identity and
$\alpha_n^2=2/(n^3-n)$ give \eqref{eq:arc-gamma}.
\end{proof}

\begin{lemma}\label{lem:critical-kernel}
For a universal constant $C$,
\begin{equation}\label{eq:critical-bound}
 \Gamma(p_I)\leq C p_I^2\log(e/p_I)\qquad\nu\text{-almost everywhere}.
\end{equation}
\end{lemma}

\begin{proof}
Put
\[
 r=\min(y-x,2\pi-y+x),
 \qquad q=\min(p_I,1-p_I)=\frac r{2\pi}.
\]
For $0<r\leq1$, the bound
$1-\cos u\leq\min(2,u^2/2)$ gives
\begin{align*}
 \sum_{n\geq2}\frac{1-\cos(nr)}{n^3-n}
 &\leq C r^2\sum_{2\leq n\leq r^{-1}}\frac1n
           +C\sum_{n>r^{-1}}\frac1{n^3}\\
 &\leq C r^2\log(e/r).
\end{align*}
For $1\leq r\leq\pi$ the series is bounded, while
$q\geq(2\pi)^{-1}$. Hence in both cases
\[
 \Gamma(p_I)\leq Cq^2\log(e/q).
\]
If $p_I\leq1/2$, then $q=p_I$. If $p_I>1/2$, then
\[
 \Gamma(p_I)\leq C
 \leq4C p_I^2\log(e/p_I),
\]
which proves \eqref{eq:critical-bound}.
\end{proof}

\subsection{Two vanishing-energy cutoffs}

We use \eqref{eq:critical-bound} for an upper-cell cutoff and for a
lower-mass cutoff on each fixed target arc.

Partition the target circle into $2^N$ dyadic cells $I_{N,j}$, with
$N\geq1$. Endpoints do not affect their masses. Write
\[
 p_{N,j}=p_{I_{N,j}},\qquad
 S_N=\sum_{j=0}^{2^N-1}p_{N,j}^2,\qquad
 M_N=\max_jp_{N,j}.
\]
\begin{lemma}\label{lem:upper-cutoff}
For every $N\geq1$, $S_N\in\Dcal\cap C_b(X)$ and $M_N^2\leq S_N$.
Moreover,
\begin{equation}\label{eq:SN-zero}
 \|S_N\|_{\E_1}\longrightarrow0\qquad(N\to\infty).
\end{equation}
\end{lemma}
\begin{proof}
Lemma~\ref{lem:arc-gradient} and the bounded product rule give
$S_N\in\Dcal\cap C_b(X)$.  For every $N$,
\begin{align*}
 \sum_jp_{N,j}&=1,
 \qquad M_N^2\leq S_N\leq M_N,\\
 \shortintertext{and}
 DS_N&=2\sum_jp_{N,j}Dp_{N,j}.
\end{align*}
If $q_h$ is any increasing lift of $h^{-1}$, uniform continuity gives
\begin{align*}
 M_N(h)&\leq\frac1{2\pi}
 \sup_{|s-t|\leq2\pi2^{-N}}|q_h(s)-q_h(t)|
 \longrightarrow0,\\
 \shortintertext{and hence}
 0\leq M_N^2&\leq S_N\leq M_N\longrightarrow0
 \quad\text{pointwise on }X.
\end{align*}
The product rule and Cauchy--Schwarz yield
\begin{align*}
 \Gamma(S_N)&=4\Big\|\sum_jp_{N,j}Dp_{N,j}\Big\|_{\WP}^2\\
 &\leq4S_N\sum_j\Gamma(p_{N,j})\\
 &\leq4CS_N\sum_jp_{N,j}^2\log(e/p_{N,j})\\
 &\leq4CS_NM_N\log(e/M_N)\sum_jp_{N,j}\\
 &\leq4CM_N^2\log(e/M_N)\longrightarrow0.
\end{align*}
Here $u\mapsto u\log(e/u)$ is increasing on $(0,1]$.  We have
\begin{align*}
 S_N^2&\longrightarrow0,
 \qquad
 \Gamma(S_N)\longrightarrow0
 \quad\text{pointwise on }X,\\
 \shortintertext{while, uniformly in $N$,}
 0\leq S_N^2&\leq1,
 \qquad
 0\leq\Gamma(S_N)
 \leq4C\sup_{0<u\leq1}u^2\log(e/u)<\infty.
\end{align*}
Dominated convergence therefore gives
\[
 \|S_N\|_2^2+\E(S_N,S_N)\longrightarrow0,
\]
which is \eqref{eq:SN-zero}.
\end{proof}

Fix a proper arc $I$ and a nondecreasing
$\eta\in C^\infty([0,\infty);[0,1])$ such that
$\eta=0$ on $[0,1/2]$ and $\eta=1$ on $[1,\infty)$. For $R\geq2$ set
\begin{equation}\label{eq:lower-cutoff}
 \chi_R(p)=
 \begin{cases}
  \eta\!\left(R^{-1}\sqrt{\log(e/p)}\right),&0<p\leq1,\\
  1,&p=0,
 \end{cases}
 \qquad
 g_{I,R}=\chi_R(p_I),
 \qquad \delta_R=\exp(1-R^2).
\end{equation}

\begin{lemma}\label{lem:lower-cutoff}
For every $R\geq2$,
\begin{align}
 &g_{I,R}\in\Dcal\cap C_b(X),\qquad
 0\leq g_{I,R}\leq1,\qquad
 g_{I,R}=1\ \text{on }\{p_I\leq\delta_R\},\notag\\
 \shortintertext{and}
 &\Gamma(g_{I,R})\leq\frac{C\|\eta'\|_\infty^2}{4R^2},
 \qquad \|g_{I,R}\|_{\E_1}\longrightarrow0\quad(R\to\infty).
 \label{eq:g-zero}
\end{align}
\end{lemma}
\begin{proof}
The choice of $\eta$ makes $\chi_R$ constant near $p=0$.
At $p=1$, it vanishes in a neighborhood when $R>2$; when $R=2$,
all one-sided derivatives vanish because $\eta$ is flat at $1/2$.
Thus $\chi_R\in C^\infty([0,1])$ and extends to a function in
$C_b^\infty(\mathbb R)$. Moreover,
\[
 \chi_R'(p)
 =-\frac{\eta'\!\left(R^{-1}\sqrt{\log(e/p)}\right)}
        {2Rp\sqrt{\log(e/p)}}
 \qquad(0<p\leq1).
\]
The chain rule gives $g_{I,R}\in\Dcal\cap C_b(X)$.  Since
$0\leq\eta\leq1$, also $0\leq g_{I,R}\leq1$; and if
$p_I\leq\delta_R$, then $\log(e/p_I)\geq R^2$, so
$g_{I,R}=1$.  Moreover, the chain rule and \eqref{eq:critical-bound}
give
\[
 \Gamma(g_{I,R})
 =\frac{\left|\eta'\!\left(R^{-1}\sqrt{\log(e/p_I)}\right)\right|^2}
        {4R^2p_I^2\log(e/p_I)}\Gamma(p_I)
 \leq\frac{C\|\eta'\|_\infty^2}{4R^2}.
\]
Since $p_I(h)>0$ for every $h\in X$, for all sufficiently large $R$,
\[
 R^{-1}\sqrt{\log(e/p_I(h))}\leq\frac12,
 \qquad g_{I,R}(h)=0.
\]
Hence $g_{I,R}(h)\to0$ pointwise.
Bounded convergence and the preceding energy bound now give
\begin{align*}
 \|g_{I,R}\|_{\E_1}^2
 &=\int g_{I,R}^2\,d\nu+\int\Gamma(g_{I,R})\,d\nu\\
 &\leq\int g_{I,R}^2\,d\nu
       +\frac{C\|\eta'\|_\infty^2}{4R^2}
 \longrightarrow0,
\end{align*}
which proves \eqref{eq:g-zero}.
\end{proof}

\subsection{Assembly of the compact nest}

We now choose the two cutoff scales so that their form norms are summable.
Enumerate all proper dyadic arcs as $(A_k)_{k\geq1}$.  For fixed
$m,k\geq1$, \eqref{eq:SN-zero} and the contraction estimate give
\[
 \|1\wedge(2^{2k}S_N)\|_{\E_1}
 \leq2^{2k}\|S_N\|_{\E_1}\longrightarrow0.
\]
Hence choose $N(m,k)\geq k$ such that
\begin{equation}\label{eq:upper-scale-choice}
 u_{m,k}=1\wedge\bigl(2^{2k}S_{N(m,k)}\bigr),
 \qquad
 \|u_{m,k}\|_{\E_1}\leq2^{-m-k-2}.
\end{equation}
By \eqref{eq:g-zero}, choose $R(m,k)\geq2$ such that
\begin{equation}\label{eq:lower-scale-choice}
 \|g_{A_k,R(m,k)}\|_{\E_1}\leq2^{-m-k-2},
 \qquad
 \delta_{m,k}=\exp(1-R(m,k)^2).
\end{equation}
Define
\begin{equation}\label{eq:compact-sets}
 K_m^0=
 \bigcap_{k\geq1}\{S_{N(m,k)}\leq2^{-2k}\}
 \cap
 \bigcap_{k\geq1}\{p_{A_k}\geq\delta_{m,k}\},
 \qquad
 K_m=\bigcup_{\ell=1}^mK_\ell^0.
\end{equation}

\begin{lemma}\label{lem:compactness}
Each set $K_m^0$ is compact; hence each $K_m$ is compact.
\end{lemma}
\begin{proof}
All functions appearing in \eqref{eq:compact-sets} are continuous, so
$K_m^0$ is closed.  For $h\in K_m^0$, let $q_h:\mathbb R\to\mathbb R$
be the increasing lift of $h^{-1}$ normalized by $q_h(0)\in[0,2\pi]$.
At level $N(m,k)$,
\[
 p_{N(m,k),j}^2\leq S_{N(m,k)}\leq2^{-2k}
 \qquad(0\leq j<2^{N(m,k)}).
\]
An interval of length at most $2\pi2^{-N(m,k)}$ meets at most two
level-$N(m,k)$ dyadic cells.  Therefore
\begin{equation}\label{eq:inverse-equicontinuity}
 0\leq q_h(\theta+u)-q_h(\theta)
 \leq4\pi2^{-k}
 \quad
 \begin{matrix}
 h\in K_m^0,\ \theta\in\mathbb R,\
 0\leq u\leq2\pi2^{-N(m,k)}.
 \end{matrix}
\end{equation}
Because $N(m,k)\geq k$, the right-hand side and the admissible mesh
both tend to zero as $k\to\infty$.  Thus the family
$\{q_h:h\in K_m^0\}$ is uniformly bounded and equicontinuous on
$[0,2\pi]$.

Given $(h_\ell)\subset K_m^0$, Arzel\`a--Ascoli yields a subsequence
and a continuous function $q$ such that
\[
 q_{h_\ell}\longrightarrow q
 \quad\text{uniformly on }[0,2\pi].
\]
The identities
$q_{h_\ell}(2\pi)=q_{h_\ell}(0)+2\pi$ pass to the limit, so extend
$q$ by $q(\theta+2\pi)=q(\theta)+2\pi$.  The function $q$ is
nondecreasing.  If $0\leq a<b\leq2\pi$, choose an enumerated dyadic
arc $A_k=(\alpha_k,\beta_k)\subset(a,b)$.  Then
\begin{align*}
 q(b)-q(a)
 &\geq q(\beta_k)-q(\alpha_k)\\
 &=\lim_{\ell\to\infty}
   \{q_{h_\ell}(\beta_k)-q_{h_\ell}(\alpha_k)\}\\
 &=\lim_{\ell\to\infty}2\pi p_{A_k}(h_\ell)
 \geq2\pi\delta_{m,k}>0.
\end{align*}
Hence $q$ is strictly increasing and induces a degree-one circle
homeomorphism $\bar q$.  Uniform convergence of the inverse lifts and
uniform continuity of $\bar q^{-1}$ give
\begin{equation}\label{eq:inverse-homeomorphism-convergence}
 d(h_\ell,\bar q^{-1})
 \leq
 \omega_{\bar q^{-1}}
   (\|\bar q_{h_\ell}-\bar q\|_\infty)
 +\|\bar q_{h_\ell}-\bar q\|_\infty
 \longrightarrow0.
\end{equation}
Continuity of the functions in \eqref{eq:compact-sets} places
$\bar q^{-1}$ in $K_m^0$.  Thus $K_m^0$ is sequentially compact and,
since $d$ is a metric, compact.  The set $K_m$ is a finite union of
such compact sets.
\end{proof}

\begin{lemma}\label{lem:nest-cutoffs}
There are functions $w_m\in\Dcal$ such that
\begin{align}
 0\leq w_m&\leq1,
 \qquad
 \|w_m\|_{\E_1}\leq2^{-m-1},\notag\\
 \shortintertext{and}
 w_m&=1\quad\nu\text{-a.e. on }X\setminus K_m^0.
 \label{eq:nest-cutoff}
\end{align}
For every $F\in\Ccyl$,
\begin{equation}\label{eq:product-cutoff}
 \|Fw_m\|_{\E_1}^2
 \leq C_F2^{-2m-2}\longrightarrow0.
\end{equation}
\end{lemma}
\begin{proof}
Let
\begin{align*}
 z_m^{(N)}&=\sum_{k=1}^{N}
 \bigl(u_{m,k}+g_{A_k,R(m,k)}\bigr),\\
 \shortintertext{so that \eqref{eq:upper-scale-choice} and
 \eqref{eq:lower-scale-choice} give}
 \sum_{k\geq1}
 \bigl(\|u_{m,k}\|_{\E_1}
       +\|g_{A_k,R(m,k)}\|_{\E_1}\bigr)
 &\leq2^{-m-1}.
\end{align*}
Hence $(z_m^{(N)})_N$ is Cauchy in the form norm; denote its limit by
$z_m\in\Dcal$.  Since $z_m^{(N)}$ is nonnegative and increasing in
$N$, the pointwise limit
\[
 z_m^{(\infty)}=\lim_{N\to\infty}z_m^{(N)}
\]
exists in $[0,\infty]$.  Moreover,
\[
 \sup_N\|z_m^{(N)}\|_2\leq2^{-m-1}.
\]
Fatou's lemma gives $z_m^{(\infty)}\in L^2(\nu)$, and
$0\leq z_m^{(\infty)}-z_m^{(N)}\leq z_m^{(\infty)}$.  Dominated
convergence therefore yields
\[
 z_m^{(N)}\longrightarrow z_m^{(\infty)}
 \quad\text{in }L^2(\nu).
\]
By uniqueness of the $L^2$ limit,
$z_m=z_m^{(\infty)}$ $\nu$-almost everywhere.

Put $w_m=1\wedge z_m$.  The contraction property gives
$\|w_m\|_{\E_1}\leq2^{-m-1}$.  Outside a $\nu$-null set,
\[
 h\notin K_m^0
 \Longrightarrow
 \begin{cases}
 S_{N(m,k)}(h)>2^{-2k},&\text{for some }k,\quad\text{or}\\
 p_{A_k}(h)<\delta_{m,k},&\text{for some }k,
 \end{cases}
 \Longrightarrow w_m(h)=1,
\]
which proves \eqref{eq:nest-cutoff}.

Let
$F=\Phi(J_{a_1,b_1},\ldots,J_{a_r,b_r})\in\Ccyl$.  The response
identity and the continuous embedding $\HWP\hookrightarrow H^1(\T)$
give
\[
 |D_vJ_{a,b}(h)|
 \leq\frac{\|b\|_\infty}{2\pi}
       \int_0^{2\pi}|(av)'(\theta)|\,d\theta
 \leq C_a\|b\|_\infty\|v\|_{\WP},
\]
whence
\begin{equation}\label{eq:cylinder-energy-bound}
 \|\Gamma(F)\|_\infty^{1/2}
 \leq\sum_{j=1}^r
 \|\partial_j\Phi\|_\infty C_{a_j}\|b_j\|_\infty<\infty.
\end{equation}
The product rule now gives
\begin{align*}
 \|Fw_m\|_{\E_1}^2
 &\leq
 \bigl(\|F\|_\infty^2+2\|\Gamma(F)\|_\infty\bigr)
       \|w_m\|_2^2
 +2\|F\|_\infty^2\E(w_m,w_m)\\
 &\leq C_F2^{-2m-2}\longrightarrow0,
\end{align*}
which is \eqref{eq:product-cutoff}.
\end{proof}

\begin{lemma}\label{lem:strong-locality}
For every $F\in\Dcal$ and $c\in\mathbb R$,
\begin{equation}\label{eq:levelset-locality}
 \mathbf 1_{\{F=c\}}DF=0
 \quad\text{in }L^2(\nu;\HWP).
\end{equation}
Consequently, if $O\subset X$ is open and $F,G\in\Dcal$ satisfy
\begin{equation}\label{eq:strong-locality-condition}
 F=c\quad\nu\text{-a.e. on }O,
 \qquad
 G=0\quad\nu\text{-a.e. on }X\setminus O,
\end{equation}
then
\begin{equation}\label{eq:strong-locality-zero}
 \E(F,G)=0.
\end{equation}
\end{lemma}
\begin{proof}
Choose $\xi\in C_c^\infty(\mathbb R)$ with $\xi=1$ on a
neighborhood of zero, and define
\[
 \varphi_\varepsilon(t)
 =\int_0^t\xi(s/\varepsilon)\,ds.
\]
Then
\begin{align*}
 |\varphi_\varepsilon(t)|&\leq\varepsilon\|\xi\|_{L^1(\mathbb R)},\\
 \shortintertext{and}
 \varphi_\varepsilon'(t)&=\xi(t/\varepsilon).
\end{align*}
The closed chain rule and dominated convergence give
\begin{align*}
 \|\varphi_\varepsilon(F-c)\|_2&\longrightarrow0,\\
 D[\varphi_\varepsilon(F-c)]
 &=\xi((F-c)/\varepsilon)DF
 \longrightarrow\mathbf 1_{\{F=c\}}DF
 \quad\text{in }L^2(\nu;\HWP).
\end{align*}
Since $D$ is closed, the limit gradient must be the gradient of the
zero function, proving \eqref{eq:levelset-locality}.

Under \eqref{eq:strong-locality-condition},
\[
 \mathbf 1_ODF=0,
 \qquad
 \mathbf 1_{X\setminus O}DG=0
 \quad\nu\text{-a.e.}
\]
by \eqref{eq:levelset-locality}.  Therefore
\[
 \E(F,G)
 =\int_O(DF,DG)_{\WP}\,d\nu
  +\int_{X\setminus O}(DF,DG)_{\WP}\,d\nu
 =0.
\]
\end{proof}

\subsection{Final proof of Theorem~\ref{thm:main}}
\label{sec:proof-ambient}

Lemma~\ref{lem:compactness} gives compact sets $K_m$.  For
$F\in\Ccyl$, let
\[
 F_m=F(1-w_m),
\]
where $w_m$ is supplied by Lemma~\ref{lem:nest-cutoffs}.  Since
$K_m^0\subset K_m$ and $w_m=1$ $\nu$-almost everywhere on
$X\setminus K_m^0$,
\begin{align*}
 F_m&=0\quad\nu\text{-a.e. on }X\setminus K_m,\\
 \shortintertext{and}
 \|F_m-F\|_{\E_1}&=\|Fw_m\|_{\E_1}\longrightarrow0.
\end{align*}
Because $\Ccyl$ is form dense by Lemma~\ref{lem:state-space-core},
\begin{equation}\label{eq:final-nest-density}
 \overline{\bigcup_m
 \{G\in\Dcal:G=0\ \nu\text{-a.e. on }X\setminus K_m\}}
 ^{\|\cdot\|_{\E_1}}
 =\Dcal.
\end{equation}
Thus $(K_m)$ is a compact $\E$-nest.

Moreover, $X\setminus K_m$ is open and
$w_m=1$ $\nu$-almost everywhere there.  Hence $w_m$ is admissible for
its $1$-capacity, and
\begin{equation}\label{eq:final-capacity-nest}
 \CapOne(X\setminus K_m)
 \leq\|w_m\|_{\E_1}^2
 \leq2^{-2m-2}\longrightarrow0.
\end{equation}
This proves \eqref{eq:intro-compact-nest}.

Lemma~\ref{lem:state-space-core} also gives a countable form-dense
sequence in $\Dcal\cap C_b(X)$ that separates points.  Together with
\eqref{eq:final-nest-density}, these are the compact-nest,
quasi-continuity, and separation conditions in the quasi-regularity
criterion of \cite[Section~2]{AMR}.  The state-space hypothesis follows
from completeness of $d$ and the existence of a countable dense subset
of $X$.  Hence $(\E,\Dcal)$ is quasi-regular.

Lemma~\ref{lem:strong-locality} gives \eqref{eq:strong-locality-zero},
which is the strong-locality condition for the form.  Since $\mathbf 1\in\Ccyl$ and $D\mathbf 1=0$,
\begin{equation}\label{eq:conservative-generator}
 \E(\mathbf 1,G)=0\quad(G\in\Dcal),
 \qquad
 \mathbf 1\in\Dom(L),
 \qquad
 L\mathbf 1=0,
 \qquad
 T_t\mathbf 1=\mathbf 1.
\end{equation}
The quasi-regular strongly local form--process correspondence
\cite[Section~2]{AMR} therefore yields a properly associated process
with continuous paths up to its lifetime $\zeta$.  For its law
$\mathbb P_h$ started from $h$, the properly associated version of the
semigroup satisfies
\[
 \mathbb P_h(\zeta>n)=T_n\mathbf 1(h)=1
 \quad\text{for quasi-every }h,
 \qquad n\in\mathbb N.
\]
Removing the countable union of these exceptional sets gives, for
quasi-every $h$,
\[
 \mathbb P_h(\zeta<\infty)
 =\lim_{n\to\infty}\mathbb P_h(\zeta\leq n)=0.
\]
Thus the paths are continuous in $X$ for all $t\geq0$.

Finally, self-adjointness of $T_t$ gives, for $F,G\in L^2(\nu)$,
\begin{equation}\label{eq:final-reversibility}
 \langle T_tF,G\rangle_{L^2(\nu)}
 =\langle F,T_tG\rangle_{L^2(\nu)}.
\end{equation}
Together with $T_t\mathbf 1=\mathbf 1$ from \eqref{eq:conservative-generator}, this is
\eqref{eq:intro-stationary-reversible}; it also gives
$\int_XT_tF\,d\nu=\int_XF\,d\nu$.  This completes the proof.

\section{The generator and its martingales}
\label{sec:generator}

We first control the Fourier response derivatives and the score drift.
These estimates identify the generator in
Section~\ref{sec:proof-generator}.  We then construct a compact regular
realization to compute the energy measures and prove
Corollary~\ref{cor:fukushima-main}.  Measurability refers to the
sigma-algebras generated by the indicated topologies.

\subsection{Response derivatives and Fourier bounds}

For the score covariance, set
\[
 c_{\mathrm L}=1+6\left(\frac{\sqrt\kappa}{2}
                      +\frac2{\sqrt\kappa}\right)^2,\qquad
 a_\kappa=\frac{c_{\mathrm L}}{12},\qquad
 (v,w)_{\VK}=\operatorname{Re}\sum_{n\geq2}n v_n\overline{w_n}.
\]
The adjoint and score covariance identities of
\cite[Theorem~1.2]{Response} are
\begin{align}
 v&\in\Dom(D^*),\qquad D^*v=\rho_v,\notag\\
 \shortintertext{with}
 \int_X\rho_v\rho_w\,d\nu
 &=a_\kappa(v,w)_{\WP}+4(v,w)_{\VK}
 \qquad(v,w\text{ finite}).
 \label{eq:input-score}
\end{align}
In particular, the basis in \eqref{eq:onb} satisfies
\begin{equation}\label{eq:gen-score-basis-norm}
 \|\rho_{n,\tau}\|_2^2
 =a_\kappa+\frac{2}{n^2-1}\leq a_\kappa+\frac23.
\end{equation}
Choose measurable representatives of these countably many scores.

The later boundary calculations require target functions with jumps.
For $a\in C^\infty(\T;\R)$ and a bounded real function $b$ of bounded
variation on $S^1$, retain the definition
\begin{align*}
 J_{a,b}(h)&=\frac1{2\pi}\int_0^{2\pi}
             a(\theta)b(h(e^{i\theta}))\,d\theta,\\
 \shortintertext{and, for the right-continuous periodic representative of $b$, write}
 b_h(\theta)&=b(h(e^{i\theta})),\qquad
 V_b=\frac{\operatorname{Var}(b)}{2\pi}.
\end{align*}
Here $db_h$ is the signed Stieltjes measure, equivalently the
distributional derivative of $b_h$.  Values at the countably many
jumps do not affect $J_{a,b}$.  All uniform bounds below hold over $h\in X$.

\begin{lemma}\label{lem:gen-BV-estimate}
For every $h\in X$ and finite response field $v$,
\begin{align}
 \partial_v^{\rm act}J_{a,b}(h)
 &=\frac1{2\pi}\int_\T av\,db_h,
 \label{eq:gen-first-Stieltjes}\\
 \shortintertext{and}
 (\partial_v^{\rm act})^2J_{a,b}(h)
 &=-\frac1{2\pi}\int_\T v(av)'\,db_h.
 \label{eq:gen-second-Stieltjes}
\end{align}
Consequently,
\begin{align}
 \|\partial_{e_{n,\tau}}^{\rm act}J_{a,b}\|_\infty
 &\leq V_b\|a\|_\infty\alpha_n,
 \label{eq:gen-first-bound}\\
 \shortintertext{and}
 \|(\partial_{e_{n,\tau}}^{\rm act})^2J_{a,b}\|_\infty
 &\leq V_b\alpha_n^2(\|a'\|_\infty+n\|a\|_\infty).
 \label{eq:gen-second-bound}
\end{align}
\end{lemma}
\begin{proof}
An increasing lift of $h$ transports cyclic partitions bijectively, so
$\operatorname{Var}(b_h)=\operatorname{Var}(b)$.  Changing variables
under the smooth circle flow gives
\begin{align*}
 J_{a,b}(h\circ\psi_t^v)
 &=\frac1{2\pi}\int_\T
       (a\circ\psi_{-t}^v)(\psi_{-t}^v)'b_h\,d\theta,\\
 \shortintertext{where the first two derivatives of the smooth factor at $t=0$, taken uniformly in $\theta$, are}
 \left.\frac{d}{dt}\right|_{t=0}
       \{(a\circ\psi_{-t}^v)(\psi_{-t}^v)'\}&=-(av)',\\
 \shortintertext{and}
 \left.\frac{d^2}{dt^2}\right|_{t=0}
       \{(a\circ\psi_{-t}^v)(\psi_{-t}^v)'\}&=((av)'v)'.
\end{align*}
Differentiating under the integral and integrating by parts on the
circle therefore yields
\begin{align*}
 \partial_v^{\rm act}J_{a,b}
 &=-J_{(av)',b}=\frac1{2\pi}\int_\T av\,db_h,\\
 (\partial_v^{\rm act})^2J_{a,b}
 &=J_{((av)'v)',b}=-\frac1{2\pi}\int_\T v(av)'\,db_h.
\end{align*}
The two integrals satisfy
\begin{align*}
 \left|\frac1{2\pi}\int_\T av\,db_h\right|
 &\leq V_b\|a\|_\infty\|v\|_\infty,\\
 \shortintertext{and}
 \left|\frac1{2\pi}\int_\T v(av)'\,db_h\right|
 &\leq V_b\|v\|_\infty
       (\|a'\|_\infty\|v\|_\infty+\|a\|_\infty\|v'\|_\infty).
\end{align*}
Substituting $\|e_{n,\tau}\|_\infty=\alpha_n$ and
$\|e_{n,\tau}'\|_\infty=n\alpha_n$ proves
\eqref{eq:gen-first-bound}--\eqref{eq:gen-second-bound}.
\end{proof}

For $m\geq1$, a BV cylinder will mean a function
\[
 F=\Phi(J_1,\ldots,J_m),\qquad
 J_r=J_{a_r,b_r},\qquad \Phi\in C_b^\infty(\mathbb R^m),
\]
with smooth $a_r$ and bounded-variation $b_r$ as above.
Smooth target functions give the core $\Ccyl$.
Write $J=(J_1,\ldots,J_m)$, $\Phi_r=\partial_r\Phi$, and
$\Phi_{rs}=\partial_r\partial_s\Phi$.  Constants $C_F$ depend only on
this cylinder data; $C_{F,\kappa}$ may also depend on $\kappa$.
Their values may increase between estimates.

\begin{lemma}\label{lem:BV-directional-calculus}
Every BV cylinder $F$ belongs to $\Dcal\cap C_b(X)$.  For every finite
response field $v$, $D_vF\in\Dcal$ and
\begin{align}
 D_vF&=-\sum_r\Phi_r(J)J_{(a_rv)',b_r},
 \label{eq:BV-first-derivative}\\
 \shortintertext{and}
 D_v^2F&=\sum_{r,s}\Phi_{rs}(J)
       J_{(a_rv)',b_r}J_{(a_sv)',b_s}
       +\sum_r\Phi_r(J)J_{((a_rv)'v)',b_r}.
 \label{eq:BV-second-derivative}
\end{align}
These closed derivatives agree with the actual-right derivatives.
\end{lemma}
\begin{proof}
If $h_\ell\to h$ in $X$, then
$b_r\circ h_\ell\to b_r\circ h$ outside the countable inverse image
of the discontinuities of $b_r$.  Bounded convergence in the angular
integral gives $J_r(h_\ell)\to J_r(h)$ and hence $F\in C_b(X)$.
For actual derivatives, the ordinary chain rule gives
\begin{align*}
 \partial_v^{\rm act}F
 &=\sum_r\Phi_r(J)\partial_v^{\rm act}J_r,\\
 \shortintertext{and}
 (\partial_v^{\rm act})^2F
 &=\sum_r\Phi_r(J)(\partial_v^{\rm act})^2J_r
   +\sum_{r,s}\Phi_{rs}(J)
          (\partial_v^{\rm act}J_r)(\partial_v^{\rm act}J_s).
\end{align*}
\begin{samepage}
Substitution of Lemma~\ref{lem:gen-BV-estimate} gives
\eqref{eq:BV-first-derivative}--\eqref{eq:BV-second-derivative} for these
actual derivatives and the bounds
\begin{align}
 \|\partial_{e_{n,\tau}}^{\rm act}F\|_\infty
 &\leq\sum_r\|\Phi_r\|_\infty V_{b_r}\|a_r\|_\infty\alpha_n
 \leq C_F\alpha_n,\notag\\
 \shortintertext{and}
 \|(\partial_{e_{n,\tau}}^{\rm act})^2F\|_\infty
 &\leq\alpha_n^2\sum_r\|\Phi_r\|_\infty V_{b_r}
                      (\|a_r'\|_\infty+n\|a_r\|_\infty)\notag\\
 &\quad+\alpha_n^2\sum_{r,s}\|\Phi_{rs}\|_\infty
             V_{b_r}V_{b_s}\|a_r\|_\infty\|a_s\|_\infty\notag\\
 &\leq C_F(1+n)\alpha_n^2.
 \label{eq:BV-proof-directional-bounds}
\end{align}
\end{samepage}

Convolve each $b_r$ with a nonnegative smooth circular approximate
identity, obtaining $b_r^\varepsilon$, and put
$F_\varepsilon=\Phi(J_{a_1,b_1^\varepsilon},\ldots,J_{a_m,b_m^\varepsilon})$.
Then
\begin{align}
 \|b_r^\varepsilon\|_\infty&\leq\|b_r\|_\infty,
 \qquad \operatorname{Var}(b_r^\varepsilon)\leq\operatorname{Var}(b_r),
 \label{eq:BV-mollifier-uniform}\\
 \shortintertext{and bounded convergence gives}
 F_\varepsilon&\longrightarrow F\quad\text{in }L^2(\nu)
 \quad(\varepsilon\downarrow0).
 \label{eq:BV-mollifier-convergence}
\end{align}
For each basis direction, let $G_{n,\tau}$ denote the right side of
\eqref{eq:BV-first-derivative} with $v=e_{n,\tau}$.  The same angular
convergence gives
\begin{align*}
 D_{n,\tau}F_\varepsilon&\longrightarrow G_{n,\tau}
            \quad\text{pointwise and in }L^2(\nu),\\
 \shortintertext{with}
 |D_{n,\tau}F_\varepsilon-G_{n,\tau}|
 &\leq2C_F\alpha_n,\qquad
 \|D_{n,\tau}F_\varepsilon-G_{n,\tau}\|_2^2
 \leq4C_F^2\alpha_n^2,\qquad
 \sum_{n,\tau}4C_F^2\alpha_n^2<\infty.
\end{align*}
Define $G_F=\sum_{n,\tau}G_{n,\tau}e_{n,\tau}$ in
$L^2(\nu;\HWP)$.  Parseval's identity and dominated convergence give
\[
 \|DF_\varepsilon-G_F\|_2^2
 =\sum_{n,\tau}\|D_{n,\tau}F_\varepsilon-G_{n,\tau}\|_2^2
 \longrightarrow0\qquad(\varepsilon\downarrow0).
\]
Closedness of $D$ proves $F\in\Dcal$, $DF=G_F$, and
\eqref{eq:BV-first-derivative}.

For fixed $v$, the function
$G_v=-\sum_r\Phi_r(J)J_{(a_rv)',b_r}$ is again a BV cylinder after a
smooth cutoff outside its bounded coordinate ranges.  Applying the
first-order result to $G_v$ gives $G_v\in\Dcal$ and
\[
 D_vG_v=\sum_{r,s}\Phi_{rs}(J)
          J_{(a_rv)',b_r}J_{(a_sv)',b_s}
          +\sum_r\Phi_r(J)J_{((a_rv)'v)',b_r},
\]
which proves \eqref{eq:BV-second-derivative} as a closed derivative.
\end{proof}

For a BV cylinder define the finite-mode expression
\[
 L_NF=\sum_{n=2}^N\sum_{\tau\in\{c,s\}}
             (D_{n,\tau}^2F-\rho_{n,\tau}D_{n,\tau}F),\qquad N\geq2.
\]
The closed derivatives just identified satisfy
\begin{align}
 \|D_{n,\tau}F\|_\infty&\leq C_F\alpha_n,
 \qquad \|D_{n,\tau}^2F\|_\infty\leq C_F(1+n)\alpha_n^2,
 \label{eq:gen-cylinder-bounds}\\
 \shortintertext{and the score bound \eqref{eq:gen-score-basis-norm} gives}
 \|\rho_{n,\tau}D_{n,\tau}F\|_2
 &\leq\left(a_\kappa+\frac23\right)^{1/2}C_F\alpha_n.
 \notag
\end{align}
\begin{samepage}
Since $\alpha_n=O(n^{-3/2})$ and $(1+n)\alpha_n^2=O(n^{-2})$,
\begin{align}
 \sum_{n>N,\tau}\|D_{n,\tau}^2F\|_\infty
 &\leq2C_F\sum_{n>N}(1+n)\alpha_n^2
 \leq C_F\sum_{n>N}n^{-2}\leq C_FN^{-1},\notag\\
 \shortintertext{and}
 \sum_{n>N,\tau}\|\rho_{n,\tau}D_{n,\tau}F\|_2
 &\leq2C_F\left(a_\kappa+\frac23\right)^{1/2}
                    \sum_{n>N}\alpha_n\notag\\
 &\leq C_{F,\kappa}\sum_{n>N}n^{-3/2}
 \leq C_{F,\kappa}N^{-1/2}.
 \label{eq:generator-tail}
\end{align}
\end{samepage}
The same estimates with all $n\geq2$ prove absolute summability.
Their constants depend only on the bounds displayed in
\eqref{eq:BV-proof-directional-bounds} and the fixed value of $\kappa$.
Thus \eqref{eq:BV-mollifier-uniform} makes them uniform for the
mollifications $F_\varepsilon$.

The summable expression and its tail are therefore
\begin{align}
 L_0F&:=\lim_{N\to\infty}L_NF
       =\sum_{n\geq2,\tau}(D_{n,\tau}^2F-\rho_{n,\tau}D_{n,\tau}F)
                          \quad\text{in }L^2(\nu),\notag\\
 \shortintertext{with}
 \|L_NF-L_0F\|_2&\leq C_{F,\kappa}N^{-1/2}.
 \label{eq:gen-absolute-summability}
\end{align}
Its identification with $L$ is deferred to
Section~\ref{sec:proof-generator}.

\subsection{The trace and covariance kernel}

Pair the cosine and sine modes at each frequency and set
\begin{equation}\label{eq:gen-noise-kernel}
 K(\theta)=\sum_{n\geq2}\alpha_n^2\cos(n\theta),\qquad
 B_N(h,\theta)=\sum_{n=2}^N\sum_{\tau\in\{c,s\}}
                  \rho_{n,\tau}(h)e_{n,\tau}(\theta).
\end{equation}
When a single index is convenient, write
$(e_1,e_2,e_3,e_4,\ldots)=(e_{2,c},e_{2,s},e_{3,c},e_{3,s},\ldots)$,
$D_j=D_{e_j}$, and $\rho_j=\rho_{e_j}$.
The kernel series converges uniformly absolutely.  The scores satisfy
\begin{align}
 \sum_j\|\rho_je_j\|_{L^2(\nu;C(\T))}
 &\leq2\left(a_\kappa+\frac23\right)^{1/2}
                         \sum_{n\geq2}\alpha_n<\infty,
 \label{eq:gen-drift-Bochner-bound}\\
 \shortintertext{so the drift is well defined by}
 B&=\lim_{N\to\infty}B_N=\sum_j\rho_je_j
       \quad\text{in }L^2(\nu;C(\T)).
 \label{eq:gen-drift-field}
\end{align}
Tonelli's theorem also gives
\[
 \int_X\sum_j|\rho_j(h)|\|e_j\|_\infty\,d\nu
 \leq\sum_j\|\rho_j\|_2\|e_j\|_\infty<\infty.
\]
Use the uniform sum wherever the series of norms is finite and zero
elsewhere.  Its partial sums and their limit are measurable maps into
$C(\T)$, so evaluation gives a jointly measurable $B(h,\theta)$.

The drift contribution to $J_{a,b}$ is
\begin{align*}
 B_{a,b}(h)&=\frac1{2\pi}\int_\T a(\theta)B(h,\theta)\,db_h(\theta),\\
 \shortintertext{with}
 \left\|B_{a,b}-\frac1{2\pi}\int_\T aB_N\,db_h\right\|_2
 &\leq\|a\|_\infty V_b\|B-B_N\|_{L^2(\nu;C(\T))}
          \longrightarrow0.
\end{align*}
The finite sums are measurable; this convergence identifies
$B_{a,b}\in L^2(\nu)$.
For BV cylinders use the pointwise energy representative
\begin{equation}\label{eq:pointwise-cylinder-Gamma}
 \Gamma(F,G)(h)=\sum_{n,\tau}D_{n,\tau}F(h)D_{n,\tau}G(h).
\end{equation}
The series converges uniformly absolutely by
\eqref{eq:gen-cylinder-bounds}.

\begin{proposition}\label{prop:trace-drift}
For smooth source functions and bounded-variation target functions,
\begin{align}
 L_0J_{a,b}&=\frac12J_{a'',b}-B_{a,b},
 \label{eq:gen-J-explicit}\\
 \shortintertext{and}
 \Gamma(J_{a,b},J_{\widetilde a,\widetilde b})(h)
 &=\frac1{(2\pi)^2}\int_{\T^2}
 a(\theta)\widetilde a(\varphi)K(\theta-\varphi)
 \,db_h(\theta)\,d\widetilde b_h(\varphi).
 \label{eq:gen-Gamma-Stieltjes}
\end{align}
For $F=\Phi(J_1,\ldots,J_m)$,
\begin{equation}\label{eq:gen-cylinder-explicit}
 L_0F=\sum_r\Phi_r(J)
       \left(\frac12J_{a_r'',b_r}-B_{a_r,b_r}\right)
       +\sum_{r,s}\Phi_{rs}(J)\Gamma(J_r,J_s).
\end{equation}
\end{proposition}
\begin{proof}
We use the identities
\begin{align*}
 \sum_{n=2}^N\frac2{n^3-n}&=\frac12-\frac1{N(N+1)},\qquad
 e_{n,c}^2+e_{n,s}^2=\alpha_n^2,\qquad
 e_{n,c}e_{n,c}'+e_{n,s}e_{n,s}'=0,\\
 \shortintertext{so $K(0)=1/2$, while the trace integrands satisfy}
 \sum_{n,\tau}\|e_{n,\tau}(ae_{n,\tau})'\|_\infty
 &\leq2\sum_{n\geq2}\alpha_n^2
       (\|a'\|_\infty+n\|a\|_\infty)<\infty.
\end{align*}
This bound permits termwise Stieltjes integration, while
\eqref{eq:gen-drift-Bochner-bound} controls the drift sum.  The two
response formulas therefore give
\begin{align}
 \sum_jD_j^2J_{a,b}
 &=-\frac1{2\pi}\int_\T a'\sum_{n\geq2}\alpha_n^2\,db_h
 =-\frac1{4\pi}\int_\T a'\,db_h
 =\frac12J_{a'',b},
 \label{eq:gen-trace-collapse}\\
 \shortintertext{and}
 \sum_j\rho_jD_jJ_{a,b}
 &=\frac1{2\pi}\int_\T a\sum_j\rho_je_j\,db_h=B_{a,b}
       \quad\text{in }L^2(\nu).
 \label{eq:gen-drift-observable}
\end{align}
Subtracting the drift from the trace proves \eqref{eq:gen-J-explicit}.

Since $\sum_je_j(\theta)e_j(\varphi)=K(\theta-\varphi)$ uniformly
absolutely, the finite total variations of the two Stieltjes measures give
\begin{align*}
 \Gamma(J_{a,b},J_{\widetilde a,\widetilde b})
 &=\frac1{(2\pi)^2}\int_{\T^2}a(\theta)\widetilde a(\varphi)
       \sum_je_j(\theta)e_j(\varphi)\,db_h(\theta)\,d\widetilde b_h(\varphi)\\
 &=\frac1{(2\pi)^2}\int_{\T^2}a(\theta)\widetilde a(\varphi)
       K(\theta-\varphi)\,db_h(\theta)\,d\widetilde b_h(\varphi).
\end{align*}
Finally, absolute summability allows us to collect the cylinder terms:
\begin{align*}
 L_0F
 &=\sum_r\Phi_r(J)\sum_j(D_j^2J_r-\rho_jD_jJ_r)
   +\sum_{r,s}\Phi_{rs}(J)\sum_jD_jJ_rD_jJ_s\\
 &=\sum_r\Phi_r(J)L_0J_r
   +\sum_{r,s}\Phi_{rs}(J)\Gamma(J_r,J_s).
\end{align*}
Substitution of \eqref{eq:gen-J-explicit} proves
\eqref{eq:gen-cylinder-explicit}.
\end{proof}

\subsection{The stress drift}
\label{sec:stress-drift}

The stress representation strengthens the $L^2(\nu;C(\T))$ drift
bound to a uniform bound and a circle modulus of continuity.
Write $\mathbb D=\{z\in\mathbb C:|z|<1\}$.  The normalized-frame
construction in \cite{Response} supplies a measurable set
$X_{\rm nf}\subset W_{\rm rem}$ with $\nu(X_{\rm nf})=1$ and a measurable choice
$h\mapsto f_h^\circ$ of normalized interior maps, with the topology
of locally uniform convergence and
\[
 f_h^\circ(0)=0,\qquad (f_h^\circ)'(0)=1.
\]
We use this same set $X_{\rm nf}$ in the geometric realization.

For a normalized univalent function $f(z)=z+c_2z^2+\cdots$ on
$\mathbb D$, write
$\mathcal Sf=f'''/f'-\tfrac32(f''/f')^2$ and define
\begin{align}
 Q_f(z)&=a_\kappa\mathcal Sf(z)
       +\left(\frac{f'(z)}{f(z)}\right)^2-\frac1{z^2}-\frac{2c_2}{z},
 \label{eq:regular-stress}\\
 \shortintertext{with coefficients and primitive}
 Q_f(z)&=\sum_{n\geq2}t_n(f)z^{n-2},\qquad
 A_f(z)=\sum_{n\geq2}\frac{t_n(f)}{n^3-n}z^n.
 \label{eq:analytic-drift-primitive}
\end{align}
The cancellation at zero is justified below.  Let $A_f^*$ denote
the uniform radial boundary value when it exists, and put
$\delta=d_\T(\theta,\varphi)=
\min_{\ell\in\mathbb Z}|\theta-\varphi+2\pi\ell|$.
In this subsection $C_\kappa$ depends only on $\kappa$ and may
increase from one estimate to the next.

\begin{lemma}\label{lem:stress-primitive}
For every normalized univalent $f$, $A_f^*$ exists in $C(\T)$, and
\begin{align}
 \sup_{z\in\mathbb D}(1-|z|^2)^2|Q_f(z)|&\leq C_\kappa,
 \label{eq:uniform-Bers-stress}\\
 \shortintertext{and its primitive satisfies}
 \|A_f^*\|_\infty&\leq C_\kappa,
 \label{eq:stress-boundary-uniform}\\
 \|A_f^*(e^{i(\cdot+s)})+A_f^*(e^{i(\cdot-s)})
                     -2A_f^*(e^{i\cdot})\|_\infty
 &\leq C_\kappa|s|\qquad(|s|\leq1),
 \label{eq:stress-boundary-Zygmund}\\
 |A_f^*(e^{i\theta})-A_f^*(e^{i\varphi})|
 &\leq C_\kappa\delta\log(e/\delta)\qquad(0<\delta\leq1).
 \label{eq:stress-boundary-Osgood}
\end{align}
\end{lemma}
\begin{proof}
The expansion
\[
 \left(\frac{f'}f\right)^2-\frac1{z^2}
 =\frac{2c_2}{z}+O(1)\qquad(z\to0)
\]
shows that $Q_f$ is holomorphic.  For $|z|\geq1/2$ set
$R=(1-|z|)/2$.  The disk $B(z,R)$ avoids the zero of $f$, and a
branch of $\log f$ is univalent there.  The Schwarzian chain rule
and the univalent Schwarzian bound \cite{Nehari49} give
\[
 \mathcal S(\log f)=\mathcal Sf+\tfrac12(f'/f)^2,\qquad
 |\mathcal S(\log f)(z)|+|\mathcal Sf(z)|\leq12R^{-2},\qquad
 |f'/f|^2\leq24R^{-2}.
\]
Together with $|c_2|\leq2$ and
$(1-|z|^2)^2|\mathcal Sf(z)|\leq6$, this proves
\eqref{eq:uniform-Bers-stress} for $|z|\geq1/2$.
On the remaining disk, the maximum principle gives
\[
 \sup_{|z|\leq1/2}|Q_f(z)|
 \leq\sup_{|z|=1/2}|Q_f(z)|\leq\frac{16}{9}C_\kappa.
\]

Termwise differentiation gives
\begin{equation}\label{eq:drift-primitive-derivative}
 A_f''(z)=\sum_{n\geq2}\frac{t_n(f)}{n+1}z^{n-2}
         =z^{-3}\int_0^z w^2Q_f(w)\,dw\qquad(z\ne0).
\end{equation}
Since $A_f(0)=A_f'(0)=0$, radial integration for $z=r\zeta$,
$|\zeta|=1$, yields
\begin{align}
 |A_f''(r\zeta)|
 &\leq C_\kappa r^{-3}\int_0^r\frac{s^2}{(1-s^2)^2}\,ds
 \leq\frac{C_\kappa}{1-r},
 \label{eq:stress-second-bound}\\
 \shortintertext{and consequently}
 \sup_{|z|=r}|A_f'(z)|&\leq C_\kappa\log\frac e{1-r},\qquad
 \sup_{z\in\mathbb D}|A_f(z)|\leq C_\kappa.
 \label{eq:stress-first-zero-bounds}
\end{align}
For $0<r<s<1$, the first-derivative bound gives
\[
 \|A_f(s\,\cdot)-A_f(r\,\cdot)\|_\infty
 \leq C_\kappa\int_r^s\log\frac e{1-t}\,dt
 \leq C_\kappa(1-r)\log\frac e{1-r}.
\]
Thus $A_f(r\,\cdot)$ converges uniformly to $A_f^*\in C(\T)$, with
\begin{equation}\label{eq:stress-radial-error}
 \|A_f^*-A_f(r\,\cdot)\|_\infty
 \leq C_\kappa(1-r)\log\frac e{1-r}\longrightarrow0
 \qquad(r\uparrow1).
\end{equation}

To obtain the second-difference estimate without this logarithmic
loss, set
$G_r(\zeta)=A_f(r\zeta)+(1-r)\zeta A_f'(r\zeta)$.
Taylor's integral formula gives
\[
 A_f^*(\zeta)-G_r(\zeta)
 =\zeta^2\int_r^1(1-t)A_f''(t\zeta)\,dt,\qquad
 \|A_f^*-G_r\|_\infty\leq C_\kappa(1-r).
\]
Cauchy's estimate gives $\sup_{|z|=r}|A_f'''(z)|\leq
C_\kappa(1-r)^{-2}$.  With $\zeta=e^{i\theta}$ and all derivatives
of $A_f$ evaluated at $r\zeta$,
\begin{align*}
 \partial_\theta^2G_r
 &=-\zeta A_f'-(r^2+3r(1-r))\zeta^2A_f''
                    -r^2(1-r)\zeta^3A_f''',\\
 \shortintertext{hence}
 \|\partial_\theta^2G_r\|_\infty&\leq C_\kappa(1-r)^{-1}
                    \qquad(r\geq1/2).
\end{align*}
For $0<|s|\leq1/2$, choose $r=1-|s|$.  Then
\begin{align*}
 &\|A_f^*(e^{i(\cdot+s)})+A_f^*(e^{i(\cdot-s)})
                         -2A_f^*(e^{i\cdot})\|_\infty\\
 &\qquad\leq4\|A_f^*-G_r\|_\infty
            +s^2\|\partial_\theta^2G_r\|_\infty
 \leq C_\kappa|s|.
\end{align*}
For $0<\delta\leq1/2$, instead take $r=1-\delta$ and use
\begin{align*}
 |A_f^*(e^{i\theta})-A_f^*(e^{i\varphi})|
 &\leq2\|A_f^*-A_f(r\,\cdot)\|_\infty
      +\delta\sup_{|z|=r}|A_f'(z)|\\
 &\leq C_\kappa\delta\log(e/\delta).
\end{align*}
The uniform bound on $A_f^*$ treats $1/2<|s|\leq1$ and
$1/2<\delta\leq1$.
\end{proof}

Put $t_n(h)=t_n(f_h^\circ)$ on $X_{\rm nf}$.  The real score convention
in \cite{Response} reads
\begin{equation}\label{eq:stress-coefficients}
 \rho_{n,c}=\alpha_n\operatorname{Im}t_n,\qquad
 \rho_{n,s}=\alpha_n\operatorname{Re}t_n
       \quad\nu\text{-a.e. on }X_{\rm nf}.
\end{equation}
For a circle function $v$, let
$P_rv(\theta)=\sum_{n\in\mathbb Z}r^{|n|}v_ne^{in\theta}$
be its Poisson regularization, $0<r<1$.

\begin{proposition}\label{prop:uniform-drift}
The drift $B$ has a jointly measurable version satisfying, for all
$h\in X$,
\begin{align}
 \|B(h,\cdot)\|_\infty&\leq C_\kappa,
 \label{eq:uniform-drift-bound}\\
 |B(h,\theta+s)+B(h,\theta-s)-2B(h,\theta)|
 &\leq C_\kappa|s|\qquad(|s|\leq1),
 \label{eq:uniform-drift-Zygmund}\\
 |B(h,\theta)-B(h,\varphi)|
 &\leq C_\kappa\delta\log(e/\delta)\qquad(0<\delta\leq1).
 \label{eq:uniform-drift-Osgood}
\end{align}
For every BV cylinder $F$, $L_0F\in L^\infty(\nu)$.
\end{proposition}
\begin{proof}
Take the common measurable set
\[
 E_0=\left\{h\in X_{\rm nf}:
 \begin{array}{l}
 \sum_j|\rho_j(h)|\|e_j\|_\infty<\infty,\\
 \rho_{n,c}(h)=\alpha_n\operatorname{Im}t_n(h),\quad
 \rho_{n,s}(h)=\alpha_n\operatorname{Re}t_n(h)\quad(n\geq2)
 \end{array}\right\}.
\]
The summability above and the countably many equalities in
\eqref{eq:stress-coefficients} give $\nu(E_0)=1$.
For $h\in E_0$, uniform absolute convergence implies
\begin{align}
 (P_rB)(h,\theta)
 &=\sum_{n\geq2}\frac{2r^n}{n^3-n}
   \{(\operatorname{Im}t_n(h))\cos n\theta
           +(\operatorname{Re}t_n(h))\sin n\theta\}\notag\\
 &=2\operatorname{Im}A_{f_h^\circ}(re^{i\theta}).
 \label{eq:drift-Abel-identification}
\end{align}
Both sides converge in $C(\T)$ as $r\uparrow1$.  We may therefore
choose the version
\[
 B(h,\theta)=
 \begin{cases}
 2\operatorname{Im}A_{f_h^\circ}^*(e^{i\theta}),&h\in X_{\rm nf},\\
 0,&h\notin X_{\rm nf}.
 \end{cases}
\]
For each fixed $r<1$, the coefficients of the holomorphic map give
a measurable $C(\T)$-valued function
$h\mapsto A_{f_h^\circ}(r\,\cdot)$.  Its limit along
$r_k=1-2^{-k}$ has the same measurability.  Thus the displayed version
is jointly measurable, agrees with \eqref{eq:gen-drift-field} almost
everywhere, and satisfies the three bounds by
Lemma~\ref{lem:stress-primitive}.

The Stieltjes representation and $|K(\theta)|\leq K(0)=1/2$ give
\begin{align*}
 \|B_{a,b}\|_\infty&\leq C_\kappa\|a\|_\infty V_b,\\
 \shortintertext{and}
 \|\Gamma(J_r,J_s)\|_\infty
 &\leq\tfrac12\|a_r\|_\infty\|a_s\|_\infty V_{b_r}V_{b_s}.
\end{align*}
Consequently, \eqref{eq:gen-cylinder-explicit} yields
\begin{align*}
 \|L_0F\|_\infty
 &\leq\sum_r\|\Phi_r\|_\infty
       \left(\tfrac12\|a_r''\|_\infty\|b_r\|_\infty
                       +C_\kappa\|a_r\|_\infty V_{b_r}\right)\\
 &\quad+\tfrac12\sum_{r,s}\|\Phi_{rs}\|_\infty
          \|a_r\|_\infty\|a_s\|_\infty V_{b_r}V_{b_s}<\infty.\qedhere
\end{align*}
\end{proof}

\subsection{Final proof of Theorem~\ref{thm:generator-main}}
\label{sec:proof-generator}

The form representation identifies the coordinate-free operator:
for $F\in\Dcal$ and $H\in L^2(\nu)$,
\begin{align*}
 F\in\Dom(L),\ LF=H
 &\ \Longleftrightarrow\ 
 \E(F,G)=-\langle H,G\rangle_2\quad(G\in\Dcal)\\
 &\ \Longleftrightarrow\ 
 DF\in\Dom(D^*),\ D^*DF=-H.
\end{align*}
Thus \eqref{eq:gen-DstarD} holds.  To identify its Fourier realization,
fix $F,G\in\Ccyl$.  For each finite field $v$, the response formulas
express $D_vF$ as a smooth cylinder after a cutoff outside bounded
coordinate ranges; hence $(D_vF)G\in\Ccyl$.  The adjoint identity and
product rule give
\[
 \int_X\rho_v(D_vF)G\,d\nu
 =\int_XD_v\{(D_vF)G\}\,d\nu
 =\int_X\{(D_v^2F)G+(D_vF)(D_vG)\}\,d\nu.
\]
Let $P_N$ be the orthogonal projection onto the modes $2\leq n\leq N$.
Summing the preceding identity gives
\[
 -\langle L_NF,G\rangle_2
 =\int_X(P_NDF,DG)_{\WP}\,d\nu.
\]
By \eqref{eq:gen-absolute-summability} and Parseval's identity,
\begin{align*}
 \|L_NF-L_0F\|_2&\longrightarrow0,\qquad
 \|(1-P_N)DF\|_2\longrightarrow0,\\
 \shortintertext{and therefore}
 |\E(F,G)+\langle L_0F,G\rangle_2|
 &\leq\|(1-P_N)DF\|_2\|DG\|_2
       +\|L_NF-L_0F\|_2\|G\|_2\longrightarrow0.
\end{align*}
Thus $\E(F,G)=-\langle L_0F,G\rangle_2$ for $G\in\Ccyl$.
For an arbitrary $G\in\Dcal$, choose $G_j\in\Ccyl$ converging in
form norm.  Then
\[
 |\E(F,G-G_j)+\langle L_0F,G-G_j\rangle_2|
 \leq(\|DF\|_2+\|L_0F\|_2)\|G-G_j\|_{\E_1}\longrightarrow0.
\]
It follows that $F\in\Dom(L)$ and $LF=L_0F$.
The summability and tail estimates
\eqref{eq:gen-cylinder-bounds}--\eqref{eq:gen-absolute-summability} prove
\eqref{eq:gen-infinite-Fourier}--\eqref{eq:main-generator-summability},
and Proposition~\ref{prop:uniform-drift} gives $LF\in L^\infty(\nu)$.

We also record the BV extension needed in Section~\ref{sec:boundary-dynamics}.
For a BV cylinder $F$, use $F_\varepsilon$ from
Lemma~\ref{lem:BV-directional-calculus}.  For each fixed $N$,
\begin{equation}\label{eq:BV-proof-fixed-N}
 \|F_\varepsilon-F\|_2\longrightarrow0,\qquad
 \|L_NF_\varepsilon-L_NF\|_2\longrightarrow0
                   \quad(\varepsilon\downarrow0).
\end{equation}
Indeed, the first- and second-derivative coordinate formulas converge
pointwise with uniform bounds; multiplication by the finitely many
$\rho_{n,\tau}\in L^2(\nu)$ is handled by dominated convergence.
The uniform tail estimate then yields
\begin{align*}
 \|LF_\varepsilon-L_0F\|_2
 &\leq2C_{F,\kappa}N^{-1/2}
            +\|L_NF_\varepsilon-L_NF\|_2,\\
 \shortintertext{so, first letting $\varepsilon\downarrow0$ at fixed $N$,}
 \limsup_{\varepsilon\downarrow0}\|LF_\varepsilon-L_0F\|_2
 &\leq2C_{F,\kappa}N^{-1/2}.\\
 \shortintertext{Letting $N\to\infty$ then gives}
 \lim_{\varepsilon\downarrow0}\|LF_\varepsilon-L_0F\|_2&=0.
\end{align*}
Closedness of $L$ proves
\begin{equation}\label{eq:BV-generator}
 F\in\Dom(L),\qquad
 LF=L_0F=\sum_{n\geq2,\tau}
              (D_{n,\tau}^2F-\rho_{n,\tau}D_{n,\tau}F)
                 \quad\text{in }L^2(\nu).
\end{equation}
The two series of norms are finite by \eqref{eq:generator-tail}.

The operator realization is also determined.  For
$L_{\rm cyl}=L|_{\Ccyl}$,
\[
 \langle-L_{\rm cyl}F,G\rangle_2=\E(F,G)\quad(F,G\in\Ccyl),
 \qquad \overline{\Ccyl}^{\|\cdot\|_{\E_1}}=\Dcal.
\]
Thus the closed form of the Friedrichs extension of $-L_{\rm cyl}$
is $\E$, and that extension is $-L$.  For $u\in\Dom(L)$, choose
$f_j\in\Ccyl$ with $f_j\to(1-L)u$ in $L^2(\nu)$ and set
$u_j=(1-L)^{-1}f_j$.  The spectral bounds and the resulting graph-norm approximation are
\begin{align*}
 \|(1-L)^{-1}\|&\leq1,\qquad \|L(1-L)^{-1}\|\leq1,\\
 \shortintertext{and hence}
 \|u_j-u\|_2^2+\|Lu_j-Lu\|_2^2
 &\leq2\|f_j-(1-L)u\|_2^2\longrightarrow0.
\end{align*}
Hence $(1-L)^{-1}\Ccyl$ is an operator core for $L$.

\subsection{A compact regular realization}
\label{sec:compactification}

To compute energy measures, we transfer the form to a compact space
without changing the diffusion outside exceptional sets.
Take the form-dense, point-separating sequence
$(F_j)\subset\Ccyl$ from Lemma~\ref{lem:state-space-core}, include the
constant function $\mathbf 1$, and set
\[
 \Xi(h)=\left(\frac{F_j(h)}{1+\sup_{g\in X}|F_j(g)|}\right)_{j\geq1},
 \qquad
 \bar\nu=\Xi_*\nu,\qquad
 \widehat X=\operatorname{supp}\bar\nu
       \subset[-1,1]^{\mathbb N}.
\]
The map $\Xi$ is continuous and injective.  Let
$X_\nu=\Xi^{-1}(\widehat X)$ and $O=X\setminus X_\nu$.
The set $X_\nu$ is closed and $\nu(X_\nu)=1$.
The Lusin--Souslin theorem makes $\Xi(X_\nu)$ measurable for the
sigma-algebra generated by open subsets of $\widehat X$ and gives a
measurable inverse on that image.  Pullback and its inverse are
\begin{align*}
 \Xi^*f&=f\circ\Xi\text{ on }X_\nu,\qquad
 (\Xi^*)^{-1}F=F\circ\Xi^{-1}\text{ on }\Xi(X_\nu),\\
 \shortintertext{with zero extensions off the stated domains, and}
 \|\Xi^*f\|_2^2&=\int_{\widehat X}|f|^2\,d\bar\nu,
 \qquad \bar\nu(\Xi(X_\nu))=1.
\end{align*}
These maps give a unitary correspondence between
$L^2(\widehat X,\bar\nu)$ and $L^2(X,\nu)$.
Define
\begin{align*}
 \widehat{\Dcal}&=(\Xi^*)^{-1}\Dcal,\qquad
 \widehat{\E}(f,g)=\E(\Xi^*f,\Xi^*g),\\
 \shortintertext{and use the transported form norm}
 \|f\|_{\widehat{\E}_1}&=\|\Xi^*f\|_{\E_1}.
\end{align*}
Let $C_{\rm pol}$ be the real algebra of coordinate polynomials on
$\widehat X$.

\begin{proposition}\label{prop:compactification}
The coordinate algebra satisfies
\begin{align}
 C_{\rm pol}&\subset\widehat{\Dcal}\cap C(\widehat X),\notag\\
 \shortintertext{with density in both norms,}
 \overline{C_{\rm pol}}^{\|\cdot\|_{\widehat{\E}_1}}
 &=\widehat{\Dcal},\qquad
 \overline{C_{\rm pol}}^{\|\cdot\|_\infty}=C(\widehat X).
 \label{eq:compact-regular-cores}
\end{align}
There are compact form nests $(K'_m)$ for $\E$ and
$(\widehat K_m)$ for $\widehat\E$ such that
\[
 \Xi:K'_m\longrightarrow\widehat K_m
 \quad\text{is a homeomorphism for every }m,\qquad
 \widehat X\setminus\Xi(X)
 \subset\widehat X\setminus\bigcup_m\widehat K_m.
\]
The transported form is strongly local.
\end{proposition}
\begin{proof}
Pullbacks of the coordinate functions are nonzero multiples of the
$F_j$, modulo null sets.  The bounded product rule and
Stone--Weierstrass theorem prove the form-norm and uniform densities
in \eqref{eq:compact-regular-cores}, respectively.  These identities
establish regularity on the compact metric space $\widehat X$.

For the compact nest from Theorem~\ref{thm:main}, set
\[
 K'_m=K_m\cap X_\nu,\qquad \widehat K_m=\Xi(K'_m).
\]
Both sets are compact.  Since $O$ is open and $\nu(O)=0$, zero is
admissible in its capacity definition, so $\CapOne(O)=0$.
Moreover,
\[
 \{F\in\Dcal:F=0\text{ a.e. on }X\setminus K_m\}
 =\{F\in\Dcal:F=0\text{ a.e. on }X\setminus K'_m\}.
\]
Thus $(K'_m)$ is a form nest.  For each measurable
$C\subset\widehat X$, put
$\widehat{\Dcal}_C=\{f\in\widehat{\Dcal}:
 f=0\ \bar\nu\text{-a.e. on }\widehat X\setminus C\}$.
The pullback isometry gives
\begin{align*}
 \Xi^*\widehat{\Dcal}_{\widehat K_m}
 &=\{F\in\Dcal:F=0\ \nu\text{-a.e. on }X\setminus K'_m\},\\
 \shortintertext{and consequently}
 \overline{\bigcup_m\widehat{\Dcal}_{\widehat K_m}}
       ^{\|\cdot\|_{\widehat{\E}_1}}
 &=\widehat{\Dcal}.
\end{align*}
The continuous bijection $\Xi:K'_m\to\widehat K_m$ has a continuous
inverse by compactness.  These restrictions give the
quasi-homeomorphism.  Since $\widehat K_m\subset\Xi(X)$, the displayed
inclusion in the statement places the added boundary in an
exceptional set.

Finally, suppose $f,g\in\widehat{\Dcal}$, $f=c$ almost everywhere on
an open subset $V\subset\widehat X$, and $g=0$ almost everywhere off
$V$.  Write $V=\widehat X\cap W$ with $W$ open in the cube, and set
$V_X=\Xi^{-1}(W)$.  Then $V_X$ is open in $X$, and
\[
 \Xi^*f=c\quad\nu\text{-a.e. on }V_X,\qquad
 \Xi^*g=0\quad\nu\text{-a.e. on }X\setminus V_X.
\]
Lemma~\ref{lem:strong-locality} gives
$\widehat{\E}(f,g)=\E(\Xi^*f,\Xi^*g)=0$, proving the required
locality on the regular realization.
\end{proof}

\subsection{Energy measures}

Let $\mathbb P_h$ and $\mathbb E_h$ denote the Markov law started
from $h$ and its expectation.  Write
$\mathbb P=\int_X\mathbb P_h\,\nu(dh)$ and $\mathbb E$ for the
stationary law and expectation.  Under $\mathbb P$, $(\mathcal G_t)$
denotes the completed right-continuous natural filtration.
Statements for individual initial states use the corresponding
usual augmentation under $\mathbb P_h$.

For $F\in\Dcal$, the quasi-regular Fukushima
theorem \cite[Theorem~2.1]{AMR} gives
\begin{equation}\label{eq:general-Fukushima}
 F(X_t)-F(X_0)=M_t^F+N_t^F,
\end{equation}
where $M^F$ is a martingale additive functional of finite energy and
$N^F$ is a continuous additive functional of zero energy.
Strong locality on the regular
realization eliminates the jump and killing terms; hence $M^F$ is
continuous by \cite[Theorem~5.2.2 and Lemma~5.5.1(ii)]{FOT}.
These assertions transfer back through $\Xi$.
Write $\mu_F$ for the Revuz measure of $\langle M^F\rangle$.
For each fixed finite or countable family of functions, the associated
identities are taken outside one properly exceptional set; they then
hold almost surely simultaneously for all $t\geq0$.

\begin{lemma}\label{lem:gen-energy-measures}
For every $F\in\Dcal$,
\begin{equation}\label{eq:energy-measure-density}
 \mu_F=2\Gamma(F)\nu.
\end{equation}
\end{lemma}
\begin{proof}
First let $F\in\Dcal\cap L^\infty(\nu)$.  For a bounded domain
function $\varphi$, denote its quasi-continuous representative by
$\widetilde\varphi$.  The energy-measure identity
\cite[Theorem~5.2.3]{FOT} and the closed product rule give
\begin{align*}
 \int_X\widetilde\varphi\,d\mu_F
 &=2\E(F,F\varphi)-\E(F^2,\varphi)\\
 &=2\int_X\varphi\|DF\|_{\WP}^2\,d\nu.
\end{align*}
On the compact realization we may test on
$\widehat{\Dcal}\cap C(\widehat X)$, which is uniformly dense in
$C(\widehat X)$ by \eqref{eq:compact-regular-cores}.
The two finite measures therefore agree, proving
\eqref{eq:energy-measure-density} for bounded $F$.

For general $F\in\Dcal$, choose smooth bounded $\chi_k$ with
$\chi_k(0)=0$, $|\chi_k'|\leq1$, and $\chi_k(x)=x$ for $|x|\leq k$.
Set $F_k=\chi_k(F)$.  The chain rule and dominated convergence give
\begin{align*}
 DF_k&=\chi_k'(F)DF\longrightarrow DF
              \quad\text{in }L^2(\nu;\HWP),\\
 \shortintertext{and}
 \|F_k-F\|_{\E_1}^2
 &=\|\chi_k(F)-F\|_2^2
       +\|(\chi_k'(F)-1)DF\|_2^2\longrightarrow0.
\end{align*}
The general energy-measure mass bound
$\mu_H(X)\leq2\E(H,H)$ and its Cauchy--Schwarz inequality imply
\begin{align*}
 \|\mu_{F_k}-\mu_F\|_{\rm TV}
 &\leq2\|D(F_k-F)\|_2(\|DF_k\|_2+\|DF\|_2),\\
 \shortintertext{while the densities satisfy}
 \|2\Gamma(F_k)-2\Gamma(F)\|_{L^1(\nu)}
 &\leq2\|D(F_k-F)\|_2(\|DF_k\|_2+\|DF\|_2).
\end{align*}
Both bounds tend to zero.  Passing to the limit in
$\mu_{F_k}=2\Gamma(F_k)\nu$ proves the lemma.
\end{proof}

\subsection{Proof of Corollary~\ref{cor:fukushima-main}}
\label{sec:proof-fukushima}

We first identify the bracket for any fixed $F\in\Dcal$.
Choose a finite nonnegative measurable representative
$a=2\Gamma(F)$ and put
$A_t^{(k)}=\int_0^t(a\wedge k)(X_s)\,ds$.
Its Revuz measure is $(a\wedge k)\nu$.
By \eqref{eq:energy-measure-density},
$\mu_F-(a\wedge k)\nu$ is a positive smooth measure.
Additivity and uniqueness of the Revuz correspondence
\cite[Theorems~5.1.3--5.1.4]{FOT} give
\[
 A_t^{(k)}\leq\langle M^F\rangle_t
 \quad\text{for all }t\geq0
\]
outside a properly exceptional set.  Taking one exceptional set for
the countably many $k$, monotone convergence gives
\[
 A_t:=\int_0^t a(X_s)\,ds
 =\lim_{k\to\infty}A_t^{(k)}
 \leq\langle M^F\rangle_t<\infty.
\]
For each such path this finite integral is absolutely continuous on
every bounded time interval.  For bounded nonnegative measurable $b$,
stationarity and Tonelli's theorem give
\[
 \frac1t\mathbb E\int_0^t b(X_s)\,dA_s
 =\int_X ba\,d\nu.
\]
Thus $A$ has Revuz measure $a\nu$, and uniqueness yields
\begin{align}
 \langle M^F\rangle_t&=2\int_0^t\Gamma(F)(X_s)\,ds,
 \label{eq:gen-Fukushima-bracket}\\
 \shortintertext{and polarization gives, for $F,G\in\Dcal$,}
 \langle M^F,M^G\rangle_t&=2\int_0^t\Gamma(F,G)(X_s)\,ds.
 \label{eq:gen-cross-bracket}
\end{align}
The cross integral is finite by Cauchy--Schwarz and the two diagonal
identities.

Now suppose $F\in\Dom(L)$ and choose a finite real measurable
representative $H=LF$.  Let $u=(1-L)^{-1}|H|$ and choose its quasi-continuous representative
$\widetilde u$.  To pass from bounded to $L^2$ resolvent data, put
$u_k=(1-L)^{-1}(|H|\wedge k)$.  The resolvent bound gives
\[
 \|u_k-u\|_{\E_1}
 \leq\||H|\wedge k-|H|\|_2\longrightarrow0.
\]
A subsequence of the quasi-continuous representatives converges to
$\widetilde u$ quasi-everywhere.  On a common exceptional-set complement,
the bounded-data resolvent formula and monotone convergence therefore give
\[
 \widetilde u(h)=\mathbb E_h\int_0^\infty e^{-s}|H(X_s)|\,ds.
\]
The weak resolvent identity and this formula imply
\begin{align*}
 \|u\|_{\E_1}^2&=\langle|H|,u\rangle_2
                         \leq\|H\|_2\|u\|_{\E_1},\\
 \mathbb E_h\int_0^T|H(X_s)|\,ds
 &\leq e^T\mathbb E_h\int_0^\infty e^{-s}|H(X_s)|\,ds
 =e^T\widetilde u(h)<\infty
\end{align*}
for quasi-every $h$.  Using positive integers $T$ gives finite
integrals on all bounded time intervals outside one properly
exceptional set.  Hence
$A_t^H=\int_0^tH(X_s)\,ds$ is a continuous additive functional and
$\mathbb E_h|A_t^H|<\infty$ there.
Cauchy--Schwarz and stationarity give its zero-energy estimate:
\[
 \frac1t\mathbb E|A_t^H|^2
 \leq\int_0^t\mathbb E|H(X_s)|^2\,ds
 =t\|H\|_2^2\longrightarrow0\qquad(t\downarrow0).
\]
For every $\varphi\in\Dcal$,
\begin{align*}
 \int_0^t\langle|\varphi|,T_s|H|\rangle_2\,ds
 &\leq\int_0^t\|\varphi\|_2\|T_s|H|\|_2\,ds
 \leq t\|\varphi\|_2\|H\|_2,\\
 \shortintertext{so Fubini and strong continuity give}
 \frac1t\int_X\varphi(h)\mathbb E_h[A_t^H]\,\nu(dh)
 &=\frac1t\int_0^t\langle\varphi,T_sH\rangle_2\,ds\\
 &\longrightarrow\langle\varphi,H\rangle_2=-\E(F,\varphi).
\end{align*}
This verifies condition (iii) in the zero-energy characterization
\cite[Theorem~5.2.4]{FOT}, applied on the regular realization.
Therefore
\begin{equation}\label{eq:gen-Fukushima-generator}
 N_t^F=\int_0^tLF(X_s)\,ds.
\end{equation}
Combining \eqref{eq:general-Fukushima},
\eqref{eq:gen-Fukushima-bracket}, and
\eqref{eq:gen-Fukushima-generator} proves \eqref{eq:main-Fukushima}.
Changing the chosen densities on a $\nu$-null set does not change the
corresponding time integrals: the occupation functional of that set
has zero Revuz measure.  All equalities consequently hold almost
surely for every $t\geq0$, for quasi-every initial state, with the
version convention stated above.

\section{Boundary dynamics}
\label{sec:boundary-dynamics}

This section supplies the boundary dynamics used in the geometric realization of
Theorem~\ref{thm:full-main}.  We first identify the inverse boundary
evaluations, then reconstruct one common Fourier noise, and finally approximate
the resulting rough circle flow by smooth flows.

\subsection{Inverse boundary evaluations}
\label{sec:inverse-points}

For $z\in\T$, set $q_z(h)=h^{-1}(z)\in\T$.
For $a\in C^\infty(\T;\R)$, define
\[
 F_{z,a}(h)=a(q_z(h)).
\]
Periodic functions of $q_z$ are independent of any choice of angular lift.
To place $F_{z,a}$ in the form domain, we use averages over shrinking
target arcs. For a proper oriented target arc $I$ and $h\in X$, choose a lift
$(x,y)$ of $h^{-1}(I)$ with $0<l:=y-x<2\pi$, and define
\begin{equation}\label{eq:inverse-arc-average}
 p_I=\frac l{2\pi},\qquad
 r_I=\frac1{2\pi}\int_x^ya(\theta)\,d\theta,
 \qquad A_I=\frac{r_I}{p_I}=\frac1l\int_x^ya(\theta)\,d\theta.
\end{equation}

Since
\[
 r_I=J_{a,\mathbf 1_I},\qquad p_I=J_{1,\mathbf 1_I},
 \qquad d(\mathbf 1_I\circ h)=\delta_x-\delta_y,
\]
Lemma~\ref{lem:BV-directional-calculus} gives
\begin{equation}\label{eq:inverse-arc-endpoint-gradients}
 D_vp_I=\frac{v(x)-v(y)}{2\pi},\qquad
 D_vr_I=\frac{a(x)v(x)-a(y)v(y)}{2\pi}.
\end{equation}
Define $H_I$ by its coordinates in the orthonormal WP basis,
\begin{equation}\label{eq:inverse-arc-HI-coordinates}
 (H_I,e_{n,\tau})_{\WP}
 =\frac{(a(x)-A_I)e_{n,\tau}(x)
       -(a(y)-A_I)e_{n,\tau}(y)}l.
\end{equation}
The endpoint estimates give
\begin{align*}
 |(H_I,e_{n,\tau})_{\WP}|
 &\leq2\|a'\|_\infty\alpha_n,\\
 \shortintertext{and, since $\sum_{n,\tau}\alpha_n^2=1$,}
 \sum_{n,\tau}|(H_I,e_{n,\tau})_{\WP}|^2
 &\leq4\|a'\|_\infty^2<\infty.
\end{align*}
Moreover,
\[
 \sup_{h\in X}\sum_{n>N,\tau}|(H_I,e_{n,\tau})_{\WP}|^2
 \leq4\|a'\|_\infty^2\sum_{n>N,\tau}\alpha_n^2
 \longrightarrow0.
\]
The inverse endpoints and $A_I$ depend continuously on $h$, so the
finite partial sums are continuous.  Hence their uniform $\HWP$ limit
defines a bounded continuous map
$H_I:X\to\HWP$.  For every finite field $v$,
\[
 (H_I,v)_{\WP}
 =\frac{(a(x)-A_I)v(x)-(a(y)-A_I)v(y)}l.
\]
Write $1-g_{I,R}=\eta_R(p_I)$, with $g_{I,R}$ from
\eqref{eq:lower-cutoff}, and set $q_R(p)=\eta_R(p)/p$.  Since $q_R$
has a bounded smooth extension,
\begin{align*}
 A_{I,R}&=q_R(p_I)r_I\in\Dcal,\\
 \shortintertext{with}
 DA_{I,R}&=q_R(p_I)Dr_I+r_Iq_R'(p_I)Dp_I.
\end{align*}
The endpoint formulas give
\[
 Dr_I-A_IDp_I=p_IH_I,
 \qquad r_I=p_IA_I,
 \qquad D\eta_R(p_I)=-Dg_{I,R},
\]
and therefore
\begin{equation}\label{eq:inverse-arc-truncated-gradient}
 DA_{I,R}=\eta_R(p_I)H_I-A_IDg_{I,R}.
\end{equation}
Using $|A_I|\leq\|a\|_\infty$ and \eqref{eq:g-zero},
\begin{align*}
 \|A_{I,R}-A_I\|_2
 &\leq\|a\|_\infty\|g_{I,R}\|_2\longrightarrow0,\\
 \shortintertext{and}
 \|DA_{I,R}-H_I\|_2
 &\leq\|H_I\|_\infty\|g_{I,R}\|_2
       +\|a\|_\infty\|Dg_{I,R}\|_2\longrightarrow0
 \qquad(R\to\infty).
\end{align*}
Closedness gives $A_I\in\Dcal$ and $DA_I=H_I$.  Therefore, for every
finite real field $v$,
\begin{align}
 D_vA_I
 &=\frac{(a(x)-A_I)v(x)-(a(y)-A_I)v(y)}l,
 \label{eq:inverse-average-gradient}\\
 \shortintertext{and}
 |D_{n,\tau}A_I|&\leq2\|a'\|_\infty\alpha_n.
 \label{eq:inverse-average-gradient-bound}
\end{align}

\begin{lemma}\label{lem:inverse-evaluation-domain}
The function $F_{z,a}$ belongs to $\Dcal\cap C_b(X)$. For every finite
real field $e$,
\begin{equation}\label{eq:inverse-evaluation-gradient}
 D_eF_{z,a}(h)=-a'(q_z(h))e(q_z(h)).
\end{equation}
The derivative in this formula is also the derivative of
$t\mapsto F_{z,a}(h\circ\psi_t^e)$ at zero.
\end{lemma}

\begin{proof}
Choose proper oriented target arcs $I_m\downarrow z$ from the positive
direction, and define $H_z$ by
\[
 (H_z,e_{n,\tau})_{\WP}
   =-a'(q_z)e_{n,\tau}(q_z).
\]
Let $\omega_a$ be the modulus of continuity of $a$. Continuity of
inversion gives, whenever $h_k\to h$ in $X$,
\[
 |F_{z,a}(h_k)-F_{z,a}(h)|
 \leq\omega_a\!\left(d_\T(h_k^{-1}(z),h^{-1}(z))\right)
 \longrightarrow0.
\]
Thus $F_{z,a}\in C_b(X)$. Write the inverse lift of $I_m$ as
$(x_m,y_m)$, put $l_m=y_m-x_m$, and choose the lifts so that
$x_m,y_m\to q_z$ and $l_m\downarrow0$. Then
\begin{align*}
 A_{I_m}&=\int_0^1a(x_m+ul_m)\,du,\\
 \frac{a(x_m)-A_{I_m}}{l_m}
 &=-\int_0^1(1-s)a'(x_m+sl_m)\,ds,\\
 \frac{a(y_m)-A_{I_m}}{l_m}
 &=\int_0^1s a'(x_m+sl_m)\,ds.
\end{align*}
The average formula and \eqref{eq:inverse-average-gradient} give,
for every $h\in X$ and finite field $e$,
\begin{align*}
 A_{I_m}(h)&\longrightarrow a(q_z(h))=F_{z,a}(h),\\
 \shortintertext{and}
 D_eA_{I_m}&\longrightarrow-a'(q_z)e(q_z)=(H_z,e)_{\WP}.
\end{align*}
Moreover,
\begin{align*}
 \|H_z\|_{\WP}^2
 &=|a'(q_z)|^2\sum_{n,\tau}e_{n,\tau}(q_z)^2
 =\frac12|a'(q_z)|^2,\\
 \shortintertext{and}
 |D_{n,\tau}A_{I_m}|
   +|(H_z,e_{n,\tau})_{\WP}|
 &\leq3\|a'\|_\infty\alpha_n.
\end{align*}
The coefficient functions of $H_z$ are continuous and their squared
tails are bounded uniformly by
\[
 \|a'\|_\infty^2\sum_{n>N,\tau}\alpha_n^2\longrightarrow0.
\]
Thus $H_z:X\to\HWP$ is bounded and continuous.
Since $\sum_{n,\tau}\alpha_n^2<\infty$, dominated convergence in
the basis sum and then over $(X,\nu)$ gives
\[
 \|A_{I_m}-F_{z,a}\|_2^2+\|DA_{I_m}-H_z\|_2^2
 \longrightarrow0,\qquad m\to\infty.
\]
Closedness proves \eqref{eq:inverse-evaluation-gradient}. Its
actual-right interpretation follows from
\[
 (h\circ\psi_t^e)^{-1}(z)=\psi_{-t}^e(h^{-1}(z)).
\]
\end{proof}

The gradient formula makes the generator accessible by Fourier summation.
\begin{proposition}\label{prop:inverse-evaluation-generator}
For $z,w\in\T$ and smooth real periodic functions $a,b$, the function
$F_{z,a}$ belongs to $\Dom(L)$, with
\begin{align}
 LF_{z,a}(h)&=\frac12a''(q_z(h))+B(h,q_z(h))a'(q_z(h)),
 \label{eq:inverse-evaluation-generator}\\
 \shortintertext{and its cross energy with $F_{w,b}$ is}
 \Gamma(F_{z,a},F_{w,b})(h)
 &=a'(q_z(h))b'(q_w(h))K(q_z(h)-q_w(h)).
 \label{eq:inverse-evaluation-cross-energy}
\end{align}
\end{proposition}

\begin{proof}
For a fixed $e$, the function $D_eF_{z,a}$ is another inverse evaluation,
with smooth periodic test function $-a'e$.  Thus Lemma~\ref{lem:inverse-evaluation-domain} gives
\[
 D_eF_{z,a}\in\Dcal,\qquad
 D_e^2F_{z,a}=a''(q_z)e(q_z)^2+a'(q_z)e'(q_z)e(q_z).
\]
Hence
\begin{align*}
 \|D_{n,\tau}F_{z,a}\|_\infty
 &\leq\|a'\|_\infty\alpha_n,\\
 \|D_{n,\tau}^2F_{z,a}\|_\infty
 &\leq\|a''\|_\infty\alpha_n^2
       +\|a'\|_\infty n\alpha_n^2.
\end{align*}
Using \eqref{eq:gen-score-basis-norm},
\[
 \begin{aligned}
 &\sum_{n\geq2,\,\tau}
 \|D_{n,\tau}^2F_{z,a}
          -\rho_{n,\tau}D_{n,\tau}F_{z,a}\|_2\\
 &\quad\leq2\sum_{n\geq2}
 \left\{\|a''\|_\infty\alpha_n^2
       +\|a'\|_\infty n\alpha_n^2
       +\left(a_\kappa+\frac23\right)^{1/2}
                         \|a'\|_\infty\alpha_n\right\}<\infty.
 \end{aligned}
\]
Hence
\begin{equation}\label{eq:inverse-absolute-generator-series}
 \ell_{z,a}:=
 \sum_{n\geq2,\,\tau}
 \{D_{n,\tau}^2F_{z,a}
              -\rho_{n,\tau}D_{n,\tau}F_{z,a}\}
 \quad\text{in }L^2(\nu).
\end{equation}
Let $\ell_{z,a}^{(N)}$ be the partial sum through frequency $N$.
For $\varphi\in\Ccyl$, the bounded domain function
$\varphi D_{n,\tau}F_{z,a}$ is admissible in the one-direction IBP.
Thus
\[
 \begin{aligned}
 -\int\varphi\,\ell_{z,a}^{(N)}\,d\nu
 &=\sum_{\substack{2\leq n\leq N\\\tau}}
       \int D_{n,\tau}\varphi\,D_{n,\tau}F_{z,a}\,d\nu\\
 &\longrightarrow\E(\varphi,F_{z,a}),
 \qquad N\to\infty.
 \end{aligned}
\]
Since $\ell_{z,a}^{(N)}\to\ell_{z,a}$ in $L^2(\nu)$ and
$\Ccyl$ is a form core,
\[
 \E(F_{z,a},\varphi)=-\langle\ell_{z,a},\varphi\rangle_2,
 \qquad\varphi\in\Dcal.
\]
Hence $F_{z,a}\in\Dom(L)$ and $LF_{z,a}=\ell_{z,a}$.

For every frequency,
\[
 e_{n,c}^2+e_{n,s}^2=\alpha_n^2,\qquad
 e_{n,c}e'_{n,c}+e_{n,s}e'_{n,s}=0,
 \qquad \sum_{n\ge2}\alpha_n^2=\frac12.
\]
Therefore
\begin{align*}
 \sum_{n,\tau}D_{n,\tau}^2F_{z,a}&=\frac12a''(q_z),\\
 \shortintertext{and}
 -\sum_{n,\tau}\rho_{n,\tau}D_{n,\tau}F_{z,a}
 &=B(\,\cdot\,,q_z)a'(q_z).
\end{align*}
This gives \eqref{eq:inverse-evaluation-generator}, including its
positive drift sign.
The uniform drift estimate gives
\[
 \|LF_{z,a}\|_\infty
 \leq\tfrac12\|a''\|_\infty+C_\kappa\|a'\|_\infty.
\]
Summing the products of the first derivatives gives
\[
 \sum_jD_jF_{z,a}D_jF_{w,b}
 =a'(q_z)b'(q_w)\sum_je_j(q_z)e_j(q_w)
 =a'(q_z)b'(q_w)K(q_z-q_w),
\]
which proves \eqref{eq:inverse-evaluation-cross-energy}.
\end{proof}

From this point, complex martingales are obtained by complexifying
their real and imaginary parts. Their brackets are complex bilinear;
conjugation is written explicitly in the second argument when used.

For each target label $z\in\T$, let $\theta^z$ be the continuous
angular lift of $q_z(X_t)$ with a chosen value at time zero. For the
boundary equations, write
\[
 b_t(\theta)=B(X_t,\theta),\qquad \|b_t\|_\infty\leq C_\kappa.
\]
This random field is jointly progressively measurable:
$(X_t)_{t\geq0}$ is continuous and adapted, while $(h,\theta)\mapsto B(h,\theta)$ is
measurable for the product sigma-algebra and has the common spatial modulus
from Proposition~\ref{prop:uniform-drift}.
\begin{corollary}\label{cor:inverse-particle-martingales}
Fix a finite or countable family of target labels
$z_1,z_2,\ldots\in\T$. Outside one properly exceptional set of starting
states, the corresponding lifts satisfy, simultaneously for all labels
$z,w$ in the family and all $t\geq0$,
\begin{align}
 \theta_t^z-\theta_0^z
 &=M_t^z+\int_0^t b_s(\theta_s^z)\,ds,
 \label{eq:inverse-particle-decomposition}\\
 \shortintertext{where $M^z$ is a continuous local martingale with}
 \langle M^z\rangle_t&=t,
 \qquad
 \langle M^z,M^w\rangle_t
 =2\int_0^t K(\theta_s^z-\theta_s^w)\,ds.
 \label{eq:inverse-particle-brackets}
\end{align}
Thus each $M^z$ is a standard Brownian motion in the filtration of the
associated process.
\end{corollary}

\begin{proof}
For each $j,k$, let $E_{j,k}$ be a properly exceptional set outside
which the Fukushima decompositions and their cross-bracket identities
hold for every pair in
\[
 \{F_{z_j,\cos},F_{z_j,\sin},F_{z_k,\cos},F_{z_k,\sin}\}.
\]
Then
\[
 N_{\rm eval}:=\bigcup_{j,k}E_{j,k},\qquad
 \CapOne(N_{\rm eval})\leq\sum_{j,k}\CapOne(E_{j,k})=0,
\]
and $N_{\rm eval}$ is properly exceptional. Fix an initial state outside $N_{\rm eval}$.

Set
\[
 Z_t^z=e^{i\theta_t^z},\qquad
 C^z=M^{F_{z,\cos}}+iM^{F_{z,\sin}}.
\]
Proposition~\ref{prop:inverse-evaluation-generator},
Corollary~\ref{cor:fukushima-main}, and
\eqref{eq:gen-cross-bracket} give
\begin{equation}\label{eq:circle-complex-semimartingale}
 dZ_t^z=dC_t^z+
 \left(-\frac12+ib_t(\theta_t^z)\right)Z_t^z\,dt
\end{equation}
and, with bilinear complex brackets,
\begin{align}
 d\langle C^z,C^w\rangle_t
 &=-2Z_t^zZ_t^wK(\theta_t^z-\theta_t^w)\,dt,\notag\\
 d\langle C^z,\overline{C^w}\rangle_t
 &=2Z_t^z\overline{Z_t^w}K(\theta_t^z-\theta_t^w)\,dt.
 \label{eq:circle-complex-brackets}
\end{align}

Define
\[
 M_t^z=-i\int_0^t\overline{Z_s^z}\,dC_s^z.
\]
Since $2K(0)=1$,
\[
 d\langle M^z,M^z\rangle_t
 =d\langle M^z,\overline{M^z}\rangle_t=dt,
\qquad
 d\langle\operatorname{Im}M^z\rangle_t=0.
\]
Since $\operatorname{Im}M_0^z=0$, it follows that
$\operatorname{Im}M^z\equiv0$. Thus $M^z$ is real, and
\[
 d\langle M^z,M^w\rangle_t
 =(-i)^2\overline{Z_t^zZ_t^w}\,
       d\langle C^z,C^w\rangle_t
 =2K(\theta_t^z-\theta_t^w)\,dt.
\]
In particular, $\langle M^z\rangle_t=t$, so $M^z$ is a standard
Brownian motion.

Put
\[
 \Theta_t^z=\theta_0^z+M_t^z+
                 \int_0^t b_s(\theta_s^z)\,ds.
\]
The definition of $M^z$ gives $dC_t^z=iZ_t^z\,dM_t^z$. Hence
\[
 \begin{aligned}
 d(e^{-i\Theta_t^z})
 &=-ie^{-i\Theta_t^z}\,dM_t^z
   +\left(-\tfrac12-ib_t(\theta_t^z)\right)
                         e^{-i\Theta_t^z}\,dt,\\
 d\langle e^{-i\Theta^z},Z^z\rangle_t
 &=e^{-i\Theta_t^z}Z_t^z\,dt.
 \end{aligned}
\]
Together with \eqref{eq:circle-complex-semimartingale}, the product
rule gives
\[
 d(e^{-i\Theta_t^z}Z_t^z)=0,\qquad
 e^{-i\Theta_0^z}Z_0^z=1.
\]
Thus $e^{i\Theta_t^z}=e^{i\theta_t^z}$. The continuous difference
$\Theta^z-\theta^z$ is $2\pi\mathbb Z$-valued and starts from zero;
hence $\Theta^z=\theta^z$. This proves
\eqref{eq:inverse-particle-decomposition} and
\eqref{eq:inverse-particle-brackets}.
\end{proof}

\subsection{Common Fourier noise}

The joint brackets determine a single cylindrical noise for all inverse
evaluations. We construct it in the diffusion filtration, so the same
noise can be used in the regularizations below.

Retain the notation $\theta^z$ for a continuous angular lift of
$X_t^{-1}(z)$. For any prescribed countable family,
Corollary~\ref{cor:inverse-particle-martingales} supplies the corresponding
decompositions outside one properly exceptional set. All statements below
use the stationary path law and the filtration $(\mathcal G_t)$.
Define the relative circle map
\[
 \Psi_t=X_t^{-1}\circ X_0,\qquad \Psi_0=\id.
\]
For simultaneous identities in the label, we use the jointly continuous
modification of the stochastic integral field constructed below.

\begin{lemma}\label{lem:dense-noise-reconstruction}
Fix distinct deterministic labels $(z_j)_{j\geq1}$ dense in $\T$.
There are independent standard $(\mathcal G_t)$-Brownian motions
$W^{n,c},W^{n,s}$ such that, for every $j$ and $T<\infty$,
\[
 M_t^{z_j}
 =\sqrt2\sum_{n\geq2,\,\tau\in\{c,s\}}
   \int_0^t e_{n,\tau}(\theta_s^{z_j})\,dW_s^{n,\tau},
\]
with convergence in $L^2(\mathbb P;C([0,T]))$.
\end{lemma}

\begin{proof}
At each time set
\[
 s_j(t)=\bigl(\sqrt2 e_{n,\tau}(\theta_t^{z_j})\bigr)_{n,\tau}
          \in\ell^2.
\]
In this proof inner products of the vectors $s_j$ are those of $\ell^2$.
Corollary~\ref{cor:inverse-particle-martingales} gives their Gram identity
\begin{equation}\label{eq:noise-Gram}
 d\langle M^{z_i},M^{z_j}\rangle_t=\langle s_i(t),s_j(t)\rangle\,dt.
\end{equation}

Because $X_t^{-1}$ is a homeomorphism, the points
$\theta_t^{z_1},\ldots,\theta_t^{z_m}$ are distinct modulo $2\pi$ for
every $t$, and the full set $(\theta_t^{z_j})_{j\geq1}$ remains dense
in the circle. Every finite set of rows is therefore linearly independent.
Indeed, if real coefficients $c_1,\ldots,c_m$ satisfy
$\sum_{j=1}^m c_js_j(t)=0$, set
$\mu=\sum_{j=1}^m c_j\delta_{\theta_t^{z_j}}$ and
$\widehat\mu(n)=\int e^{-in\theta}\,d\mu(\theta)$. Then
$\widehat\mu(n)=0$ for $|n|\geq2$. Define the finite signed measure
\[
 d\mu_{\rm ac}(\theta)=
 \left(\sum_{|n|\leq1}\widehat\mu(n)e^{in\theta}\right)
                                      \frac{d\theta}{2\pi}.
\]
Its Fourier coefficients agree with those of $\mu$ for every
$n\in\mathbb Z$. Hence uniqueness of Fourier coefficients for finite
signed measures gives
\[
 \mu=\mu_{\rm ac}.
\]
Here $\mu$ is atomic whereas $\mu_{\rm ac}$ is absolutely continuous, so
$\mu=0$. The distinctness of the points then gives
$c_1=\cdots=c_m=0$.
The span of all rows is dense in $\ell^2$ as well. If
$\xi=(\xi_{n,\tau})\in\ell^2$ is orthogonal to every row,
the WP field
\[
 v_\xi=\sum_{n\geq2,\,\tau\in\{c,s\}}
               \xi_{n,\tau}e_{n,\tau}
\]
satisfies
\[
 \|v_\xi\|_\infty\leq K(0)^{1/2}\|\xi\|_{\ell^2},\qquad
 v_\xi(\theta_t^{z_j})=0\quad(j\geq1).
\]
Its continuous version vanishes on a dense set, so
$v_\xi=0$ and $\xi=0$.

The row processes are continuous in $\ell^2$, since
\[
 \|s_j(t)-s_j(u)\|_{\ell^2}^2
 =4\{K(0)-K(\theta_t^{z_j}-\theta_u^{z_j})\}
 \longrightarrow0\qquad(u\to t).
\]
They are adapted and hence predictable. For $m\geq1$, let $G_m(t)$
denote the Gram matrix of the first $m$ rows,
\[
 G_m(t)=\bigl(\langle s_i(t),s_j(t)\rangle_{\ell^2}\bigr)_{i,j\leq m}.
\]
Finite independence and continuity give
\[
 \det G_m(t)>0,
 \qquad
 \inf_{0\leq t\leq T}\det G_m(t)>0
 \quad\text{almost surely}.
\]
Consequently, the predictable Gram--Schmidt processes
\begin{equation}\label{eq:predictable-Gram-Schmidt}
 r_j=s_j-\sum_{i<j}\langle s_j,u_i\rangle u_i,
 \qquad u_j=\frac{r_j}{\|r_j\|_{\ell^2}}
\end{equation}
satisfy
\[
 \det G_m(t)=\prod_{j=1}^m\|r_j(t)\|_{\ell^2}^2,
 \qquad
 \inf_{\substack{0\leq t\leq T\\j\leq m}}
 \|r_j(t)\|_{\ell^2}>0
 \quad\text{almost surely}.
\]
Orthogonalize the martingales in the same order as the rows. Inductively
define the residual martingale $N_j$ and its normalization $\beta_j$ by
\[
 dN_j=dM^{z_j}-\sum_{i<j}\langle s_j,u_i\rangle\,d\beta_i,
 \qquad
 d\beta_j=\frac{dN_j}{\|r_j\|_{\ell^2}}.
\]
Assume the construction has been completed up to $j-1$.  For $i<j$,
\begin{align*}
 d\langle M^{z_j},N_i\rangle_t
 &=\left(\langle s_j,s_i\rangle
   -\sum_{k<i}\langle s_i,u_k\rangle\langle s_j,u_k\rangle\right)dt\\
 &=\langle s_j,r_i\rangle\,dt
 =\|r_i\|_{\ell^2}\langle s_j,u_i\rangle\,dt,\\
 \shortintertext{and therefore}
 d\langle M^{z_j},\beta_i\rangle_t
 &=\langle s_j,u_i\rangle\,dt.
\end{align*}
Consequently,
\begin{align*}
 d\langle N_j,\beta_i\rangle_t&=0\qquad(i<j),\\
 \shortintertext{and}
 d\langle N_j\rangle_t
 &=\left(\|s_j\|_{\ell^2}^2
       -\sum_{i<j}|\langle s_j,u_i\rangle|^2\right)dt
 =\|r_j(t)\|_{\ell^2}^2\,dt.
\end{align*}
Thus the stochastic integrals are defined on every finite interval, and
\begin{equation}\label{eq:beta-brackets-representation}
 d\langle\beta_i,\beta_j\rangle_t=\delta_{ij}\,dt,
 \qquad
 dM^{z_j}=\sum_{i\leq j}\langle s_j,u_i\rangle\,d\beta_i.
\end{equation}
The finite-dimensional L\'evy characterization makes the $\beta_i$
independent Brownian motions in $(\mathcal G_t)$. Write $u_j^k$ for the
$k$th coordinate of $u_j$, where $k=(n,\tau)$. Since the rows are dense
in $\ell^2$ at every time, Gram--Schmidt also gives
\begin{equation}\label{eq:noise-ONB-completeness}
 \overline{\operatorname{span}}\{u_j(t):j\geq1\}=\ell^2,
 \qquad
 \sum_{j\geq1}u_j^k(t)u_j^l(t)=\delta_{kl}.
\end{equation}
Define
\[
 W_t^k=\sum_{j\geq1}\int_0^t u_j^k(s)\,d\beta_j(s).
\]
Completeness, Doob's inequality, and the martingale isometry give
\begin{align*}
 \mathbb E\sup_{t\leq T}
 \left|\sum_{j>J}\int_0^t u_j^k(s)\,d\beta_j(s)\right|^2
 &\leq4\mathbb E\int_0^T
       \left(1-\sum_{j\leq J}|u_j^k(s)|^2\right)ds
       \longrightarrow0\quad(J\to\infty),\\
 \shortintertext{while \eqref{eq:noise-ONB-completeness} gives}
 d\langle W^k,W^l\rangle_t
 &=\sum_j u_j^k(t)u_j^l(t)\,dt=\delta_{kl}\,dt.
\end{align*}
Thus the $W^k$ are independent standard Brownian motions on every
finite interval; uniqueness of the $L^2(\mathbb P;C([0,T]))$ limits on
overlaps gives consistent processes on $[0,\infty)$.  Let $\pi_N$
be the coordinate projection in $\ell^2$ onto $2\leq n\leq N$.
For a fixed row $s_j$, the definition of $W$ and stochastic isometry give
\[
 \sqrt2\sum_{\substack{2\leq n\leq N\\\tau}}
       \int_0^t e_{n,\tau}(\theta_s^{z_j})\,dW_s^{n,\tau}
 =\sum_{i\geq1}\int_0^t
       \langle\pi_Ns_j(s),u_i(s)\rangle\,d\beta_i(s).
\]
The series on the right converges in $L^2(\mathbb P;C([0,T]))$.
Since $s_j$ lies in the span of $u_1,\ldots,u_j$, Parseval's identity gives
\[
 \begin{aligned}
 &\mathbb E\sup_{t\leq T}
 \left|M_t^{z_j}
   -\sqrt2\sum_{\substack{2\leq n\leq N\\\tau}}
      \int_0^t e_{n,\tau}(\theta_s^{z_j})\,dW_s^{n,\tau}\right|^2\\
 &\qquad\leq4\mathbb E\int_0^T
                      \|(1-\pi_N)s_j(s)\|_{\ell^2}^2\,ds
 \leq8T\sum_{n>N}\alpha_n^2\longrightarrow0
       \quad(N\to\infty).
 \end{aligned}
\]
Together with the angular decomposition, this proves
\[
 M_t^{z_j}
 =\sqrt2\sum_{n\geq2,\,\tau\in\{c,s\}}
   \int_0^t e_{n,\tau}(\theta_s^{z_j})\,dW_s^{n,\tau}
\]
for every dense label.
\end{proof}

\begin{lemma}\label{lem:all-label-noise}
Let the Brownian motions be those of
Lemma~\ref{lem:dense-noise-reconstruction}. For each deterministic
label $z\in\T$, the partial sums
\[
 I_N^z(t)=\sqrt2\sum_{\substack{2\leq n\leq N\\\tau\in\{c,s\}}}
       \int_0^t e_{n,\tau}(\theta_s^z)\,dW_s^{n,\tau}
\]
converge in $L^2(\mathbb P;C([0,T]))$ for every $T<\infty$.
The limits have a version jointly continuous in $(t,z)$ and,
simultaneously for all deterministic labels,
\begin{equation}\label{eq:all-label-noise-identity}
 I_t^z=\theta_t^z-\theta_0^z
             -\int_0^t b_s(\theta_s^z)\,ds.
\end{equation}
If $Z$ is a fixed $\mathcal G_0$-measurable $\T$-valued label, define
$I_N^Z$ by the same formula, with the periodic integrands
$e_{n,\tau}(\theta_s^Z)=e_{n,\tau}(X_s^{-1}(Z))$. Then $I_N^Z$
converges in $L^2(\mathbb P;C([0,T]))$, and the jointly continuous
version evaluated at $Z$ agrees almost surely with this limit and
satisfies the same identity for any continuous angular lift of
$X_t^{-1}(Z)$.
\end{lemma}

\begin{proof}
For a deterministic label $z$ and $M>N$, Doob's inequality gives
\[
 \mathbb E\sup_{t\leq T}|I_M^z(t)-I_N^z(t)|^2
 \leq8T\sum_{n>N}\alpha_n^2\longrightarrow0
       \quad(N\to\infty).
\]
Thus $I_N^z\to I^z$ in $L^2(\mathbb P;C([0,T]))$. If
$z_\ell\to z$, then
\[
 \begin{aligned}
 \mathbb E\sup_{t\leq T}|I^{z_\ell}(t)-I^z(t)|^2
 &\leq8\mathbb E\int_0^T
   \{2K(0)-2K(\theta_s^{z_\ell}-\theta_s^z)\}\,ds\\
 &\longrightarrow0.
 \end{aligned}
\]
Indeed the integrand is bounded and tends pointwise to zero by
continuity of $X_s^{-1}$ and $K$. For the dense labels, the preceding
projection calculation gives
\[
 I^{z_j}=M^{z_j}\quad\text{in }L^2(\mathbb P;C([0,T])).
\]

By path lifting, choose a jointly continuous degree-one lift
\[
 \widetilde q:[0,T]\times\mathbb R\to\mathbb R,
 \qquad e^{i\widetilde q_t(u)}=X_t^{-1}(e^{iu}),
 \qquad \widetilde q_t(u+2\pi)=\widetilde q_t(u)+2\pi.
\]
For $z=e^{iu}$, set
\[
 \widehat M_t^z
 =\widetilde q_t(u)-\widetilde q_0(u)-\int_0^t b_s(\widetilde q_s(u))\,ds.
\]
This is independent of the representative $u$. Define
\[
 \omega_B(\delta)=
 \begin{cases}
  0,&\delta=0,\\
  C_\kappa\delta\log(e/\delta),&0<\delta\leq1,\\
  2C_\kappa,&\delta>1.
 \end{cases}
\]
For $0\leq s\leq t\leq T$,
\begin{align*}
 |\widehat M_t^{e^{iu}}-\widehat M_s^{e^{iv}}|
 &\leq|\widetilde q_t(u)-\widetilde q_s(v)|+|\widetilde q_0(u)-\widetilde q_0(v)|+C_\kappa(t-s)\\
 &\quad+\int_0^s
 \omega_B\!\left(d_\T(\widetilde q_\ell(u),\widetilde q_\ell(v))\right)\,d\ell.
\end{align*}
Hence $(t,z)\mapsto\widehat M_t^z$ is continuous on
$[0,T]\times\T$ almost surely. Put
\[
 \Omega_T^0=
 \{(t,z)\mapsto\widehat M_t^z\text{ is continuous on }[0,T]\times\T\}
 \cap\bigcap_{j\geq1}
 \{I^{z_j}=\widehat M^{z_j}\text{ in }C([0,T])\}.
\]
The angular decompositions and the preceding identities give
$\mathbb P(\Omega_T^0)=1$. If $z_{j_\ell}\to z$, then on
$\Omega_T^0$
\[
 \widehat M^{z_{j_\ell}}\longrightarrow\widehat M^z
       \quad\text{in }C([0,T]),
\]
whereas $I^{z_{j_\ell}}\to I^z$ in probability in the same space.
Thus $I^z=\widehat M^z$ almost surely for each deterministic $z$.
On $\Omega_T^0$ we use $\widehat M$ as the jointly continuous version
of the stochastic integral field; the identity then holds simultaneously
for all deterministic labels.

Let $Z$ be a fixed $\mathcal G_0$-measurable label.  The process
$e_{n,\tau}(\theta_t^Z)$ is continuous and adapted, hence predictable.
For $M>N$,
\[
 \mathbb E\sup_{t\leq T}|I_M^Z(t)-I_N^Z(t)|^2
 \leq8T\sum_{n>N}\alpha_n^2\longrightarrow0
       \qquad(N\to\infty),
\]
so $I_N^Z\to I^Z$ in $L^2(\mathbb P;C([0,T]))$.
Choose a finite $\mathcal G_0$-measurable partition
$(A_\ell)_{\ell\leq q_m}$ and dense labels $z_{j_\ell}$ such that the
finite-valued label $Z_m$ defined by
\[
 Z_m=z_{j_\ell}\quad\text{on }A_\ell,
 \qquad
 d_\T(Z_m,Z)\leq\frac1m
\]
satisfies $Z_m\to Z$ pointwise.  Since the
indicators $\mathbf 1_{A_\ell}$ are predictable,
\begin{equation}\label{eq:random-label-piecewise}
 I_N^{Z_m}=\sum_{\ell=1}^{q_m}\mathbf 1_{A_\ell}I_N^{z_{j_\ell}},
 \qquad
 \widehat M^{Z_m}
 =\sum_{\ell=1}^{q_m}\mathbf 1_{A_\ell}\widehat M^{z_{j_\ell}}.
\end{equation}
Thus $I^{Z_m}=\widehat M^{Z_m}$ almost surely. Moreover,
\[
 \begin{aligned}
 \mathbb E\sup_{t\leq T}|I^{Z_m}(t)-I^Z(t)|^2
 &\leq8\mathbb E\int_0^T
 \{2K(0)-2K(\theta_s^{Z_m}-\theta_s^Z)\}\,ds\\
 &\longrightarrow0\qquad(m\to\infty).
 \end{aligned}
\]
Joint continuity gives
$\widehat M^{Z_m}\to\widehat M^Z$ almost surely in $C([0,T])$.
Consequently,
\[
 I^Z=\widehat M^Z
 \quad\text{in }C([0,T])\quad\text{almost surely}.
\]
Taking a countable intersection over $T\in\mathbb N$ makes the jointly
continuous version and the stated identities consistent on
$[0,\infty)\times\T$.
\end{proof}

\begin{proposition}\label{prop:single-noise}
Let $W^{n,c},W^{n,s}$ and the jointly continuous field $I$ be those of
Lemmas~\ref{lem:dense-noise-reconstruction} and
\ref{lem:all-label-noise}. There is a probability-one event on which the
jointly continuous degree-one lift $(t,u)\mapsto\widetilde\Psi_t(u)$ of
$\Psi_t=X_t^{-1}\circ X_0$, normalized by $\widetilde\Psi_0=\id$,
satisfies, simultaneously for $t\geq0$ and $u\in\mathbb R$,
\begin{equation}\label{eq:relative-common-Fourier-SDE}
 \widetilde\Psi_t(u)
 =u+I_t^{X_0(e^{iu})}
       +\int_0^t b_s(\widetilde\Psi_s(u))\,ds.
\end{equation}
\end{proposition}

\begin{proof}
Intersect the probability-one event on which $(t,z)\mapsto I_t^z$ is jointly
continuous with the countably many random-label events corresponding
to $T\in\mathbb N$ and rational angles $u\in2\pi\mathbb Q$.
For
\[
 Z_u=X_0(e^{iu}),
\]
choose the angular lift with $\theta_0^{Z_u}=u$.  Then
\begin{align*}
 \theta_t^{Z_u}&=\widetilde\Psi_t(u),\\
 \shortintertext{and Lemma~\ref{lem:all-label-noise} gives}
 \widetilde\Psi_t(u)
 &=u+I_t^{Z_u}+\int_0^t b_s(\widetilde\Psi_s(u))\,ds
\end{align*}
simultaneously for rational $u$ and all $t\geq0$ on this event.
The map $u\mapsto X_0(e^{iu})$ is continuous, the field $I$ is jointly
continuous, and the two sides are $2\pi$-equivariant.  Density of
$2\pi\mathbb Q$ therefore extends the identity simultaneously to every
$u\in\mathbb R$, proving \eqref{eq:relative-common-Fourier-SDE}.
\end{proof}

\subsection{Smooth approximations}

Keep the adapted field $b_t$ fixed, and extend all circle fields
$2\pi$-periodically to $\mathbb R$. In pointwise formulas below,
$\Psi_t(\theta)$ and $\Psi_t^r(\theta)$ denote compatible degree-one
lifts; when the metric $d$ is used, the same symbols denote the induced
circle maps. Equation~\eqref{eq:relative-common-Fourier-SDE} shows that,
for every initial angle $\theta\in\mathbb R$, the lifted trajectory solves
the one-point equation
\begin{equation}\label{eq:fixed-environment-SDE}
 d\theta_t=\sqrt2\sum_{n\geq2,\,\tau\in\{c,s\}}
 e_{n,\tau}(\theta_t)\,dW_t^{n,\tau}+b_t(\theta_t)\,dt.
\end{equation}
Write
$\sigma(x)=(\sqrt2e_{n,\tau}(x))_{n\geq2,\,\tau\in\{c,s\}}$.
Pathwise uniqueness below compares
solutions using the same field $b_t$, the same Brownian motions, and
the same initial value. For $0<r<1$, let
$\Psi^r$ solve
\begin{equation}\label{eq:regularized-circle-flow}
 d\Psi_t^r(\theta)=\sqrt2\sum_{n\geq2,\,\tau\in\{c,s\}}r^n
       e_{n,\tau}(\Psi_t^r(\theta))\,dW_t^{n,\tau}
       +(P_rb_t)(\Psi_t^r(\theta))\,dt,
 \qquad \Psi_0^r=\id.
\end{equation}
Fix the increasing concave Osgood function
\begin{equation}\label{eq:Osgood-function}
 \omega(0)=0,\qquad
 \omega(u)=
 \begin{cases}
  u\log(e^2/u),&0<u\leq1,\\
  u+1,&u\geq1.
 \end{cases}
\end{equation}
Let $\delta=d_\T(x,y)$.  For $0<\delta\leq1$, the Fourier series for
$K$ and $1-\cos u\leq\min(2,u^2/2)$ give
\begin{align*}
 K(0)-K(x-y)
 &\leq C\delta^2\sum_{2\leq n\leq\delta^{-1}}\frac1n
      +C\sum_{n>\delta^{-1}}\frac1{n^3}\\
 &\leq C\delta^2\log(e/\delta),
\end{align*}
while boundedness of $K$ treats the remaining separations.  If
$|x-y|\leq1$, then $\delta=|x-y|$; if $|x-y|>1$, boundedness gives a
constant upper bound.  Thus, with $u=|x-y|^2$,
\begin{equation}\label{eq:noise-Osgood}
 \|\sigma(x)-\sigma(y)\|_{\ell^2}^2
 =4\{K(0)-K(x-y)\}\leq C\omega(u).
\end{equation}
For the drift, Proposition~\ref{prop:uniform-drift} gives, when
$|x-y|\leq1$,
\[
 2|x-y|\,|b_t(x)-b_t(y)|
 \leq C_\kappa |x-y|^2\log(e/|x-y|)
 \leq C_\kappa\omega(u),
\]
while for $|x-y|>1$ the uniform drift bound gives
$2|x-y|\,|b_t(x)-b_t(y)|\leq C_\kappa(u+1)$.  Consequently,
\begin{equation}\label{eq:drift-Osgood}
 2(x-y)\{b_t(x)-b_t(y)\}
 \leq C_\kappa\omega(|x-y|^2)
\end{equation}
uniformly in $t,x,y$.
Moreover,
\[
 \int_{0+}\frac{du}{\omega(u)}=\infty.
\]
If $\theta,\widetilde\theta$ solve
\eqref{eq:fixed-environment-SDE} with the same noise and initial value,
then It\^o's formula and Jensen's inequality give
\[
 m(t):=\mathbb E|\theta_t-\widetilde\theta_t|^2
 \leq C_\kappa\int_0^t\omega(m(s))\,ds.
\]
The Osgood criterion gives $m=0$, hence pathwise uniqueness. For
$r<1$ and $k\geq0$,
\begin{equation}\label{eq:regularized-smooth-coefficients}
 \sum_{n,\tau}\left\|
  \partial_x^k(\sqrt2r^ne_{n,\tau})\right\|_\infty^2
 \leq C_k\sum_{n\geq2}r^{2n}n^{2k}\alpha_n^2<\infty,
 \qquad
 \sup_t\|\partial_x^kP_rb_t\|_\infty\leq C_{r,k}C_\kappa.
\end{equation}
The smooth stochastic-flow theorem \cite{Kunita} therefore gives a
unique strong smooth flow of orientation-preserving circle
diffeomorphisms for \eqref{eq:regularized-circle-flow}.
\begin{lemma}\label{lem:one-point-Poisson-stability}
For $T<\infty$ and every initial angle $\theta\in\mathbb R$,
\begin{equation}\label{eq:one-point-Poisson-stability}
 \mathbb E\sup_{t\leq T}
 |\Psi_t^r(\theta)-\Psi_t(\theta)|^2
 \longrightarrow0\qquad(r\uparrow1).
\end{equation}
\end{lemma}

\begin{proof}
For the regularizations put $\sigma_r(x)=(\sqrt2r^ne_{n,\tau}(x))_{n,\tau}$
and define the deterministic error
\[
 \varepsilon_r=
 \sup_{h\in X}\|P_rB(h,\cdot)-B(h,\cdot)\|_\infty^2
       +2\sum_{n\geq2}(1-r^n)^2\alpha_n^2.
\]
Let $p_r$ be the circular Poisson kernel. For every
$0<\delta<1$,
\[
 \begin{aligned}
 \sup_{h\in X}\|P_rB(h,\cdot)-B(h,\cdot)\|_\infty
 &\leq C_\kappa\delta\log(e/\delta)\\
 &\quad+2C_\kappa
 \int_{d_{\T}(s,0)>\delta}p_r(s)\,\frac{ds}{2\pi}.
 \end{aligned}
\]
For each fixed $\delta>0$,
\[
 \int_{d_\T(s,0)>\delta}p_r(s)\,\frac{ds}{2\pi}
 \longrightarrow0\qquad(r\uparrow1),
\]
so
\begin{align*}
 \limsup_{r\uparrow1}
 \sup_{h\in X}\|P_rB(h,\cdot)-B(h,\cdot)\|_\infty
 &\leq C_\kappa\delta\log(e/\delta),\\
 \shortintertext{and then, letting $\delta\downarrow0$,}
 \sup_{h\in X}\|P_rB(h,\cdot)-B(h,\cdot)\|_\infty
 &\longrightarrow0\qquad(r\uparrow1).
\end{align*}
Moreover, dominated convergence gives
\[
 \sum_{n\geq2}(1-r^n)^2\alpha_n^2\longrightarrow0,
 \qquad r\uparrow1.
\]
Thus $\varepsilon_r\to0$ as $r\uparrow1$.
Let $\theta^r$ and $\theta$ start at the same point and use the same Brownian
motions. Put
\[
 \Delta_t=\theta_t^r-\theta_t,\qquad
 \delta\sigma_t=\sigma_r(\theta_t^r)-\sigma(\theta_t),\qquad
 \delta b_t=(P_rb_t)(\theta_t^r)-b_t(\theta_t).
\]
Equations~\eqref{eq:noise-Osgood}--\eqref{eq:drift-Osgood} and
Young's inequality give, for $x,y\in\mathbb R$,
\begin{align}
 \|\sigma_r(x)-\sigma(y)\|_{\ell^2}^2
 &\leq C\omega(|x-y|^2)+C\varepsilon_r,\notag\\
 2(x-y)\{(P_rb_t)(x)-b_t(y)\}
 +\|\sigma_r(x)-\sigma(y)\|_{\ell^2}^2
 &\leq C_\kappa\omega(|x-y|^2)+C_\kappa\varepsilon_r.
 \label{eq:regularized-coefficient-Osgood}
\end{align}
Apply It\^o's formula up to
$\tau_R=\inf\{t:|\Delta_t|\geq R\}$. Boundedness of the coefficients
permits $R\to\infty$ and gives
\begin{equation}\label{eq:regularized-difference-Ito}
 |\Delta_t|^2
 =2\int_0^t\Delta_s
       \langle\delta\sigma_s,dW_s\rangle_{\ell^2}
  +\int_0^t\{2\Delta_s\delta b_s
                 +\|\delta\sigma_s\|_{\ell^2}^2\}\,ds.
\end{equation}
For fixed $T<\infty$, expectation, concavity of $\omega$, and
\eqref{eq:regularized-coefficient-Osgood} yield
\begin{equation}\label{eq:regularized-stability}
 m_r(t):=\mathbb E|\theta_t^r-\theta_t|^2
 \leq C_\kappa\int_0^t\omega(m_r(s))\,ds+C_\kappa T\varepsilon_r,
 \qquad 0\leq t\leq T.
\end{equation}
Put $c_r=C_\kappa T\varepsilon_r$, and let
\[
 y_r(t)=c_r+C_\kappa\int_0^t\omega(y_r(s))\,ds.
\]
Comparison in \eqref{eq:regularized-stability} gives
\[
 m_r(t)\leq y_r(t),\qquad
 \int_{c_r}^{y_r(t)}\frac{du}{\omega(u)}=C_\kappa t.
\]
Since $\omega(u)=u\log(e^2/u)$ on $(0,1]$, for $r$ sufficiently close
to one,
\[
 \sup_{t\leq T}m_r(t)
 \leq C_{\kappa,T}\varepsilon_r^{\exp(-C_\kappa T)}
 \longrightarrow0\qquad(r\uparrow1).
\]

To pass to the time supremum, apply BDG and Young's inequality to the
martingale in \eqref{eq:regularized-difference-Ito}:
\[
 \begin{aligned}
 &\mathbb E\sup_{t\leq T}
   \left|2\int_0^t\Delta_s
                   \langle\delta\sigma_s,dW_s\rangle_{\ell^2}\right|\\
 &\qquad\leq C\mathbb E
       \left(\sup_{t\leq T}|\Delta_t|^2
                    \int_0^T\|\delta\sigma_s\|_{\ell^2}^2\,ds\right)^{1/2}\\
 &\qquad\leq\tfrac12\mathbb E\sup_{t\leq T}|\Delta_t|^2
       +C\int_0^T\{\omega(m_r(s))+\varepsilon_r\}\,ds.
 \end{aligned}
\]
Together with \eqref{eq:regularized-coefficient-Osgood}, this gives
\[
 \mathbb E\sup_{t\leq T}|\theta_t^r-\theta_t|^2
 \leq C_{\kappa,T}\left(\varepsilon_r+
                     \int_0^T\omega(m_r(s))\,ds\right)
 \longrightarrow0\qquad(r\uparrow1).
\]
\end{proof}

\begin{lemma}\label{lem:Poisson-stability}
For every $T<\infty$,
\begin{equation}\label{eq:Poisson-flow-convergence}
 \sup_{0\leq t\leq T}d(\Psi_t^r,\Psi_t)
 \longrightarrow0
 \quad\text{in probability as }r\uparrow1.
\end{equation}
\end{lemma}

\begin{proof}
For a grid of mesh $2\pi/m$, monotonicity of compatible lifts gives
\[
 \sup_{t\leq T,\theta}|\Psi_t^r(\theta)-\Psi_t(\theta)|
 \leq\max_{0\leq j<m}\sup_{t\leq T}
 |\Psi_t^r(2\pi j/m)-\Psi_t(2\pi j/m)|
       +\omega_{\Psi,T}(2\pi/m),
\]
where, for the compatible real lift used in this estimate,
\[
 \omega_{\Psi,T}(\delta)
 =\sup_{0\leq t\leq T}
   \sup_{\substack{x,y\in\mathbb R\\|x-y|\leq\delta}}
      |\Psi_t(x)-\Psi_t(y)|.
\]
Degree-one equivariance and joint continuity give
$\omega_{\Psi,T}(\delta)\to0$ almost surely as $\delta\downarrow0$:
for $\delta\leq1$ the supremum may be restricted to
$x\in[0,2\pi]$ and $y\in[-1,2\pi+1]$.
For fixed $m$, Lemma~\ref{lem:one-point-Poisson-stability} makes the
finite maximum tend to zero in probability as $r\uparrow1$. Hence, for
every $\varepsilon>0$,
\[
 \limsup_{r\uparrow1}\mathbb P\left(
      \sup_{t\leq T,\theta}|\Psi_t^r(\theta)-\Psi_t(\theta)|>\varepsilon\right)
 \leq\mathbb P\left(\omega_{\Psi,T}(2\pi/m)>\varepsilon/2\right)
 \longrightarrow0\quad(m\to\infty).
\]
Put
\[
 \begin{gathered}
 d_r(T)=\sup_{t\leq T}\sup_{\theta\in\T}
       d_\T(\Psi_t^r(\theta),\Psi_t(\theta)),\\
 \omega_{\Psi^{-1},T}(\delta)
 =\sup_{t\leq T}\sup_{d_\T(x,y)\leq\delta}
       d_\T(\Psi_t^{-1}(x),\Psi_t^{-1}(y)).
 \end{gathered}
\]
The inverse family is jointly continuous. Indeed, if
$(t_n,y_n)\to(t,y)$ and $x_n=\Psi_{t_n}^{-1}(y_n)$, every subsequence
with $x_{n_k}\to x$ satisfies
\[
 \Psi_t(x)=\lim_{k\to\infty}\Psi_{t_{n_k}}(x_{n_k})=y,
 \qquad x=\Psi_t^{-1}(y).
\]
Compactness of $\T$ therefore gives
\[
 \omega_{\Psi^{-1},T}(\delta)\downarrow0
 \quad\text{almost surely as }\delta\downarrow0.
\]
Then
\[
 \sup_{t\leq T}\sup_{y\in\T}
 d_\T((\Psi_t^r)^{-1}(y),\Psi_t^{-1}(y))
 \leq\omega_{\Psi^{-1},T}(d_r(T)).
\]
For $\delta,\eta>0$,
\[
 \begin{aligned}
 &\mathbb P\left(
 \sup_{t\leq T}\sup_{y\in\T}
 d_\T((\Psi_t^r)^{-1}(y),\Psi_t^{-1}(y))>\eta
 \right)\\
 &\qquad\leq\mathbb P(d_r(T)>\delta)
 +\mathbb P\bigl(\omega_{\Psi^{-1},T}(\delta)>\eta\bigr).
 \end{aligned}
\]
Consequently,
\[
 \limsup_{r\uparrow1}\mathbb P\left(
 \sup_{t\leq T}\sup_{y\in\T}
 d_\T((\Psi_t^r)^{-1}(y),\Psi_t^{-1}(y))>\eta
 \right)
 \leq\mathbb P\bigl(\omega_{\Psi^{-1},T}(\delta)>\eta\bigr)
 \longrightarrow0\quad(\delta\downarrow0).
\]
Since $|e^{ix}-e^{iy}|\leq d_\T(x,y)$, the forward and inverse
angular estimates together prove convergence in probability in the
metric $d$.
\end{proof}

\section{Planar motion and two-point estimates}
\label{sec:welding-support}

\subsection{The normalized welding space}

Let
\[
 \mathbb D=\{z\in\mathbb C:|z|<1\},\qquad
 \mathbb D^*=\{z:|z|>1\}\cup\{\infty\}.
\]
Boundary distances $|\zeta-\xi|$ are chord distances;
angular distances are still denoted by $d_\T$.
Consider triples $(h,f,g)$ with $h\in X$, where
$f:\overline{\mathbb D}\to\mathbb C$ and
$g:\overline{\mathbb D^*}\to\widehat{\mathbb C}$ are continuous and
injective, and conformal in the interiors. We use the exterior
normalization
\[
 g(z)=z+u(1/z),
\]
where $u$ is holomorphic in $\mathbb D$, continuous on
$\overline{\mathbb D}$, and $u(0)=0$.  The geometric state space is
then
\begin{equation}\label{eq:Jordan-space}
 Y=\{(h,f,g):h\in X,\ f=g\circ h\text{ on }\T\}.
\end{equation}
Here $f$ and $g$ are always understood with the regularity and
normalization just specified.  We give $Y$ the topology of $d$ together
with uniform convergence of $f$ and $u$; the normalization at infinity
removes the common affine freedom.

Let $\prX(h,f,g)=h$, and retain the removable welding locus
$W_{\rm rem}$ and the full-measure set $X_{\rm nf}$ from
Section~\ref{sec:stress-drift}. To fix the normalization of the lift,
let $(f_h^\circ,g_h^\circ)$ be the intrinsic pair supplied by the
normalized-frame construction in \cite{Response}, and write
\[
 g_h^\circ(z)=a_\infty(h)z+b_\infty(h)+O(1/z),
 \qquad a_\infty(h)\ne0.
\]
For $h\in X_{\rm nf}$, set
\[
 f_h=\frac{f_h^\circ-b_\infty(h)}{a_\infty(h)},\qquad
 g_h=\frac{g_h^\circ-b_\infty(h)}{a_\infty(h)}.
\]
This common affine change preserves $f_h=g_h\circ h$ on $\T$ and
gives $g_h(z)=z+O(1/z)$. Define $\iota(h)=(h,f_h,g_h)$ for $h\in X_{\rm nf}$ and extend $\iota$
measurably across $X\setminus X_{\rm nf}$, as in the Introduction.  Then
$\prX\circ\iota=\id$ on $X_{\rm nf}$.

Set $\nu^Y=\iota_*\nu$, and define $UF=F\circ\prX$ for
$F\in L^2(\nu)$. The projection identity on $X_{\rm nf}$ gives
$(\prX)_*\nu^Y=\nu$, so $U$ is well defined and injective on
$L^2$ equivalence classes. Put $\Dcal^Y=U\Dcal$ and define
\begin{equation}\label{eq:Jordan-transferred-form}
 \E^Y(UF,UG)=\E(F,G),\qquad F,G\in\Dcal.
\end{equation}
Write $\|\cdot\|_{\E^Y_1}$ for the corresponding form norm.

\begin{lemma}\label{lem:Polish-Jordan-triples}
The specified topology on $Y$ admits a complete separable compatible
metric. The lift $\iota:X\to Y$ is measurable, and
$U:L^2(\nu)\to L^2(\nu^Y)$ is unitary.
\end{lemma}

\begin{proof}
For a compact metric space $(S,d_S)$ and a metric space $(M,d_M)$,
the continuous injective maps from $S$ to $M$ form the set
\[
 \bigcap_{m\geq1}\bigcup_{\ell\geq1}
 \left\{\phi\in C(S,M):
   \inf_{d_S(x,y)\geq1/m}d_M(\phi(x),\phi(y))>\frac1\ell\right\}.
\]
Every set in braces is open in the uniform topology, so this set
is a $G_\delta$. Use $(h,f,u)$ as coordinates in the product of $X$
and two disk algebras, subject to $u(0)=0$. This ambient space admits
a complete separable metric. The assignment $u\mapsto g$ is continuous
for uniform spherical convergence on $\overline{\mathbb D^*}$.
Applying the displayed criterion to $f$ and $g$ therefore makes both
injectivity conditions $G_\delta$ conditions in these coordinates.
To check the matching relation, put
\[
 \omega_g(\delta)=
 \sup\{|g(z)-g(w)|:z,w\in\T,\ |z-w|\leq\delta\}.
\]
If $(h_j,f_j,g_j)\to(h,f,g)$ and $f_j=g_j\circ h_j$ on $\T$, then
\[
 \|f-g\circ h\|_{C(\T)}
 \leq\|f-f_j\|_{C(\T)}+\|g_j-g\|_{C(\T)}
                  +\omega_g(\|h_j-h\|_\infty)\longrightarrow0.
\]
Thus the matching relation is closed, and $Y$ is a $G_\delta$ subset
of the ambient space. It therefore admits a complete separable
compatible metric. Injective holomorphic maps are conformal, and the
argument principle identifies their images as the two Jordan domains.

Measurability of $\iota$ follows from the measurable conformal pair in
\cite[Section~3.1]{Response}, the measurable affine normalization,
and radial boundary limits: countably many evaluations generate the
Borel sigma-algebra of the disk algebra.
Removability makes the normalized pair unique. For $F\in L^2(\nu)$,
\[
 \|UF\|_{L^2(\nu^Y)}^2
 =\int_{X_{\rm nf}}|F(\prX(\iota(h)))|^2\,d\nu(h)=\|F\|_2^2.
\]
For $G\in L^2(\nu^Y)$, choose a measurable representative, define
$F=G\circ\iota$ on $X_{\rm nf}$, and extend $F$ measurably outside $X_{\rm nf}$.
Then $UF=G$ almost everywhere.
Thus $U$ is onto and unitary.
\end{proof}

\subsection{The regularized planar motion}

The drift also has a bounded infinitesimal Beltrami extension. This is
the estimate needed for planar welding, beyond the circle modulus.
Let $E_v$ denote the linearized Douady--Earle extension of a real
circle field $v$, with boundary values
$E_v(e^{i\theta})=ie^{i\theta}v(\theta)$. For a field with Fourier expansion
\[
 v(\theta)=\sum_{n\geq2}(x_n\cos n\theta+y_n\sin n\theta),
\]
its infinitesimal Beltrami coefficient is
\begin{equation}\label{eq:linear-DE-Beltrami}
 \bar\partial E_v(z)
 =\frac{i}{4}(1-|z|^2)^2
       \sum_{n\geq2}(n^3-n)(x_n+i y_n)\bar z^{n-2}.
\end{equation}
The normalization is that of \cite[Section~2 and (6.1)]{AMT}.
For $h\in X_{\rm nf}$, use $f_h^\circ$ and its regular stress
$Q_{f_h^\circ}$ from \eqref{eq:regular-stress}.

\begin{lemma}\label{lem:drift-Beltrami}
For $h\in X_{\rm nf}$ and $z\in\mathbb D$,
\begin{equation}\label{eq:drift-Beltrami-stress}
 \bar\partial E_{B(h,\cdot)}(z)
 =-\frac12(1-|z|^2)^2\overline{Q_{f_h^\circ}(z)}.
\end{equation}
The drift satisfies
\[
 \sup_{h\in X}\|\bar\partial E_{B(h,\cdot)}\|_\infty
 \leq C_\kappa,
\]
and, uniformly in its Poisson regularizations,
\[
 \sup_{\substack{h\in X\\0<r<1}}
 \|\bar\partial E_{P_rB(h,\cdot)}\|_\infty\leq C_\kappa.
\]
\end{lemma}

\begin{proof}
Write $t_n=t_n(f_h^\circ)$ for the stress coefficients from
\eqref{eq:analytic-drift-primitive}. The Fourier coefficients of
$B(h,\cdot)$ are
\[
 x_n=\frac{2\operatorname{Im}t_n}{n^3-n},\qquad
 y_n=\frac{2\operatorname{Re}t_n}{n^3-n}.
\]
Hence
\[
 (n^3-n)(x_n+iy_n)=2i\overline{t_n},
\]
and substitution into \eqref{eq:linear-DE-Beltrami} proves the identity.
Writing $f=f_h^\circ$, Poisson regularization gives
\[
 \left|\bar\partial E_{P_rB(h,\cdot)}(z)\right|
 =\frac12r^2(1-|z|^2)^2|Q_f(rz)|
 \leq C_\kappa
       \left(\frac{r(1-|z|^2)}{1-r^2|z|^2}\right)^2
 \leq C_\kappa.
\]
Here the first inequality is \eqref{eq:uniform-Bers-stress}; the second
uses $r(1-|z|^2)\leq1-r^2|z|^2$. The same estimate at $r=1$
bounds the unregularized extension. On $X\setminus X_{\rm nf}$, the chosen
drift is zero.
\end{proof}

For a continuous circle function $v$, define its Zygmund norm by
\[
 \|v\|_{\Lambda_1}:=\|v\|_\infty+
 \sup_{0<|s|\leq1}|s|^{-1}
 \|v(\cdot+s)+v(\cdot-s)-2v\|_\infty.
\]
Work with the stationary response diffusion and the filtration
$(\mathcal G_t)$ of the common-noise construction. Let $(f_0,g_0)$
be the normalized conformal pair associated with $X_0$, and put
$h_0=X_0=g_0^{-1}\circ f_0$. For the adapted field $b_t$ in
\eqref{eq:regularized-circle-flow},
\[
 \sup_t\|b_t\|_{\Lambda_1}\leq C_\kappa,
\]
and, by Proposition~\ref{prop:uniform-drift} and
Lemma~\ref{lem:drift-Beltrami},
\[
 \sup_t\|\bar\partial E_{b_t}\|_\infty\leq C_\kappa.
\]

The It\^o and Stratonovich circle drifts agree because, at every
frequency,
\begin{equation}\label{eq:circle-Ito-correction}
 e_{n,c}e_{n,c}'+e_{n,s}e_{n,s}'=0.
\end{equation}
Set
\[
 h_t^r=h_0\circ(\Psi_t^r)^{-1},
\]
and denote its normalized conformal pair by $(f_t^r,g_t^r)$. For fixed
$r<1$ the circle flow is smooth
and hence quasisymmetric. The right-composition stability
\cite[Lemma~3.1]{Response} gives
\[
 h_0\in{W_{\rm rem}}
 \quad\Longrightarrow\quad
 h_t^r\in{W_{\rm rem}}
 \qquad(t\geq0).
\]
Thus the pair is uniquely defined. Extending the source
diffeomorphism to the disk and solving the normalized Beltrami equation
gives continuous paths for fixed $r$.

Let
\[
 \T_{\mathbb Q}=\{e^{2\pi iq}:q\in\mathbb Q\},
 \qquad
 \gamma_0=g_0(\T).
\]
The countable set $\gamma_{0,\mathbb Q}=g_0(\T_{\mathbb Q})$
is dense in $\gamma_0$. For $x=g_0(\zeta)$, set
\begin{equation}\label{eq:planar-boundary-motion}
 Z_t^r(x)=g_t^r(\zeta),\qquad |\zeta|=1.
\end{equation}
This parametrizes the deformed curve by its initial physical boundary.
For a continuous complex local martingale $N$ starting from zero and
a bounded stopping time $\sigma$, write
\[
 \|N\|_{\mathcal H^2[0,\sigma]}
 :=\bigl(\mathbb E\langle N,\overline N\rangle_\sigma\bigr)^{1/2}.
\]

For a normalized removable pair $(f,g)$, write $h=g^{-1}\circ f$ on
$\T$. If $a\notin g(\mathbb D^*)$, then
\[
 \frac1{g(1/z)-a}=z+az^2+O(z^3)
\]
is normalized univalent on $\mathbb D$. Hence $|a|\leq2$, and the
bounded complementary component of every normalized pair lies in
$2\overline{\mathbb D}$. Fix $A>4e$ once and for all; every boundary separation satisfies
$0\leq\delta\leq4$, hence $\log(A/\delta)\geq1$ for $\delta>0$.
The functions $\delta\log(A/\delta)$ and
$\delta^2\log(A/\delta)$ are assigned their continuous value $0$
at $\delta=0$.

For a smooth real circle field $v$, set
\[
 \mu_v=\left(\frac{f'}{\overline{f'}}\bar\partial E_v\right)
        \circ f^{-1}\mathbf 1_{f(\mathbb D)}.
\]
For compactly supported $\mu\in L^p(\mathbb C)$ with $p>2$,
denote its Cauchy transform by
\[
 \operatorname{Cauchy}\mu(z)
 =\frac1\pi\int_{\mathbb C}\frac{\mu(w)}{z-w}\,dA(w),
\]
where $dA$ is planar area.
For a smooth real circle field $v$, $\partial_vg$ denotes the
$t=0$ derivative of the normalized exterior map associated with
$h\circ\psi_{-t}^v$; these source-chart derivatives use the opposite
flow sign to \eqref{eq:actual-right-derivative}.
For real $\varphi\in H^4(\T)$ with $\|\varphi'\|_\infty<1/2$, set
\[
 \psi_\varphi(e^{i\theta})=e^{i(\theta+\varphi(\theta))},\qquad
 h_\varphi=h\circ\psi_\varphi^{-1}.
\]
Right-composition stability \cite[Lemma~3.1]{Response} preserves
removability, so
$(f_\varphi,g_\varphi)$ denotes the uniquely normalized conformal pair
of $h_\varphi$. We write $d^j$ for Fr\'echet differentials in the
$\varphi$ coordinate.

\begin{lemma}\label{lem:boundary-evaluation-chart}
There exist $\varepsilon,C>0$, independent of the normalized removable
pair, such that $\varphi\mapsto g_\varphi|_{\T}$ is a $C^2$ map from
the real $H^4(\T)$ ball of radius $\varepsilon$ about zero into $C(\T)$.
For $j=1,2$, real directions $\eta_1,\ldots,\eta_j\in H^4(\T)$,
and $\|\varphi\|_{H^4}<\varepsilon$,
\[
 \bigl\|d^j(g_\varphi|_{\T})[\eta_1,\ldots,\eta_j]\bigr\|_{C(\T)}
 \leq C\prod_{k=1}^j\|\eta_k\|_{H^4}.
\]
At the base pair, for every smooth real circle field $v$ and
$\zeta\in\T$,
\[
 \partial_vg(\zeta)=\operatorname{Cauchy}\mu_v(g(\zeta)).
\]
\end{lemma}

\begin{proof}
Choose a smooth function $\chi_{\rm rad}$ that is zero on $[0,1/2]$ and one on
$[3/4,1]$, and define
\[
 \Lambda_\varphi(re^{i\theta})
 =r e^{i(\theta+\chi_{\rm rad}(r)\varphi(\theta))},
\]
for $\varphi$ in a neighborhood of zero in $H^4(\T)$. If
$\|\varphi'\|_\infty<1/2$, then $\Lambda_\varphi$ is a disk
diffeomorphism. Choose $\varepsilon>0$ so that
$\|\varphi\|_{H^4}<\varepsilon$ implies this inequality, and put
\[
 \lambda_\varphi=
 \left(\frac{f'}{\overline{f'}}
 \frac{\bar\partial\Lambda_\varphi}
      {\partial\Lambda_\varphi}\right)
 \circ f^{-1}\mathbf 1_{f(\mathbb D)}.
\]
The embedding $H^4(\T)\hookrightarrow C^3(\T)$ and differentiation of
the displayed quotient give, for $j=1,2$,
\[
 \bigl\|d^j\lambda_\varphi[\eta_1,\ldots,\eta_j]\bigr\|_\infty
 \leq C\prod_{k=1}^j\|\eta_k\|_{H^4}.
\]
After decreasing $\varepsilon$ if necessary, we also have
\[
 \sup_{\|\varphi\|_{H^4}<\varepsilon}
 \|\lambda_\varphi\|_\infty\leq q<1.
\]
The bounded component of every normalized pair lies in
$2\overline{\mathbb D}$ by the fixed-domain estimate above. Hence
$\lambda_\varphi$ and its first two differentials are supported in
this fixed disk. Consequently, for every $p<\infty$,
\[
 \|\lambda_\varphi\|_{L^p}\leq C_p.
\]
The same fixed-support argument gives, for $j=1,2$,
\[
 \bigl\|d^j\lambda_\varphi[\eta_1,\ldots,\eta_j]\bigr\|_{L^p}
 \leq C_p\prod_{k=1}^j\|\eta_k\|_{H^4}.
\]
Thus the map $\varphi\mapsto\lambda_\varphi$ is $C^2$ from the
$H^4$-ball above into both $L^\infty(\mathbb C)$ and $L^p(\mathbb C)$.
The Beurling transform $\mathcal B=\partial\operatorname{Cauchy}$ is an
$L^2$ isometry and its $L^p$ norm is continuous at $p=2$. Choose
$p>2$ so that $q\|\mathcal B\|_{L^p\to L^p}<1$, and put
\[
 a_\mu=(I-\mu\mathcal B)^{-1}\mu.
\]
The corresponding principal solution is
\[
 W_\mu(z)=z+\operatorname{Cauchy}a_\mu(z).
\]
The maps $W_{\lambda_\varphi}\circ f$ and $\Lambda_\varphi$ have the
same Beltrami coefficient. Hence the normalized conformal pair
welding $h\circ\psi_\varphi^{-1}$ is
\[
 f_\varphi=W_{\lambda_\varphi}\circ f\circ\Lambda_\varphi^{-1},
 \qquad g_\varphi=W_{\lambda_\varphi}\circ g.
\]
For the resolvent differentials, let $\mu_1,\mu_2$ be bounded
variations of the Beltrami coefficient. They and $\mu$ are supported
in $2\overline{\mathbb D}$, and $\|\mu\|_\infty\leq q$.
Differentiating $(I-\mu\mathcal B)a_\mu=\mu$ gives
\begin{align*}
 da_\mu[\mu_1]
 &=(I-\mu\mathcal B)^{-1}\{\mu_1(1+\mathcal B a_\mu)\},
 \\
 \shortintertext{and, for the second differential,}
 d^2a_\mu[\mu_1,\mu_2]
 &=(I-\mu\mathcal B)^{-1}
   \{\mu_1\mathcal B\,da_\mu[\mu_2]
     +\mu_2\mathcal B\,da_\mu[\mu_1]\}.
\end{align*}
Consequently,
\begin{align*}
 \|da_\mu[\mu_1]\|_{L^p(\mathbb C)}
 &\leq C_{p,q}\|\mu_1\|_\infty,
 \\
 \shortintertext{and}
 \|d^2a_\mu[\mu_1,\mu_2]\|_{L^p(\mathbb C)}
 &\leq C_{p,q}\|\mu_1\|_\infty\|\mu_2\|_\infty.
\end{align*}
The identity
$a_\mu=\mu+\mu\mathcal B a_\mu$ and its first two differentials show
that $a_{\lambda_\varphi}$ and these differentials are supported in
$2\overline{\mathbb D}$. For $p'=p/(p-1)<2$ and every $a$ supported in this disk,
H\"older's inequality gives
\[
 \sup_{|z|\leq2}|\operatorname{Cauchy}a(z)|
 \leq \frac{\|a\|_{L^p(\mathbb C)}}{\pi}
   \sup_{|z|\leq2}\left(\int_{|w|\leq2}|z-w|^{-p'}\,dA(w)\right)^{1/p'}
 \leq C_p\|a\|_{L^p(\mathbb C)}.
\]
For the composition $\varphi\mapsto a_{\lambda_\varphi}$, the chain rule is
\begin{align*}
 d(a_{\lambda_\varphi})[\eta]
 &=da_{\lambda_\varphi}[d\lambda_\varphi[\eta]],
 \\
 \shortintertext{and}
 d^2(a_{\lambda_\varphi})[\eta_1,\eta_2]
 &=d^2a_{\lambda_\varphi}
       [d\lambda_\varphi[\eta_1],d\lambda_\varphi[\eta_2]]
   +da_{\lambda_\varphi}[d^2\lambda_\varphi[\eta_1,\eta_2]].
\end{align*}
Since $g(\T)\subset2\overline{\mathbb D}$ and
$g_\varphi=W_{\lambda_\varphi}\circ g$, these estimates give
\[
 \bigl\|d^j(g_\varphi|_{\T})[\eta_1,\ldots,\eta_j]\bigr\|_{C(\T)}
 \leq C\prod_{k=1}^j\|\eta_k\|_{H^4},
 \qquad \|\varphi\|_{H^4}<\varepsilon,\quad j=1,2.
\]
The constant depends only on $p,q$ and the fixed disks, and is therefore
uniform over the normalized removable pairs.
Conformal removability gives uniqueness in the chosen normalization.
Let $\widehat\Lambda_t$ be the Douady--Earle extension of the source
flow of $v$, and let $\widehat\lambda_t$ be its Beltrami coefficient
after transport by $f$. Formula~\eqref{eq:linear-DE-Beltrami} gives
\[
 \left\|\frac{\widehat\lambda_t}{t}-\mu_v\right\|_{L^\infty(\mathbb C)}
 \longrightarrow0,
 \qquad t\to0.
\]
Removability identifies this deformation with the preceding normalized
pair. At $\mu=0$, the formula for $W_\mu$ gives
\[
 \left.\frac{d}{dt}\right|_{t=0}W_{t\mu_v}(z)
 =\operatorname{Cauchy}\mu_v(z).
\]
Evaluating at $z=g(\zeta)$ therefore gives
\[
 \partial_vg(\zeta)
 =\operatorname{Cauchy}\mu_v(g(\zeta)).
\]
\end{proof}

Put $\gamma_t^r=Z_t^r(\gamma_0)$ and write
$\partial_{n,\tau}=\partial_{e_{n,\tau}}$, with $n\geq2$ and
$\tau\in\{c,s\}$. For the next lemma fix $0<r<1$; for labels
$\zeta,\xi\in\T$, write
$x=g_0(\zeta)$ and $y=g_0(\xi)$.  In the Cauchy variation formula below, $\mu_{P_rb_t}$ is formed
with the current interior map $f_t^r$, not with the initial pair.
Define the planar drift by
\[
 V_t^r(z)=\operatorname{Cauchy}\mu_{P_rb_t}(z).
\]

\begin{lemma}\label{lem:fixed-r-planar-characteristics}
For any $\mathcal G_0$-measurable labels $\zeta,\xi$, the corresponding
boundary motions admit continuous complex local martingales
$M^r(x),M^r(y)$, starting from zero, such that
\begin{align}
 dZ_t^r(x)
 &=dM_t^r(x)+V_t^r(Z_t^r(x))\,dt
 \notag\\
 \shortintertext{and}
 d\langle M^r(x),M^r(y)\rangle_t&=0.
 \label{eq:planar-semimartingale}
\end{align}
The identities hold almost surely for all $t\geq0$.
\end{lemma}

\begin{proof}
First take deterministic labels $\zeta,\xi$. Use the degree-one lift
of $\Psi_t^r$. Its periodic displacement
$\Psi_t^r-\id$ is $H^4(\T)$-valued; in this affine chart the
regularized flow satisfies
\begin{equation}\label{eq:H4-regularized-flow}
 d\Psi_t^r=\sqrt2\sum_{n\geq2,\tau}r^n
       e_{n,\tau}\circ\Psi_t^r\,dW_t^{n,\tau}
       +(P_rb_t)\circ\Psi_t^r\,dt.
\end{equation}
For $T<\infty$ and $m\geq1$, put
\[
 \sigma_m=\inf\left\{t\leq T:
 \|\Psi_t^r-\id\|_{C^5}
 +\|(\Psi_t^r)^{-1}-\id\|_{C^5}\geq m\right\}\wedge T.
\]
On $[0,\sigma_m]$, Sobolev composition estimates give
\begin{equation}\label{eq:H4-Hilbert-Schmidt}
 \sum_{n\geq2,\tau}r^{2n}
   \|e_{n,\tau}\circ\Psi_t^r\|_{H^4}^2
 \leq C_{m,r}\sum_{n\geq2}r^{2n}n^5<\infty.
\end{equation}
Poisson smoothing also gives
\begin{equation}\label{eq:H4-drift-bound}
 \|(P_rb_t)\circ\Psi_t^r\|_{H^4}\leq C_{m,r}.
\end{equation}
Thus \eqref{eq:H4-regularized-flow} is an $H^4$ semimartingale up to
 $\sigma_m$. The stopped path has compact image in $H^4$ and is covered
by finitely many source charts from
Lemma~\ref{lem:boundary-evaluation-chart}. On such a chart, for a source diffeomorphism $\phi$, set
\[
 F_\zeta(\phi)=g_{h_0\circ\phi^{-1}}(\zeta).
\]
The normalized image domains lie in a fixed disk. The estimates in
Lemma~\ref{lem:boundary-evaluation-chart} and the stopping condition give
\[
 \sup_{t\leq\sigma_m}\|d^jF_\zeta(\Psi_t^r)\|_{(H^4)^j\to\mathbb C}
 \leq C_m,\qquad j=1,2.
\]
The integral curve generated by $e_{n,\tau}\circ\phi$ has acceleration
$(e_{n,\tau}'e_{n,\tau})\circ\phi$. Hence
\[
 \partial_{n,\tau}^2g(\zeta)
 =d^2F_\zeta(\phi)
    [e_{n,\tau}\circ\phi,e_{n,\tau}\circ\phi]
   +dF_\zeta(\phi)
       [(e_{n,\tau}'e_{n,\tau})\circ\phi].
\]
Summing over $\tau\in\{c,s\}$ and using
\eqref{eq:circle-Ito-correction} gives
\begin{equation}\label{eq:Hilbert-Hessian-response-trace}
 \sum_{\tau\in\{c,s\}}
 d^2F_\zeta(\phi)
    [e_{n,\tau}\circ\phi,e_{n,\tau}\circ\phi]
 =\sum_{\tau\in\{c,s\}}\partial_{n,\tau}^2g(\zeta).
\end{equation}
Hilbert-space It\^o's formula now gives
\begin{equation}\label{eq:fixed-r-planar-Ito}
 \begin{aligned}
 dg_t^r(\zeta)
 &=\sqrt2\sum_{n\geq2,\tau}r^n
       \partial_{n,\tau}g_t^r(\zeta)\,dW_t^{n,\tau}\\
 &\quad+\left(\partial_{P_rb_t}g_t^r(\zeta)
       +\sum_{n\geq2,\tau}r^{2n}
          \partial_{n,\tau}^2g_t^r(\zeta)\right)dt.
 \end{aligned}
\end{equation}
The stochastic series converges in $\mathcal H^2$ on
$[0,\sigma_m]$, because
\begin{equation}\label{eq:fixed-r-planar-noise-tail}
 \mathbb E\sup_{t\leq \sigma_m}
 \left|\sqrt2\sum_{n>N,\tau}r^n
       \int_0^t\partial_{n,\tau}g_s^r(\zeta)\,dW_s^{n,\tau}\right|^2
 \leq C_{m,r}T\sum_{n>N}r^{2n}n^5
 \longrightarrow0.
\end{equation}
The trace is absolutely convergent:
\begin{equation}\label{eq:fixed-r-planar-trace-bound}
 \sum_{n\geq2,\tau}r^{2n}
       |\partial_{n,\tau}^2g_t^r(\zeta)|
 \leq C_m\sum_{n\geq2,\tau}r^{2n}
       \|e_{n,\tau}\|_{H^4}^2<\infty.
\end{equation}

At every current conformal pair,
\[
 \mu_{e_{n,s}}=i\mu_{e_{n,c}}.
\]
The Cauchy variation formula therefore also gives
\[
 \partial_{n,s}g=i\partial_{n,c}g.
\]
Consequently,
\begin{equation}\label{eq:planar-trace-cancellation}
 (\partial_{n,c}^2+\partial_{n,s}^2)g
 =i[\partial_{n,s},\partial_{n,c}]g=0.
\end{equation}
The same relation also gives
\begin{equation}\label{eq:planar-bracket-cancellation}
 \partial_{n,c}g(\zeta)\partial_{n,c}g(\xi)
 +\partial_{n,s}g(\zeta)\partial_{n,s}g(\xi)=0.
\end{equation}
Indeed, $[e_{n,c},e_{n,s}]=n\alpha_n^2\partial_\theta$, and source
rotation leaves the normalized exterior map unchanged. By
\eqref{eq:fixed-r-planar-trace-bound}, the first cancellation may be
summed in \eqref{eq:fixed-r-planar-Ito}; by
\eqref{eq:fixed-r-planar-noise-tail}, the second may be summed in the
bilinear complex bracket. Lemma~\ref{lem:boundary-evaluation-chart}
then gives
\[
 \partial_{P_rb_t}g_t^r(\zeta)
 =\operatorname{Cauchy}\mu_{P_rb_t}(g_t^r(\zeta)).
\]
Thus the martingale is
\begin{align}
 M_t^r(x)&=\sqrt2\sum_{n\geq2,\tau}r^n
   \int_0^t\partial_{n,\tau}g_s^r(\zeta)\,dW_s^{n,\tau}\notag\\
 \shortintertext{and the semimartingale identity reads}
 M_t^r(x)&=Z_t^r(x)-x-
   \int_0^tV_s^r(Z_s^r(x))\,ds.
 \label{eq:fixed-r-planar-martingale}
\end{align}
Hence \eqref{eq:planar-semimartingale} holds on $[0,\sigma_m]$. The smooth flow and its inverse have bounded
$C^5$ norms on compact time intervals, so $\sigma_m\uparrow T$ almost
surely. Since $T$ is arbitrary, the identities hold for all $t\geq0$.

Now let $\zeta$ be $\mathcal G_0$-measurable. Choose finite
$\mathcal G_0$-measurable partitions $(A_{k,j})_{j\leq q_k}$ and
points $\zeta_{k,j}\in\T_{\mathbb Q}$, and define
\[
 \zeta_k=\zeta_{k,j}\quad\text{on }A_{k,j}.
\]
By construction,
\[
 d_\T(\zeta_k,\zeta)\leq k^{-1}.
\]
Put $x_k=g_0(\zeta_k)$ and $x=g_0(\zeta)$. Since
$A_{k,j}\in\mathcal G_0$, the finite-label processes satisfy
\[
 Z^r(x_k)=\sum_{j=1}^{q_k}\mathbf 1_{A_{k,j}}
              Z^r(g_0(\zeta_{k,j})).
\]
With the same partition,
\[
 M^r(x_k)=\sum_{j=1}^{q_k}\mathbf 1_{A_{k,j}}
              M^r(g_0(\zeta_{k,j})).
\]
Define $M^r(x)$ by the stochastic series in
\eqref{eq:fixed-r-planar-martingale} with $\zeta$ in place of a
deterministic label. The integrands are predictable as pointwise limits
of the finite-label predictable integrands, and
\eqref{eq:H4-Hilbert-Schmidt} gives stopped $\mathcal H^2$ convergence.
On $[0,\sigma_m]$, continuity of the chart derivatives and
\eqref{eq:H4-Hilbert-Schmidt} imply
\begin{equation}\label{eq:random-label-H2-limit}
 \|M^r(x_k)-M^r(x)\|_{\mathcal H^2[0,\sigma_m]}
 \longrightarrow0.
\end{equation}
In the next display, $\|\cdot\|_{\mathcal H^2}$ abbreviates the
stopped norm on $[0,\sigma_m]$. For two such approximations $x_k,y_k$,
Kunita--Watanabe gives
\begin{align*}
 &\mathbb E\sup_{t\leq \sigma_m}
 \left|\langle M^r(x_k),M^r(y_k)\rangle_t
       -\langle M^r(x),M^r(y)\rangle_t\right|\\
 &\quad\leq
 \|M^r(x_k)-M^r(x)\|_{\mathcal H^2}
       \|M^r(y_k)\|_{\mathcal H^2}\\
 &\quad+
 \|M^r(x)\|_{\mathcal H^2}
       \|M^r(y_k)-M^r(y)\|_{\mathcal H^2}
 \longrightarrow0.
\end{align*}
Joint continuity of the boundary flow also gives
\[
 \sup_{t\leq \sigma_m}|Z_t^r(x_k)-Z_t^r(x)|\longrightarrow0
 \qquad\text{almost surely}.
\]
By the definition of $\mu_v$ and
Lemma~\ref{lem:drift-Beltrami},
\[
 \|\mu_{P_rb_s}\|_\infty\leq C_\kappa.
\]
The fixed-domain estimate also gives
\[
 \operatorname{supp}\mu_{P_rb_s}\subset2\overline{\mathbb D}.
\]
Hence the elementary Cauchy-transform modulus yields
\[
 |V_s^r(z)-V_s^r(w)|
 \leq C|z-w|\log\!\frac{A}{|z-w|}
\]
with the continuous convention at $z=w$.  Along the convergent boundary
paths the right-hand side tends to zero pointwise and is bounded by a
deterministic constant. Dominated convergence therefore gives
\[
 \sup_{t\leq \sigma_m}
 \left|\int_0^t
   \{V_s^r(Z_s^r(x_k))-V_s^r(Z_s^r(x))\}\,ds\right|
 \longrightarrow0.
\]
Together with \eqref{eq:random-label-H2-limit}, this passes
\eqref{eq:fixed-r-planar-martingale} to the random label. The bracket
limit above then preserves the zero complex bilinear bracket. The same
argument applies to $\xi$, and finally $m\uparrow\infty$ removes the
stopping.
\end{proof}

\subsection{Uniform planar characteristics}

In the next two lemmas, $M^r$ and $V^r$ are the martingale and drift
from Lemma~\ref{lem:fixed-r-planar-characteristics}. Random boundary
points of $\gamma_0$ are understood to be $\mathcal G_0$-measurable.
Constants denoted by $C$ may depend on $\kappa$ but are uniform in $r$
and the initial normalized pair. For current boundary points $z,w$,
write $\delta=|z-w|$, with the continuous convention at $\delta=0$
fixed above.

\begin{lemma}\label{lem:planar-drift-bounds}
With a constant $C$ as above, the drift satisfies, almost surely,
for all $t\geq0$ and $z,w\in\gamma_t^r$,
\begin{align}
 |V_t^r(z)|&\leq C
 \notag\\
 \shortintertext{and, for the increment,}
 |V_t^r(z)-V_t^r(w)|
 &\leq C\delta\log(A/\delta).
 \label{eq:planar-drift-bounds}
\end{align}
\end{lemma}

\begin{proof}
As observed before Lemma~\ref{lem:boundary-evaluation-chart}, the
closed bounded component of every current normalized pair lies in
$2\overline{\mathbb D}$. Lemma~\ref{lem:drift-Beltrami} bounds the
transported drift density uniformly, including after Poisson smoothing.
For a bounded density $\mu$ supported in this disk and
$z\in2\overline{\mathbb D}$, the Cauchy-transform estimate is
\[
 |\operatorname{Cauchy}\mu(z)|\leq C\|\mu\|_\infty.
\]
For two points $z,w$, we also have
\[
 |\operatorname{Cauchy}\mu(z)-\operatorname{Cauchy}\mu(w)|
 \leq C\|\mu\|_\infty |z-w|\log(A/|z-w|).
\]
Indeed, if $\delta=|z-w|\leq1$, the union of the two
$2\delta$-disks contributes at most $C\|\mu\|_\infty\delta$.
On its complement,
\begin{align*}
 \left|\frac1{z-\zeta}-\frac1{w-\zeta}\right|
 &\leq\frac{C\delta}{|z-\zeta|^2},\\
 \shortintertext{so polar integration gives}
 \int_{2\delta}^{8}\frac{\delta}{s^2}s\,ds
 &\leq C\delta\log(8/\delta).
\end{align*}
If $1<\delta\leq4$, the single-point bound and
$\delta\log(A/\delta)\geq1$ give the same increment estimate after
increasing $C$. This proves both drift bounds.
\end{proof}

For $0<r<1$ and $\mathcal G_0$-measurable initial points
$x,y\in\gamma_0$, write $\Delta M^r$ for the martingale difference
and $\delta_t$ for the boundary separation:
\[
 \Delta M_t^r=M_t^r(x)-M_t^r(y),\qquad
 \delta_t=|Z_t^r(x)-Z_t^r(y)|.
\]

\begin{lemma}\label{lem:planar-noise-bounds}
The brackets below are absolutely continuous in time. There is a
constant $C$, depending only on $\kappa$, such that their densities
satisfy, for $dt\otimes d\mathbb P$-almost every $(t,\omega)$,
\begin{align}
 \frac{d}{dt}\langle M^r(x),\overline{M^r(x)}\rangle_t
 &\leq C,
 \notag\\
 \shortintertext{and}
 \frac{d}{dt}\langle\Delta M^r,\overline{\Delta M^r}\rangle_t
 &\leq C\delta_t^2\log(A/\delta_t).
 \label{eq:planar-noise-bounds}
\end{align}
\end{lemma}

\begin{proof}
The stochastic representation \eqref{eq:fixed-r-planar-martingale}
gives absolute continuity of both brackets. We estimate their densities.
Let $d_H$ be the hyperbolic distance in $\mathbb D$, normalized by
\[
 d_H(0,s)=\log\frac{1+s}{1-s},\qquad 0\leq s<1.
\]
This is the convention in \cite[Theorem~6.1]{AMT}.
Equation~\eqref{eq:linear-DE-Beltrami} gives
\begin{align*}
 \bar\partial E_{e_{n,c}}(u)
 &=\frac{i\alpha_n}{4}(n^3-n)(1-|u|^2)^2\bar u^{n-2},
 \\
 \shortintertext{and}
 \bar\partial E_{e_{n,s}}(u)
 &=-\frac{\alpha_n}{4}(n^3-n)(1-|u|^2)^2\bar u^{n-2}.
\end{align*}
Since $\alpha_n^2=2/(n^3-n)$,
\[
 \sum_{\tau\in\{c,s\}}
 \bar\partial E_{e_{n,\tau}}(u)
 \overline{\bar\partial E_{e_{n,\tau}}(v)}
 =\frac14(n^3-n)(1-|u|^2)^2(1-|v|^2)^2
                  (\bar u v)^{n-2}.
\]
We also use
\[
 \sum_{n\geq2}(n^3-n)x^{n-2}=\frac6{(1-x)^4},
 \qquad |x|<1.
\]
Define the regularized covariance kernel by
\[
 \begin{aligned}
 C_r(u,v)
 &=2\sum_{n\geq2}\sum_{\tau\in\{c,s\}}r^{2n}
   \bar\partial E_{e_{n,\tau}}(u)
   \overline{\bar\partial E_{e_{n,\tau}}(v)}\\
 &=\frac12(1-|u|^2)^2(1-|v|^2)^2
   \sum_{n\geq2}(n^3-n)r^{2n}(\bar u v)^{n-2}\\
 &=\frac{3r^4(1-|u|^2)^2(1-|v|^2)^2}
          {(1-r^2\bar u v)^4}.
 \end{aligned}
\]
In particular,
\[
 |C_r(u,v)|\leq48e^{-2d_H(u,v)}.
\]
This is four times the covariance in \cite[(6.4)]{AMT}.
For $|a|\leq1$,
\[
 |1-r^2a|^2-r^2|1-a|^2
 =(1-r^2)(1-r^2|a|^2)\geq0.
\]
Consequently, $|C_r(u,v)|\leq|C_1(u,v)|$. At $r=1$, invariance of
$|(u-v)/(1-\bar u v)|$ under disk automorphisms and the radial
normalization of $d_H$ give the displayed hyperbolic bound,
uniformly in $r$.
At the current time set $f=f_t^r$ and $\Omega=f(\mathbb D)$.
On $\{x\ne y\}$, set $z=Z_t^r(x)$, $w=Z_t^r(y)$, and
$\delta=\delta_t>0$. Define the kernel difference by
\[
 k_{z,w}(u)=\frac1{z-u}-\frac1{w-u}
            =\frac{w-z}{(z-u)(w-u)},
\]
and define
\[
 I_\Omega(z,w)=\int_{\Omega^2}
 \frac{e^{-2d_H(f^{-1}(u),f^{-1}(v))}}
 {|z-u|\,|w-u|\,|z-v|\,|w-v|}\,dA(u)\,dA(v).
\]
Put $\xi_u=f^{-1}(u)$ and $\xi_v=f^{-1}(v)$.
Termwise integration is justified by the diagonal covariance bound:
Cauchy--Schwarz gives
\[
 2\sum_{n\geq2,\tau}r^{2n}
 |\bar\partial E_{e_{n,\tau}}(\xi_u)|
 |\bar\partial E_{e_{n,\tau}}(\xi_v)|
 \leq\sqrt{C_r(\xi_u,\xi_u)C_r(\xi_v,\xi_v)}\leq3.
\]
Since $k_{z,w}\in L^1(\Omega)$, the product
$3|k_{z,w}(u)k_{z,w}(v)|$ is integrable on $\Omega^2$.
The Cauchy variation formula therefore gives
\[
 \begin{aligned}
 \frac{d}{dt}\langle\Delta M^r,\overline{\Delta M^r}\rangle_t
 &=\frac1{\pi^2}\int_{\Omega^2}
 k_{z,w}(u)\overline{k_{z,w}(v)}
 \frac{f'(\xi_u)}{\overline{f'(\xi_u)}}
 \frac{\overline{f'(\xi_v)}}{f'(\xi_v)}\\
 &\hspace{16mm}\cdot
 C_r(\xi_u,\xi_v)\,dA(u)\,dA(v).
 \end{aligned}
\]
The phase factors have modulus one, and hence
\[
 \frac{d}{dt}\langle\Delta M^r,\overline{\Delta M^r}\rangle_t
 \leq\frac{48}{\pi^2}\delta^2 I_\Omega(z,w).
\]

Let $\lambda_\Omega$ be the Poincar\'e density of $\Omega$, and let
$d_*$ be the quasihyperbolic distance in
$\mathbb C\setminus\{z,w\}$. Koebe's estimate gives
\begin{align*}
 \lambda_\Omega(u)
 &\geq\frac1{4\operatorname{dist}(u,\partial\Omega)}
 \geq\frac1{4\operatorname{dist}(u,\{z,w\})},
 \\
 \shortintertext{and hence}
 2d_H(f^{-1}(u),f^{-1}(v))
 &\geq\frac12d_*(u,v).
\end{align*}
Apply the similarity
\[
 \zeta\longmapsto\frac{\zeta-(z+w)/2}{w-z}.
\]
The marked points become $-1/2$ and $1/2$. The four denominator
factors and the two area elements contribute the same factor
$\delta^4$, so $I_\Omega(z,w)$ is unchanged by this simultaneous
change of variables. The normalized domain lies in $B(0,R)$ with
$R=8/\delta\geq2$. Below, $u,v$ denote normalized coordinates and
$d_*$ the transported quasihyperbolic distance; $\delta$ retains
its original value $|z-w|$. Put
\[
 a(u)=\operatorname{dist}(u,\{-1/2,1/2\}),\qquad
 q(u)=\frac1{|u-1/2|\,|u+1/2|}.
\]
Since $|\nabla\log a|\leq a^{-1}$ almost everywhere,
\[
 d_*(u,v)\geq|\log a(u)-\log a(v)|.
\]
Let $S_0=B(0,2)$ and, for $j\geq1$, let
\[
 S_j=\{2^j\leq|u|<2^{j+1}\}.
\]
Then
\begin{align*}
 \int_{S_j}q(u)\,dA(u)&\leq C,
 \\
 \shortintertext{and}
 \sup_{u\in S_j,\,v\in S_k}e^{-d_*(u,v)/2}
 &\leq C2^{-|j-k|/2}.
\end{align*}
With $J=\lceil\log_2R\rceil\geq1$, the original integral satisfies
\[
 I_\Omega(z,w)
 \leq C\sum_{j,k=0}^J2^{-|j-k|/2}
 \leq C(J+1)
 \leq C\log(A/\delta).
\]
This proves the increment bound on $\{x\ne y\}$. On $\{x=y\}$,
$\Delta M^r=0$, so the bound holds there as well. For a single point,
set $z=Z_t^r(x)$. The same covariance formula with the kernel
$(z-u)^{-1}$ gives
\[
 \frac{d}{dt}\langle M^r(x),\overline{M^r(x)}\rangle_t
 \leq\frac{48}{\pi^2}
       \left(\int_\Omega\frac{dA(u)}{|z-u|}\right)^2
 \leq C,
\]
because $\Omega\subset2\mathbb D$ and $z\in2\overline{\mathbb D}$.
Thus both estimates in \eqref{eq:planar-noise-bounds} hold.
\end{proof}

\subsection{Two-point separation estimates}

Fix $\eta\in(0,1)$ and retain the distance process $\delta_t$ defined
above; in particular, $\delta_0=|x-y|$. The initial points $x,y\in\gamma_0$
are $\mathcal G_0$-measurable. All conditional probabilities below are
taken given $\mathcal G_0$, and their inequalities hold almost surely.
All constants are deterministic and independent of $r$ and the
initial normalized pair.

\begin{samepage}
For the power-tail estimate, the spatial exponents
$\lambda_-,\lambda_+$ and probability exponent $q_T$ must satisfy
\[
 1-\eta<\lambda_-<1<\lambda_+<1+\eta,
\]
and
\[
 q_T>\frac{4(1+\eta)}{\lambda_-}.
\]
\end{samepage}

\begin{lemma}\label{lem:planar-two-point-tail}
For every sufficiently small $T\in(0,1]$, there are exponents in the
preceding ranges and constants $c_0,C_0>0$ such that, uniformly in
$0<r<1$, every pair with $\delta_0>0$ almost surely satisfies
\begin{equation}\label{eq:planar-two-point-power-tail}
 \mathbb P\!\left(
 \exists t\in[0,T]:\delta_t\notin
 [c_0\delta_0^{\lambda_+},C_0\delta_0^{\lambda_-}]
 \,\middle|\,\mathcal G_0\right)
 \leq C_0\delta_0^{q_T}.
\end{equation}
\end{lemma}

\begin{proof}
Fix $0<r<1$ and a pair with $\delta_0>0$ almost surely. First work
on an arbitrary interval $[0,T]$ with $0<T\leq1$, and put
\[
 \Delta_t=Z_t^r(x)-Z_t^r(y),\qquad |\Delta_t|=\delta_t.
\]
Let $\mathcal Q_t$ be the predictable bracket density from
Lemma~\ref{lem:planar-noise-bounds}, so that
\[
 d\langle\Delta M^r,\overline{\Delta M^r}\rangle_t=\mathcal Q_t\,dt.
\]
For fixed $r<1$, both $g_0|_{\T}$ and $g_t^r|_{\T}$ are injective;
therefore $x\ne y$ implies
\[
 \Delta_t\ne0\qquad(0\leq t\leq T).
\]
Also $\delta_t\leq4$ by the exterior normalization. Thus, with $A>4e$,
\[
 u_t=\sqrt{\log(A/\delta_t)}\geq1
 \qquad(0\leq t\leq T).
\]
Write square brackets for quadratic variation. The zero complex
bilinear bracket gives
\begin{align*}
 d\log|\Delta_t|
 &=\operatorname{Re}\frac{d\Delta M_t^r}{\Delta_t}
   +\operatorname{Re}\frac{V_t^r(Z_t^r(x))-V_t^r(Z_t^r(y))}{\Delta_t}\,dt,
 \\
 \shortintertext{and}
 d[\log|\Delta|]_t
 &=\frac{\mathcal Q_t}{2|\Delta_t|^2}\,dt.
\end{align*}
It\^o's formula gives the local martingale $N$, with $N_0=0$, through
\[
 dN_t=-\frac1{2u_t}\operatorname{Re}
                  \frac{d\Delta M_t^r}{\Delta_t},
\]
with bracket
\[
 d\langle N\rangle_t=\frac{\mathcal Q_t}{8u_t^2|\Delta_t|^2}\,dt.
\]
The drift coefficient of $u$ is
\[
 \beta_t=-\frac1{2u_t}\operatorname{Re}
          \frac{V_t^r(Z_t^r(x))-V_t^r(Z_t^r(y))}{\Delta_t}
          -\frac{\mathcal Q_t}{16u_t^3|\Delta_t|^2}.
\]
Equations~\eqref{eq:planar-drift-bounds} and
\eqref{eq:planar-noise-bounds} give, for almost every $t$,
\[
 \frac{\mathcal Q_t}{|\Delta_t|^2}\leq Cu_t^2,
 \qquad
 \frac{|V_t^r(Z_t^r(x))-V_t^r(Z_t^r(y))|}{|\Delta_t|}
 \leq Cu_t^2.
\]
Since $u_t\geq1$, the preceding formulas give
\begin{equation}\label{eq:planar-log-SDE}
 du_t=dN_t+\beta_t\,dt.
\end{equation}
Moreover,
\begin{align*}
 d\langle N\rangle_t&\leq C\,dt,
 \\
 \shortintertext{and}
 |\beta_t|&\leq C u_t.
\end{align*}
The estimates hold for every $r<1$. In particular,
\begin{align*}
 e^{-Ct}u_t
 &\leq u_0+\int_0^t e^{-Cs}\,dN_s,
 \\
 \shortintertext{and}
 e^{Ct}u_t
 &\geq u_0+\int_0^t e^{Cs}\,dN_s.
\end{align*}
The brackets of these stochastic integrals are at most
$CTe^{2CT}$ on $[0,T]$. Their exponential maximal inequalities,
conditionally on $\mathcal G_0$, give, for $0<\varepsilon<1$,
\begin{equation}\label{eq:planar-two-point-tail}
 \mathbb P\left(\exists t\leq T:
 u_t\notin
 [e^{-CT}(1-\varepsilon)u_0,e^{CT}(1+\varepsilon)u_0]
 \mid\mathcal G_0\right)
 \leq4\exp\!\left(-\frac{c\varepsilon^2u_0^2}{Te^{2CT}}\right).
\end{equation}
Set
\begin{align*}
 \lambda_-&=e^{-2CT}(1-\varepsilon)^2,
 \qquad \lambda_+=e^{2CT}(1+\varepsilon)^2,\\
 \shortintertext{and}
 q_T&=\frac{c\varepsilon^2}{Te^{2CT}},
\end{align*}
where $c$ is the constant in \eqref{eq:planar-two-point-tail}. Since
\[
 \delta_t=Ae^{-u_t^2},
 \qquad
 u_0^2=\log(A/\delta_0),
\]
the two exit conditions are exactly
\begin{align*}
 u_t>\sqrt{\lambda_+}\,u_0
 &\quad\Longleftrightarrow\quad
 \delta_t<A^{1-\lambda_+}\delta_0^{\lambda_+},
 \\
 \shortintertext{and}
 u_t<\sqrt{\lambda_-}\,u_0
 &\quad\Longleftrightarrow\quad
 \delta_t>A^{1-\lambda_-}\delta_0^{\lambda_-}.
\end{align*}
Moreover,
\[
 4\exp(-q_Tu_0^2)=4A^{-q_T}\delta_0^{q_T}.
\]
Consequently, the power-tail estimate holds with
$c_0=A^{1-\lambda_+}$ and any deterministic
$C_0\geq\max\{A^{1-\lambda_-},4A^{-q_T},1\}$.

Choose $\varepsilon>0$ so that
\[
 (1-\varepsilon)^2>1-\eta,\qquad
 (1+\varepsilon)^2<1+\eta.
\]
As $T\downarrow0$,
\[
 \lambda_-\longrightarrow(1-\varepsilon)^2,\qquad
 \lambda_+\longrightarrow(1+\varepsilon)^2,\qquad q_T\longrightarrow\infty.
\]
Thus, for every sufficiently small $T\in(0,1]$,
\[
 1-\eta<\lambda_-<1<\lambda_+<1+\eta,
\]
and
\begin{equation}\label{eq:planar-exponent-choice}
 q_T>\frac{4(1+\eta)}{\lambda_-}.
\end{equation}
With these choices, the preceding calculation proves
\eqref{eq:planar-two-point-power-tail}.
\end{proof}

We next control pairs whose initial separation is bounded below.
The time and spatial exponents will be chosen together with those in
Lemma~\ref{lem:planar-two-point-tail}.

\begin{samepage}
\begin{lemma}\label{lem:planar-coarse-pair-tail}
There exist a time $T=T_\eta$ as in Lemma~\ref{lem:planar-two-point-tail}
and an exponent $p_T>0$, depending only on $\kappa$ and $\eta$, such
that, for every $d_0>0$,
\begin{equation}\label{eq:planar-coarse-pair-tail}
 \mathbb P\!\left(
   \inf_{t\leq T}\delta_t<\rho^{\lambda_-/2}
   \,\middle|\,\mathcal G_0\right)
 \leq C_{d_0}\rho^{p_T+4},\qquad 0<\rho<\rho_0,
\end{equation}
uniformly in $0<r<1$ and in pairs with $\delta_0\geq d_0/2$ almost
surely. Here $C_{d_0}<\infty$ and $\rho_0\in(0,1)$ depend only on
$\kappa,\eta,d_0$.
\end{lemma}
\end{samepage}

\begin{proof}
Use $u_t=\sqrt{\log(A/\delta_t)}$ from the preceding proof. For
$0<T\leq1$, its integrating-factor estimate and the exponential
maximal inequality give, for every deterministic level $R>0$,
\begin{equation}\label{eq:planar-u-upper-tail}
 \mathbb P\left(\sup_{t\leq T}u_t\geq R\,\middle|\,\mathcal G_0\right)
 \leq2\exp\left\{-\frac{c(e^{-CT}R-u_0)_+^2}
                           {Te^{2CT}}\right\}.
\end{equation}
Here $s_+=\max\{s,0\}$, and $c,C>0$ are independent of $T\leq1$,
$r$, and the initial pair. Set
\[
 c_p=\frac{ce^{-2C}}8.
\]
Choose $T=T_\eta$ sufficiently small that the preceding lemma holds
and
\[
 \frac{c_p\lambda_-}{Te^{2CT}}>4.
\]
This is possible because $\lambda_->1-\eta>0$. Define
\[
 p_T=\frac{c_p\lambda_-}{Te^{2CT}}-4>0.
\]
This choice of $T$ and $p_T$ is independent of $d_0$.

When pairs with $|x-y|\geq d_0/2$ exist, $d_0\leq8$ and
$u_0\leq\sqrt{\log(2A/d_0)}$. Put
\[
 R_\rho=\sqrt{\log(A\rho^{-\lambda_-/2})}.
\]
For all sufficiently small $\rho$, uniformly over these initial pairs,
$e^{-CT}R_\rho\geq2\sqrt{\log(2A/d_0)}$. Since $T\leq1$,
\begin{align*}
 (e^{-CT}R_\rho-u_0)^2
 &\geq\tfrac14e^{-2CT}R_\rho^2\\
 &\geq\frac{e^{-2C}\lambda_-}{8}\log\frac1\rho.
\end{align*}
Since $\delta_t=Ae^{-u_t^2}$, the event
$\{\inf_{t\leq T}\delta_t<\rho^{\lambda_-/2}\}$ is contained in
$\{\sup_{t\leq T}u_t\geq R_\rho\}$.
Applying \eqref{eq:planar-u-upper-tail} at $R_\rho$ therefore gives
\[
 \mathbb P\!\left(
   \inf_{t\leq T}\delta_t<\rho^{\lambda_-/2}
   \,\middle|\,\mathcal G_0\right)
 \leq C_{d_0}
 \exp\left\{-\frac{c_p\lambda_-}{Te^{2CT}}
                         \log\frac1\rho\right\}
 =C_{d_0}\rho^{p_T+4}.
\]
This is \eqref{eq:planar-coarse-pair-tail}.
\end{proof}

\section{Continuous Jordan realization}
\label{sec:Jordan-realization}

\subsection{Boundary tightness}

Retain the stationary response probability space and the regularized
planar motions from Section~\ref{sec:welding-support}; no H\"older
regularity of the initial curve is assumed here. Fix $\eta\in(0,1)$
and the common choices $T=T_\eta$, $\lambda_-$, $\lambda_+$, $q_T$,
and $p_T$ from Lemmas~\ref{lem:planar-two-point-tail} and
\ref{lem:planar-coarse-pair-tail}. Constants denoted by $c,C$ in this
subsection depend only on $\kappa$ and $\eta$. Conditional probabilities
are taken given $\mathcal G_0$; their $L^\infty$ norms refer to the
underlying probability measure.

Enumerate $\gamma_{0,\mathbb Q}$ using a fixed enumeration of
$\T_{\mathbb Q}$. For $m\geq1$, retain its first point, and then
successively retain the first point at distance greater than $2^{-m}$
from all retained points. This is a $\mathcal G_0$-measurable selection.
The disjoint disks of radius $2^{-m-1}$ about the retained points lie
in $B(0,2+2^{-m-1})$, so the procedure stops after finitely many steps.
Write $\mathcal N_m$ for the resulting set. The packing estimate and
density of $\gamma_{0,\mathbb Q}$ give
\begin{equation}\label{eq:planar-net-size}
 \#\mathcal N_m\leq C2^{2m},\qquad
 \sup_{x\in\gamma_0}\operatorname{dist}(x,\mathcal N_m)\leq2^{-m}.
\end{equation}

\begin{lemma}\label{lem:planar-upper-holder}
For each $0<r<1$, there is a random variable $H_r^+\geq1$, finite
almost surely, for which the bound
\begin{equation}\label{eq:planar-upper-random-constant}
 \sup_{t\leq T}|Z_t^r(x)-Z_t^r(y)|
 \leq H_r^+|x-y|^{\lambda_-}
\end{equation}
holds almost surely for all $x,y\in\gamma_0$ simultaneously.
Moreover,
\[
 \lim_{R\to\infty}\sup_{0<r<1}
 \left\|\mathbb P(H_r^+>R\mid\mathcal G_0)\right\|_\infty=0.
\]
\end{lemma}

\begin{proof}
Let $\mathcal P_m^+$ consist of the distinct pairs in
$\mathcal N_m^2\cup(\mathcal N_m\times\mathcal N_{m+1})$ whose
initial distance is at most $4\,2^{-m}$, and set
\[
 E_{m,+}^r=
 \bigcup_{(x,y)\in\mathcal P_m^+}
 \left\{\sup_{t\leq T}|Z_t^r(x)-Z_t^r(y)|
       >C_0|x-y|^{\lambda_-}\right\}.
\]
Since $\#\mathcal P_m^+\leq C2^{4m}$,
equations~\eqref{eq:planar-two-point-power-tail} and
\eqref{eq:planar-net-size} give
\begin{equation}\label{eq:planar-upper-net-failure}
 \mathbb P(E_{m,+}^r\mid\mathcal G_0)
 \leq C_0\sum_{(x,y)\in\mathcal P_m^+}|x-y|^{q_T}
 \leq C2^{-(q_T-4)m}.
\end{equation}
Hence
\[
 \sup_{0<r<1}\left\|\mathbb P\left(
    \bigcup_{m\geq m_0}E_{m,+}^r\,\middle|\,\mathcal G_0\right)
 \right\|_\infty
 \leq C\sum_{m\geq m_0}2^{-(q_T-4)m}\longrightarrow0.
\]
Fix $m_0\geq1$ and work outside $\bigcup_{m\geq m_0}E_{m,+}^r$.
For $x\in\gamma_0$, choose a nearest point $x_m\in\mathcal N_m$ at
each level $m\geq m_0$. Since
$|x_m-x_{m+1}|\leq2^{-m}+2^{-m-1}$, every distinct successive pair
belongs to $\mathcal P_m^+$. For $m'>m\geq m_0$,
\[
 \sup_{t\leq T}|Z_t^r(x_{m'})-Z_t^r(x_m)|
 \leq\sum_{\ell=m}^{m'-1}\sup_{t\leq T}
             |Z_t^r(x_{\ell+1})-Z_t^r(x_\ell)|
 \leq C\sum_{\ell=m}^{m'-1}2^{-\lambda_-\ell}.
\]
Uniform continuity of the fixed-$r$ boundary flow gives
\[
 \sup_{t\leq T}|Z_t^r(x_{m'})-Z_t^r(x)|\longrightarrow0
 \qquad(m'\to\infty).
\]
Letting $m'\to\infty$ therefore yields
\[
 \sup_{t\leq T}|Z_t^r(x)-Z_t^r(x_m)|
 \leq\frac{C2^{-\lambda_-m}}{1-2^{-\lambda_-}}.
\]
For $2^{-m-1}<|x-y|\leq2^{-m}$ with $m\geq m_0$, the nearest
points satisfy $|x_m-y_m|\leq3\,2^{-m}$. Thus the middle pair also
belongs to $\mathcal P_m^+$ unless its two points coincide, and
\[
 \begin{aligned}
 \sup_{t\leq T}|Z_t^r(x)-Z_t^r(y)|
 &\leq\sup_{t\leq T}|Z_t^r(x)-Z_t^r(x_m)|
       +\sup_{t\leq T}|Z_t^r(x_m)-Z_t^r(y_m)|\\
 &\quad+\sup_{t\leq T}|Z_t^r(y_m)-Z_t^r(y)|\\
 &\leq C2^{-\lambda_-m}
 \leq C2^{\lambda_-}|x-y|^{\lambda_-}.
 \end{aligned}
\]
Define the measurable extended random variable
\[
 H_r^+
 =1\vee
 \sup_{\substack{t\in\mathbb Q\cap[0,T]\\
                   x,y\in\gamma_{0,\mathbb Q},\ x\ne y}}
 \frac{|Z_t^r(x)-Z_t^r(y)|}{|x-y|^{\lambda_-}}
 \in[1,\infty].
\]
Continuity in $(t,x,y)$ makes this equal to the supremum over
$[0,T]\times\gamma_0^2$ off the diagonal. On
$\bigcap_{m\geq m_0}(E_{m,+}^r)^c$, the preceding chaining estimate
controls all pairs with $|x-y|\leq2^{-m_0}$. For the remaining pairs,
the normalized boundary lies in $2\overline{\mathbb D}$, so
\[
 \sup_{t\leq T}|Z_t^r(x)-Z_t^r(y)|\leq4.
\]
Consequently,
\[
 H_r^+\leq C2^{\lambda_-m_0}
 \quad\text{on }\bigcap_{m\geq m_0}(E_{m,+}^r)^c,
\]
after increasing $C$ if necessary.
Therefore
\[
 \sup_{0<r<1}\left\|\mathbb P\left(
 H_r^+>C2^{\lambda_-m_0}
 \,\middle|\,\mathcal G_0\right)\right\|_\infty
 \leq C\sum_{m\geq m_0}2^{-(q_T-4)m}\longrightarrow0.
\]
This proves both almost-sure finiteness and the asserted uniform tail
tightness.
\end{proof}

For the lower bound, choose $\lambda_*\in(\lambda_+,1+\eta)$ and set
\[
 \ell_m=\left\lceil\frac{\lambda_*m}{\lambda_-}\right\rceil+1.
\]
These finer nets will make the approximation error smaller than the
lower separation scale.

\begin{lemma}\label{lem:planar-lower-holder}
For each $0<r<1$, there is a random variable $H_r^-\geq1$, finite
almost surely, for which the bound
\begin{equation}\label{eq:planar-lower-random-constant}
 \inf_{t\leq T}|Z_t^r(x)-Z_t^r(y)|
 \geq (H_r^-)^{-1}|x-y|^{\lambda_+}
\end{equation}
holds almost surely for all $x,y\in\gamma_0$ simultaneously.
Moreover,
\[
 \lim_{R\to\infty}\sup_{0<r<1}
 \left\|\mathbb P(H_r^->R\mid\mathcal G_0)\right\|_\infty=0.
\]
\end{lemma}

\begin{proof}
Let $\mathcal P_m^-$ consist of the pairs in $\mathcal N_{\ell_m}^2$
whose initial distances lie in $[2^{-m-2},2^{-m+2}]$. Define
\[
 E_{m,-}^r=
 \bigcup_{(x,y)\in\mathcal P_m^-}
 \left\{\inf_{t\leq T}|Z_t^r(x)-Z_t^r(y)|
       <c_0|x-y|^{\lambda_+}\right\}.
\]
Since $\#\mathcal P_m^-\leq C2^{4\ell_m}$,
\begin{equation}\label{eq:planar-lower-net-failure}
 \mathbb P(E_{m,-}^r\mid\mathcal G_0)
 \leq C2^{4\ell_m-q_Tm}
 \leq C2^{-(q_T-4\lambda_*/\lambda_-)m}.
\end{equation}
Take $2^{-m-1}<|x-y|\leq2^{-m}$ and choose
$x',y'\in\mathcal N_{\ell_m}$ with
\[
 |x-x'|\leq2^{-\ell_m},\qquad |y-y'|\leq2^{-\ell_m}.
\]
Since $\lambda_*/\lambda_->1$, for all sufficiently large $m$,
$2^{-\ell_m}\leq2^{-m-4}$ and hence
\[
 2^{-m-2}\leq|x'-y'|\leq2^{-m+2}.
\]
Thus $(x',y')\in\mathcal P_m^-$. On
$\{H_r^+\leq R\}\setminus E_{m,-}^r$,
\begin{align*}
 \inf_{t\leq T}|Z_t^r(x)-Z_t^r(y)|
 &\geq c2^{-\lambda_+m}-2CR2^{-\lambda_-\ell_m}\\
 &\geq c2^{-\lambda_+m}-2CR2^{-\lambda_-}2^{-\lambda_*m}\\
 &\geq\frac c2\,2^{-\lambda_+m}
\end{align*}
for all sufficiently large $m$, because
\[
 \frac{2CR2^{-\lambda_-\ell_m}}{c2^{-\lambda_+m}}
 \leq\frac{2CR}{c}\,2^{-\lambda_-}
             2^{-(\lambda_*-\lambda_+)m}\longrightarrow0.
\]
Moreover,
\[
 q_T-\frac{4\lambda_*}{\lambda_-}
 >q_T-\frac{4(1+\eta)}{\lambda_-}>0,
\]
so the failure probabilities in \eqref{eq:planar-lower-net-failure}
are summable.

For the remaining scales fix $d_0>0$, let
$\mathcal N(\rho)$ be the $\rho$-net obtained by the same measurable
selection used above, and define
\[
 E_{\rho,d_0}^r=
 \bigcup_{\substack{x,y\in\mathcal N(\rho)\\|x-y|\geq d_0/2}}
 \left\{\inf_{t\leq T}|Z_t^r(x)-Z_t^r(y)|
                 <\rho^{\lambda_-/2}\right\}.
\]
The packing bound gives
$\#(\mathcal N(\rho)^2)\leq C\rho^{-4}$. Applying
\eqref{eq:planar-coarse-pair-tail} to these pairs and taking their
union gives
\begin{equation}\label{eq:planar-coarse-net-failure}
 \mathbb P(E_{\rho,d_0}^r\mid\mathcal G_0)
 \leq C_{d_0}\rho^{p_T}\longrightarrow0
 \qquad(\rho\downarrow0),
\end{equation}
uniformly in $0<r<1$.
Define the measurable extended random variable
\[
 H_r^-
 =1\vee
 \sup_{\substack{t\in\mathbb Q\cap[0,T]\\
                   x,y\in\gamma_{0,\mathbb Q},\ x\ne y}}
 \frac{|x-y|^{\lambda_+}}{|Z_t^r(x)-Z_t^r(y)|}
 \in[1,\infty].
\]
By continuity, this equals the supremum over all times and all
distinct pairs. To prove its uniform tail bound, fix $\varepsilon_0>0$.
First choose $R\geq1$ so that
\[
 \sup_{0<r<1}
 \left\|\mathbb P(H_r^+>R\mid\mathcal G_0)\right\|_\infty
 <\frac{\varepsilon_0}{3}.
\]
Next choose $m_0$ so large that the fine-net lower bound holds for
every $m\geq m_0$ on $\{H_r^+\leq R\}\setminus E_{m,-}^r$ and
\[
 \sup_{0<r<1}
 \left\|\mathbb P\left(
   \bigcup_{m\geq m_0}E_{m,-}^r
   \,\middle|\,\mathcal G_0\right)\right\|_\infty
 <\frac{\varepsilon_0}{3}.
\]
Set $d_0=2^{-m_0-3}$. With $m_0$ fixed, choose
$0<\rho<d_0/8$, below the threshold in
Lemma~\ref{lem:planar-coarse-pair-tail}, such that
\begin{samepage}
\begin{align*}
 2CR\rho^{\lambda_-}&\leq\tfrac12\rho^{\lambda_-/2},\\*
 \shortintertext{and}
 \sup_{0<r<1}
 \left\|\mathbb P(E_{\rho,d_0}^r\mid\mathcal G_0)\right\|_\infty
 &<\frac{\varepsilon_0}{3}.
\end{align*}
\end{samepage}
If $|x-y|\geq d_0$ and $x_\rho,y_\rho\in\mathcal N(\rho)$ are
nearest net points, then $\rho<d_0/8$ gives
\[
 |x_\rho-y_\rho|\geq |x-y|-2\rho>d_0/2,
\]
so the pair $(x_\rho,y_\rho)$ is among those defining
$E_{\rho,d_0}^r$. On $\{H_r^+\leq R\}\setminus E_{\rho,d_0}^r$,
\[
 \begin{aligned}
 \inf_{t\leq T}|Z_t^r(x)-Z_t^r(y)|
 &\geq\inf_{t\leq T}|Z_t^r(x_\rho)-Z_t^r(y_\rho)|
       -R|x-x_\rho|^{\lambda_-}-R|y-y_\rho|^{\lambda_-}\\
 &\geq\rho^{\lambda_-/2}-2R\rho^{\lambda_-}
 \geq\tfrac12\rho^{\lambda_-/2},
 \end{aligned}
\]
where the preceding choice of $\rho$ applies after taking $C\geq1$.
Combining this with the fine-net estimate, outside the three failure
events we obtain, for a deterministic $c_*>0$,
\[
 \inf_{t\leq T}|Z_t^r(x)-Z_t^r(y)|
 \geq
 \begin{cases}
  c_*|x-y|^{\lambda_+},&0<|x-y|<d_0,\\
  \frac12\rho^{\lambda_-/2},&|x-y|\geq d_0.
 \end{cases}
\]
Since every boundary separation is at most $4$,
\[
 H_r^-\leq
 C_{\varepsilon_0}:=
 \max\{1,c_*^{-1},\,2\,4^{\lambda_+}\rho^{-\lambda_-/2}\}
\]
on that event. Equivalently,
\[
 \{H_r^->C_{\varepsilon_0}\}
 \subset\{H_r^+>R\}\cup\bigcup_{m\geq m_0}E_{m,-}^r
                         \cup E_{\rho,d_0}^r.
\]
The conditional union bound and the three parameter choices give
\[
 \sup_{0<r<1}
 \left\|\mathbb P(H_r^->C_{\varepsilon_0}\mid\mathcal G_0)\right\|_\infty
 <\frac{\varepsilon_0}{3}+\frac{\varepsilon_0}{3}
                        +\frac{\varepsilon_0}{3}=\varepsilon_0.
\]
This proves the lemma.
\end{proof}

Combining Lemmas~\ref{lem:planar-upper-holder} and
\ref{lem:planar-lower-holder}, keep the same $T=T_\eta$ and set
\[
 C_\eta=\max\{1,4^{\lambda_--(1-\eta)},
                 4^{(1+\eta)-\lambda_+}\}.
\]
With $H_r=C_\eta(H_r^+\vee H_r^-)$, for each $0<r<1$ the following
bounds hold almost surely, simultaneously for every $t\in[0,T_\eta]$ and
$x,y\in\gamma_0$,
\begin{equation}\label{eq:planar-biHolder}
 H_r^{-1}|x-y|^{1+\eta}
 \leq |Z_t^r(x)-Z_t^r(y)|
 \leq H_r|x-y|^{1-\eta}.
\end{equation}
Moreover,
\begin{equation}\label{eq:planar-Holder-tightness}
 \lim_{R\to\infty}\sup_{0<r<1}
 \left\|\mathbb P(H_r>R\mid\mathcal G_0)\right\|_\infty=0.
\end{equation}

Regard $g^r:(t,\zeta)\mapsto g_t^r(\zeta)$ as a random element of
$C([0,T_\eta]\times\T;\mathbb C)$, equipped with the uniform norm.

\begin{corollary}\label{cor:planar-tightness}
The laws of $g^r$, $0<r<1$, are tight in this space. For every
subsequential weak limit $g$, almost surely,
\[
 g_t(\zeta)\ne g_t(\xi)
 \qquad(0\leq t\leq T_\eta,\ \zeta,\xi\in\T,\ \zeta\ne\xi).
\]
\end{corollary}

\begin{proof}
Set $T=T_\eta$. The drift bound, the single-point bracket bound, and
the conditional Burkholder--Davis--Gundy inequality give, for every
$q\geq2$ and every $\mathcal G_0$-measurable $x\in\gamma_0$,
\[
 \mathbb E\!\left[
 |Z_t^r(x)-Z_s^r(x)|^q\,\middle|\,\mathcal G_0\right]
       \leq C_q\bigl(|t-s|^{q/2}+|t-s|^q\bigr).
\]
Choose $q>2$ and $0<\gamma_{\rm t}<1/2-1/q$, and put
$t_{m,k}=kT2^{-m}$ for $0\leq k\leq2^m$. The increment bound gives
\[
 \sum_{m\geq0}2^{\gamma_{\rm t} qm}
 \mathbb E\!\left[
  \max_{k<2^m}|Z^r_{t_{m,k+1}}(x)-Z^r_{t_{m,k}}(x)|^q
                 \,\middle|\,\mathcal G_0\right]
 \leq C_{q,T}\sum_{m\geq0}2^{-(q/2-1-\gamma_{\rm t} q)m}<\infty.
\]
For $z\in C([0,T];\mathbb C)$, write
\[
 [z]_{C^{\gamma_{\rm t}}([0,T])}
 =\sup_{0\leq s<t\leq T}\frac{|z(t)-z(s)|}{|t-s|^{\gamma_{\rm t}}}.
\]
Dyadic chaining gives
\begin{align*}
 [z]_{C^{\gamma_{\rm t}}([0,T])}
 &\leq C_{\gamma_{\rm t},T}\sup_{m\geq0}2^{\gamma_{\rm t} m}
       \max_{k<2^m}|z(t_{m,k+1})-z(t_{m,k})|,\\
 \shortintertext{and hence}
 [z]_{C^{\gamma_{\rm t}}([0,T])}^q
 &\leq C_{q,\gamma_{\rm t},T}\sum_{m\geq0}2^{\gamma_{\rm t} qm}
       \max_{k<2^m}|z(t_{m,k+1})-z(t_{m,k})|^q.
\end{align*}
Apply this inequality to $z=Z^r(x)$ and take the conditional expectation.
Thus, for every $\mathcal G_0$-measurable $x$,
\begin{equation}\label{eq:planar-time-Holder-moment}
 \sup_{0<r<1}
 \left\|\mathbb E\!\left[
  [Z^r(x)]_{C^{\gamma_{\rm t}}([0,T])}^q\,\middle|\,\mathcal G_0\right]
 \right\|_\infty\leq C_{q,\gamma_{\rm t},T}.
\end{equation}
Let $\omega_{g_0}(\delta)$ be the spatial modulus of $g_0$ and set
\[
 \omega_{\rm time}(g^r,\rho)
 =\sup_{\substack{s,t\leq T\\|t-s|\leq\rho}}
       \sup_{\zeta\in\T}|g_t^r(\zeta)-g_s^r(\zeta)|.
\]
Take a finite $\delta$-net $\zeta_1,\ldots,\zeta_J$ in $\T$.
For $R_0\geq1$ and $v_0>0$, on
$\{H_r\leq R_0,\ \omega_{g_0}(\delta)\leq v_0\}$,
\[
 \omega_{\rm time}(g^r,\rho)
 \leq2R_0v_0^{1-\eta}
   +\max_{j\leq J}\sup_{|t-s|\leq\rho}
                   |g_t^r(\zeta_j)-g_s^r(\zeta_j)|.
\]
For $a>0$ and every $j\leq J$, \eqref{eq:planar-time-Holder-moment}
and Markov's inequality give
\[
 \sup_{0<r<1}\mathbb P\left(
   \sup_{|t-s|\leq\rho}|g_t^r(\zeta_j)-g_s^r(\zeta_j)|>a/2\right)
 \leq C_{q,\gamma_{\rm t},T}(a/2)^{-q}\rho^{\gamma_{\rm t} q}.
\]
Consequently, whenever $2R_0v_0^{1-\eta}\leq a/2$,
\[
 \begin{aligned}
 \sup_{0<r<1}\mathbb P(\omega_{\rm time}(g^r,\rho)>a)
 &\leq\sup_{0<r<1}\mathbb P(H_r>R_0)
       +\mathbb P(\omega_{g_0}(\delta)>v_0)\\
 &\quad+\sum_{j=1}^J\sup_{0<r<1}\mathbb P\left(
   \sup_{|t-s|\leq\rho}|g_t^r(\zeta_j)-g_s^r(\zeta_j)|>a/2\right)\\
 &\leq\sup_{0<r<1}\mathbb P(H_r>R_0)
       +\mathbb P(\omega_{g_0}(\delta)>v_0)
       +JC_{q,\gamma_{\rm t},T}(a/2)^{-q}\rho^{\gamma_{\rm t} q}.
 \end{aligned}
\]
Given $a,\varepsilon_0>0$, first choose $R_0\geq1$ so that
\[
 \sup_{0<r<1}\mathbb P(H_r>R_0)<\frac{\varepsilon_0}{3}.
\]
Set $v_0=(a/(4R_0))^{1/(1-\eta)}$ and choose $\delta>0$ so that
\[
 \mathbb P(\omega_{g_0}(\delta)>v_0)<\frac{\varepsilon_0}{3}.
\]
With $J=J(\delta)$ fixed, choose $\rho>0$ such that
\begin{equation}\label{eq:time-tightness-choices}
 JC_{q,\gamma_{\rm t},T}(a/2)^{-q}\rho^{\gamma_{\rm t} q}<\frac{\varepsilon_0}{3}.
\end{equation}
The combined bound is then less than $\varepsilon_0$, proving
\[
 \lim_{\rho\downarrow0}\sup_{0<r<1}
 \mathbb P(\omega_{\rm time}(g^r,\rho)>a)=0.
\]
The spatial estimate gives
\[
 \sup_{t\leq T}|g_t^r(\zeta)-g_t^r(\xi)|
       \leq H_r\,\omega_{g_0}(|\zeta-\xi|)^{1-\eta}.
\]
For $a>0$ and $R_0\geq1$, this gives
\[
 \begin{aligned}
 &\sup_{0<r<1}\mathbb P\left(
  \sup_{\substack{t\leq T,\ \zeta,\xi\in\T\\|\zeta-\xi|\leq\delta}}
       |g_t^r(\zeta)-g_t^r(\xi)|>a\right)\\
 &\quad\leq\sup_{0<r<1}\mathbb P(H_r>R_0)
       +\mathbb P\left(\omega_{g_0}(\delta)>
                  (a/R_0)^{1/(1-\eta)}\right).
 \end{aligned}
\]
First let $\delta\downarrow0$, then $R_0\to\infty$. The spatial
failure probability tends to zero, uniformly in $r$. Together with
the time estimate and the normalization bound
\[
 \sup_{0<r<1}\sup_{0\leq t\leq T}\sup_{\zeta\in\T}
 |g_t^r(\zeta)|\leq2,
\]
this proves tightness by Arzel\`a--Ascoli.

For any convergent subsequence, pass to a further joint limit with
$(g_0,H_r)$ and use a Skorokhod representation. The copied initial
maps converge uniformly to $g_0$, whose law is the original injective
boundary-map law, while $H_r\to H<\infty$. Hence the lower bound
passes to the uniform limit:
\[
 |g_t(\zeta)-g_t(\xi)|
 \geq H^{-1}|g_0(\zeta)-g_0(\xi)|^{1+\eta}>0
 \quad(\zeta\ne\xi,\ 0\leq t\leq T),
\]
which proves boundary injectivity at every time.
\end{proof}

\subsection{The continuous geometric lift}

Set $T_*=T_{1/2}$ using the preceding construction. For a circle
function $g$ and
$0<\beta\leq1$, write
\[
 [g]_\beta=\sup_{\zeta\ne\xi\in\T}
                 \frac{|g(\zeta)-g(\xi)|}{|\zeta-\xi|^\beta}.
\]

\begin{proposition}
\label{prop:continuous-geometric-lift}
The stationary response path law admits a continuous $Y$-valued
lift $Y_t=(X_t,f_t,g_t)$. For each fixed $t\geq0$,
\begin{align*}
 Y_t&=\iota(X_t)\quad\text{almost surely},\\
 \shortintertext{and consequently}
 \operatorname{Law}(Y_t)&=\nu^Y.
\end{align*}
\end{proposition}

\begin{proof}
Set $T=T_*$. Regard $h^r=(h_t^r)_{0\leq t\leq T}$ and
$X=(X_t)_{0\leq t\leq T}$ as random elements of $C([0,T],X)$,
and $g^r$ as the corresponding boundary function on $[0,T]\times\T$.
In \eqref{eq:regularized-circle-flow} take the Brownian motions
and the adapted environment of Proposition~\ref{prop:single-noise}.
The initial curve is the measurable removable welding of $X_0$.
The relative-flow identity gives
\[
 h_t^r=X_0\circ(\Psi_t^r)^{-1},\qquad
 X_t=X_0\circ\Psi_t^{-1}.
\]
In the following estimate, $\Psi_t^r$ and $\Psi_t$ denote circle maps,
and $\omega_{X_0}$ is the chord-distance modulus of $X_0$.
Since $(h_t^r)^{-1}=\Psi_t^r\circ X_0^{-1}$, the metric $d$ satisfies
\begin{align*}
 \sup_{t\leq T}d(h_t^r,X_t)
 &\leq\omega_{X_0}\!\left(
       \sup_{t\leq T}\|(\Psi_t^r)^{-1}-\Psi_t^{-1}\|_\infty\right)\\
 &\quad+\sup_{t\leq T}\|\Psi_t^r-\Psi_t\|_\infty
 \longrightarrow0\quad\text{in probability}.
\end{align*}
Here Lemma~\ref{lem:Poisson-stability} controls both suprema, and
$\omega_{X_0}(a)\to0$ almost surely as $a\downarrow0$.
Thus $h^r\to X$ in probability in $C([0,T],X)$.
For every sequence $r_j\uparrow1$,
Corollary~\ref{cor:planar-tightness} gives a subsequence, still denoted
$r_j$, such that
\[
 (h^{r_j},g^{r_j},g_0,H_{r_j})
 \Rightarrow(X,g,g_0,H)
\]
in
$C([0,T],X)\times C([0,T]\times\T)\times C(\T)\times[1,\infty)$.
The limiting first coordinate has the stationary response path law.
On one Skorokhod representation, denote the copied initial boundary
in the $j$th approximation by $g_0^{(j)}$. Then
\begin{align*}
 \sup_{t\leq T}d(h_t^{r_j},X_t)
 +\sup_{t\leq T}\|g_t^{r_j}-g_t\|_{C(\T)}&\longrightarrow0,\\
 \shortintertext{and}
 \|g_0^{(j)}-g_0\|_{C(\T)}+|H_{r_j}-H|&\longrightarrow0
 \quad\text{almost surely}.
\end{align*}
For $w\in\overline{\mathbb D}$, extend
$g_t^{r_j}(1/w)-1/w$ continuously to $w=0$ by the value $0$.
The maximum principle gives, for $j,k\geq1$,
\[
 \sup_{t\leq T}\sup_{|w|\leq1}
 \left|\left(g_t^{r_j}(1/w)-\frac1w\right)
       -\left(g_t^{r_k}(1/w)-\frac1w\right)\right|
 \leq\sup_{t\leq T}\|g_t^{r_j}-g_t^{r_k}\|_{C(\T)}
 \longrightarrow0.
\]
The uniform limit is holomorphic in $\mathbb D$ and vanishes at zero;
hence the limiting exterior map has the required normalization
$g_t(z)=z+O(z^{-1})$ at infinity.
Each $g_t^{r_j}$ is univalent on $\mathbb D^*$. The locally uniform
limit theorem for univalent functions and the normalization
$g_t(z)=z+O(1/z)$ show that $g_t$ is univalent on $\mathbb D^*$.
The lower estimate passes through the same uniform limit:
\[
 \begin{aligned}
 |g_t(\zeta)-g_t(\xi)|
 &=\lim_{j\to\infty}|g_t^{r_j}(\zeta)-g_t^{r_j}(\xi)|\\
 &\geq\lim_{j\to\infty}H_{r_j}^{-1}
             |g_0^{(j)}(\zeta)-g_0^{(j)}(\xi)|^{3/2}\\
 &=H^{-1}|g_0(\zeta)-g_0(\xi)|^{3/2}>0
 \qquad(\zeta\ne\xi).
 \end{aligned}
\]
Here $g_0$ has the original injective boundary-map law. The estimate
holds simultaneously for all $t\leq T$ and all distinct boundary
points. The argument principle therefore makes $g_t$ injective on
$\overline{\mathbb D^*}$.
Let $\omega_{g,T}$ be the uniform spatial modulus of
$(t,\zeta)\mapsto g_t(\zeta)$ on $[0,T]\times\T$. The boundary
matching identities give
\begin{align*}
 &\sup_{t\leq T}\|f_t^{r_j}-g_t\circ X_t\|_{C(\T)}\\
 &\quad\leq\sup_{t\leq T}\|g_t^{r_j}-g_t\|_{C(\T)}
       +\omega_{g,T}\!\left(\sup_{t\leq T}
                            \|h_t^{r_j}-X_t\|_\infty\right)
       \longrightarrow0.
\end{align*}
The maximum principle now makes $(t,z)\mapsto f_t^{r_j}(z)$ Cauchy
uniformly on $[0,T]\times\overline{\mathbb D}$. Its limit $f_t$ is
holomorphic and satisfies $f_t=g_t\circ X_t$ on $\T$. Each $f_t$ is
nonconstant because its boundary values form a homeomorphism onto a
nondegenerate Jordan curve. The nonconstant
version of the univalent limit theorem gives univalence in $\mathbb D$.
Moreover, the argument principle applied to the boundary Jordan curve
shows that $f_t(\mathbb D)$ is its bounded complementary component.
Hence $f_t(\mathbb D)$ is disjoint from $f_t(\T)$, and the boundary
homeomorphism makes $f_t$ injective on $\overline{\mathbb D}$.
For $s,t\in[0,T]$, the maximum principle also gives
\begin{align*}
 \|f_t-f_s\|_{C(\overline{\mathbb D})}
 &\leq\|g_t-g_s\|_{C(\T)}
       +\omega_{g,T}(\|X_t-X_s\|_\infty),\\
 \shortintertext{and, with the continuous value at $w=0$,}
 \sup_{|w|\leq1}
 \left|\left(g_t(1/w)-\frac1w\right)
       -\left(g_s(1/w)-\frac1w\right)\right|
 &\leq\|g_t-g_s\|_{C(\T)}.
\end{align*}
Both right-hand sides tend to zero as $s\to t$. These are the two
disk-algebra coordinates in the topology of $Y$, so the lift is
continuous in $Y$.

At each deterministic time, $X_t$ is $\nu$-distributed and its
removable welding is unique. Consequently, for each fixed $t$,
\[
 Y_t=\iota(X_t)\quad\text{almost surely}.
\]
It follows that
\[
 \operatorname{Law}(Y_t)=\iota_*\nu=\nu^Y.
\]

For $\ell\geq0$, shift the local lift law to
$I_\ell=[\ell T,(\ell+1)T]$ and denote it by $Q_\ell^Y$ on
$C(I_\ell,Y)$. By stationarity of the response diffusion, its projected
law is $Q_\ell^X=\operatorname{Law}(X|_{I_\ell})$. These path spaces
admit complete separable metrics. Disintegration therefore gives a
measurable probability kernel $\Pi_\ell(\omega,dy)$, with
$\omega\in C(I_\ell,X)$ and $y\in C(I_\ell,Y)$, such that
\[
 Q_\ell^Y(dy)=\int\Pi_\ell(\omega,dy)\,Q_\ell^X(d\omega).
\]
For $Q_\ell^X$-almost every $\omega$,
\begin{equation}\label{eq:block-kernel-support}
 \Pi_\ell(\omega,\{y:\prX\circ y=\omega\})=1.
\end{equation}
The fixed-time identity at both endpoints also gives, for
$\Pi_\ell(\omega,\cdot)$-almost every $y$,
\begin{align*}
 y(\ell T)&=\iota(\omega(\ell T)),\\
 \shortintertext{and}
 y((\ell+1)T)&=\iota(\omega((\ell+1)T)).
\end{align*}
Conditional on a full response path $X$, the measurable kernels define
the countable product probability
\[
 \bigotimes_{\ell=0}^{\infty}\Pi_\ell(X|_{I_\ell},dy_\ell)
 \quad\text{on }\prod_{\ell\geq0}C(I_\ell,Y)
\]
by the Ionescu--Tulcea construction. The endpoint identities give,
on a common event of probability one,
\[
 y_\ell((\ell+1)T)=\iota(X_{(\ell+1)T})
                 =y_{\ell+1}((\ell+1)T)
 \qquad(\ell\geq0).
\]
Thus $Y_t=y_\ell(t)$ for $t\in I_\ell$ defines a continuous path and
$\prX(Y_t)=X_t$ for every $t\geq0$. Integrate the conditional product
against the response path law. For every measurable
$A\subset C(I_\ell,Y)$, the resulting block marginal is $Q_\ell^Y$, since
\[
 \mathbb E\bigl[\Pi_\ell(X|_{I_\ell},A)\bigr]
 =\int\Pi_\ell(\omega,A)\,Q_\ell^X(d\omega)=Q_\ell^Y(A).
\]
Thus the deterministic-time marginals and the estimates on each block
are preserved by concatenation.
\end{proof}

Use the lift and its joint probability law from
Proposition~\ref{prop:continuous-geometric-lift}.

\begin{corollary}\label{cor:geometric-lift-estimates}
For every $\zeta\in\T$ and $0\leq t\leq T_*$,
\begin{equation}\label{eq:exterior-small-time-energy}
 \mathbb E|g_t(\zeta)-g_0(\zeta)|^2\leq C(t+t^2),
\end{equation}
where $C$ depends only on $\kappa$. On one event of probability one,
for every $0<\beta\leq1$,
\begin{equation}\label{eq:exterior-Holder-propagation}
 [g_0]_\beta<\infty
 \quad\Longrightarrow\quad
 \sup_{0\leq t\leq T_*}[g_t]_{\beta/2}<\infty.
\end{equation}
\end{corollary}

\begin{proof}
For the approximations in the preceding proof,
\[
 \begin{aligned}
 \mathbb E|g_t^r(\zeta)-g_0(\zeta)|^2
 &\leq2\mathbb E|M_t^r(g_0(\zeta))|^2\\
 &\quad+2t\mathbb E\int_0^t
       |V_s^r(Z_s^r(g_0(\zeta)))|^2\,ds
 \leq2Ct+2C^2t^2.
 \end{aligned}
\]
Uniform boundary convergence on the Skorokhod representation and
Fatou's lemma give \eqref{eq:exterior-small-time-energy}.
The upper bound in \eqref{eq:planar-biHolder} with $\eta=1/2$ passes
to the same limit. Hence, for a finite random $H$,
\[
 \begin{aligned}
 |g_t(\zeta)-g_t(\xi)|
 &\leq H|g_0(\zeta)-g_0(\xi)|^{1/2}\\
 &\leq H[g_0]_\beta^{1/2}|\zeta-\xi|^{\beta/2},
 \qquad 0\leq t\leq T_*.
 \end{aligned}
\]
This is \eqref{eq:exterior-Holder-propagation}.
\end{proof}

\subsection{From geometric tightness to a compact form nest}

Probabilities here refer to the lift just constructed; its first
coordinate has the stationary response law. For an open or compact
set $A\subset X$, let $u_A^{\rm eq}$ be its $1$-equilibrium potential and put
\[
 \sigma_A=\inf\{t\geq0:X_t\in A\},\qquad e^{-\infty}=0.
\]
This convention includes the initial point.

\begin{lemma}\label{lem:stationary-capacity}
For every open $O\subset X$,
\begin{equation}\label{eq:stationary-capacity-identity}
 \CapOne(O)=\int u_O^{\rm eq}\,d\nu
        =\mathbb E e^{-\sigma_O}.
\end{equation}
Consequently, for every $T>0$,
\[
 \CapOne(O)\leq\mathbb P(\sigma_O<T)+e^{-T}.
\]
For every compact $K\subset X$,
\begin{equation}\label{eq:stationary-capacity-compact}
 \CapOne(K)=\int_X\mathbb E_h[e^{-\sigma_K}]\,d\nu(h).
\end{equation}
In particular,
\[
 \mathbb P(\sigma_K<\infty)=0
 \quad\Longrightarrow\quad \CapOne(K)=0.
\]
\end{lemma}

\begin{proof}
For either an open or a compact set $A$, the potential $u_A^{\rm eq}$ equals $1$
quasi-everywhere on $A$. The perturbations $u_A^{\rm eq}+s(\mathbf 1-u_A^{\rm eq})$ are
admissible in the equilibrium variational problem for every real $s$.
Consequently,
\[
 0=\left.\frac12\frac{d}{ds}\right|_{s=0}
      \|u_A^{\rm eq}+s(\mathbf 1-u_A^{\rm eq})\|_{\E_1}^2
   =\E_1(u_A^{\rm eq},\mathbf 1-u_A^{\rm eq}).
\]
Since $\E(u_A^{\rm eq},\mathbf 1)=0$, it follows that
\[
 \CapOne(A)=\E_1(u_A^{\rm eq},u_A^{\rm eq})=\E_1(u_A^{\rm eq},\mathbf 1)=\int_X u_A^{\rm eq}\,d\nu.
\]
The equilibrium-potential representation
\cite[Lemma~4.2.1 and Theorem~4.2.5]{FOT}, transferred from the compact
regular realization of Proposition~\ref{prop:compactification}, gives
\[
 u_A^{\rm eq}(h)=\mathbb E_h[e^{-\sigma_A}]
       \quad\text{for quasi-every }h.
\]
Integrating against $\nu$ proves both capacity identities. For an
open set $O$, splitting at $T$ gives
\[
 \CapOne(O)=\mathbb E[e^{-\sigma_O}]
 \leq\mathbb P(\sigma_O<T)
       +e^{-T}\mathbb P(\sigma_O\geq T)
 \leq\mathbb P(\sigma_O<T)+e^{-T}.
\]
For compact $K$, if $\mathbb P(\sigma_K<\infty)=0$, then
$\mathbb E e^{-\sigma_K}=0$, so the compact-set identity gives
$\CapOne(K)=0$.
\end{proof}

For closed $K\subset X$ and $C\subset Y$, write
\begin{align*}
 \Dcal_K&=\{F\in\Dcal:F=0\ \nu\text{-a.e. on }X\setminus K\},\\
 \shortintertext{and}
 \Dcal^Y_C&=\{G\in\Dcal^Y:G=0\ \nu^Y\text{-a.e. on }Y\setminus C\}.
\end{align*}

\begin{lemma}\label{lem:Jordan-compact-nest}
There is an increasing compact $\E^Y$-nest $(C_m)$ in $Y$ such that
\begin{equation}\label{eq:projected-compact-capacity-main}
 \CapOne(X\setminus\prX(C_m))\longrightarrow0.
\end{equation}
\end{lemma}
\begin{proof}
By Proposition~\ref{prop:continuous-geometric-lift}, the law of
$(Y_t)_{0\leq t\leq m}$ is a probability on $C([0,m],Y)$. Choose a
compact set $A_m\subset C([0,m],Y)$ such that
\[
 \mathbb P\bigl((Y_t)_{0\leq t\leq m}\notin A_m\bigr)\leq2^{-m},
\]
and put
\[
 C_m=\bigcup_{j=1}^m
       \{\omega(t):\omega\in A_j,\ 0\leq t\leq j\}.
\]
Continuity of evaluation makes $C_m$ compact, and $(C_m)$ is
increasing. Since $\prX(Y_t)=X_t$ for all $t$, we have
\[
 \begin{aligned}
 \{\sigma_{X\setminus\prX(C_m)}<m\}
 &\subset\{Y_t\notin C_m\text{ for some }0\leq t\leq m\}\\
 &\subset\{Y|_{[0,m]}\notin A_m\}.
 \end{aligned}
\]
The last event has probability at most $2^{-m}$, and $\prX(C_m)$ is
compact. Lemma~\ref{lem:stationary-capacity} therefore gives
\begin{equation}\label{eq:projected-compact-capacity}
 \CapOne(X\setminus \prX(C_m))\leq2^{-m}+e^{-m}\longrightarrow0.
\end{equation}
For $u_m^{\rm eq}=u_{X\setminus \prX(C_m)}^{\rm eq}$ and $F\in\Ccyl$, we have
$0\leq u_m^{\rm eq}\leq1$.  The product rule and
\eqref{eq:cylinder-energy-bound} therefore give
\[
 \begin{aligned}
 \|Fu_m^{\rm eq}\|_{\E_1}^2
 &\leq
 \bigl(\|F\|_\infty^2+2\|\Gamma(F)\|_\infty\bigr)\|u_m^{\rm eq}\|_2^2
   +2\|F\|_\infty^2\E(u_m^{\rm eq},u_m^{\rm eq})\\
 &\leq C_F\|u_m^{\rm eq}\|_{\E_1}^2
  =C_F\CapOne(X\setminus \prX(C_m))\longrightarrow0.
 \end{aligned}
\]
For $F\in\Ccyl$, we have
\begin{align*}
 F(1-u_m^{\rm eq})&\in\Dcal_{\prX(C_m)},\\
 \shortintertext{and}
 \|F-F(1-u_m^{\rm eq})\|_{\E_1}&=\|Fu_m^{\rm eq}\|_{\E_1}\longrightarrow0.
\end{align*}
The form density of $\Ccyl$ therefore gives
\[
 \overline{\bigcup_m\Dcal_{\prX(C_m)}}^{\|\cdot\|_{\E_1}}=\Dcal.
\]

For almost every $h$, its removable welding is unique among
\emph{all} Jordan weldings in the infinity normalization.
Consequently
\[
 \mathbf 1_{C_m}(\iota(h))=\mathbf 1_{\prX(C_m)}(h)
          \quad\text{for }\nu\text{-almost every }h.
\]
Consequently,
\[
 U\Dcal_{\prX(C_m)}\subset\Dcal^{Y}_{C_m}
 \qquad(m\geq1).
\]
Since $U$ preserves the form norm,
\[
 \overline{\bigcup_mU\Dcal_{\prX(C_m)}}^{\|\cdot\|_{\E_1^Y}}
 =U\Dcal=\Dcal^Y.
\]
The support inclusion above proves that $(C_m)$ is a compact
$\E^Y$-nest.
\end{proof}

Retain the countable form-dense point-separating family
$(F_j)_{j\geq1}\subset\Ccyl$ fixed in
Section~\ref{sec:compactification}. Fix a dense sequence
$(\zeta_j)_{j\geq1}$ in $\T$, and define on $Y$
\[
 G_{j,1}(h,f,g)=\arctan(\operatorname{Re}g(\zeta_j)),\qquad
 G_{j,2}(h,f,g)=\arctan(\operatorname{Im}g(\zeta_j)).
\]
For $\ell\in\{1,2\}$, use the representative
$H_{j,\ell}=G_{j,\ell}\circ\iota$ on $X_{\rm nf}$, extended by zero off $X_{\rm nf}$.
It represents $U^{-1}G_{j,\ell}$ in $L^2(\nu)$.

\begin{lemma}\label{lem:Jordan-separating-coordinates}
The functions $UF_j$ and $G_{j,\ell}$ belong to
$\Dcal^Y\cap C_b(Y)$ and together separate the points of $Y$.
Moreover,
\[
 \overline{\{UF_j:j\geq1\}}^{\|\cdot\|_{\E_1^Y}}=\Dcal^Y.
\]
\end{lemma}

\begin{proof}
The functions $G_{j,\ell}$ are bounded and continuous on $Y$.
At each of the deterministic times $s=0,t$, the lift identity gives
$G_{j,\ell}(Y_s)=H_{j,\ell}(X_s)$ almost surely. Stationarity and
symmetry of the response process, the $1$-Lipschitz property of
$\arctan$, and \eqref{eq:exterior-small-time-energy} therefore give,
for $0<t\leq T_*$,
\[
 \begin{aligned}
 \frac{\langle H_{j,\ell}-T_tH_{j,\ell},H_{j,\ell}\rangle_2}{t}
 &=\frac1{2t}\mathbb E
       |H_{j,\ell}(X_t)-H_{j,\ell}(X_0)|^2\\
 &\leq\frac1{2t}\mathbb E
       |g_t(\zeta_j)-g_0(\zeta_j)|^2
 \leq C(1+t).
 \end{aligned}
\]
Let $\mu_{j,\ell}^{\rm sp}$ be the spectral measure of $H_{j,\ell}$ for
$-L$. Since $(1-e^{-t\lambda})/t\uparrow\lambda$ as $t\downarrow0$,
\[
 \begin{aligned}
 \int_{[0,\infty)}\lambda\,d\mu_{j,\ell}^{\rm sp}(\lambda)
 &=\lim_{t\downarrow0}\int_{[0,\infty)}
          \frac{1-e^{-t\lambda}}{t}\,d\mu_{j,\ell}^{\rm sp}(\lambda)\\
 &=\lim_{t\downarrow0}
   \frac{\langle H_{j,\ell}-T_tH_{j,\ell},H_{j,\ell}\rangle_2}{t}
 \leq C.
 \end{aligned}
\]
Thus $H_{j,\ell}\in\Dcal$ and
$\E(H_{j,\ell},H_{j,\ell})=\int\lambda\,d\mu_{j,\ell}^{\rm sp}(\lambda)<\infty$.
Unitarity of $U$ and form density of $(F_j)$ prove the asserted
density of $(UF_j)$. For $y_i=(h_i,f_i,g_i)$, $i=1,2$,
\begin{align*}
 (UF_j)(y_1)=(UF_j)(y_2)\quad(j\geq1)
 &\ \Longrightarrow\ h_1=h_2,\\
 \shortintertext{while}
 G_{j,\ell}(y_1)=G_{j,\ell}(y_2)
 \quad(j\geq1,\ \ell\in\{1,2\})
 &\ \Longrightarrow\ g_1|_\T=g_2|_\T.
\end{align*}
For the second implication use injectivity of $\arctan$ and density
of $(\zeta_j)$. Exterior normalization and the maximum principle
give $g_1=g_2$. The matching relation and the maximum principle then
give $f_1=f_2$, proving separation.
\end{proof}

\subsection{The transported form and its projection}

Lemma~\ref{lem:Polish-Jordan-triples} provides a complete separable
metric topology on $Y$. Together with the compact form nest from
Lemma~\ref{lem:Jordan-compact-nest} and the countable continuous
form-dense separating family from
Lemma~\ref{lem:Jordan-separating-coordinates}, the three
quasi-regularity conditions hold for $(\E^Y,\Dcal^Y)$.

Let
\[
 T_t^Y=UT_tU^{-1}\qquad(t\geq0),
 \qquad L^Y=ULU^{-1},
\]
where $L^Y$ is the generator of the transported
$L^2(\nu^Y)$-semigroup. Write $\operatorname{Cap}_1^Y$ for the
corresponding $1$-capacity.

\begin{lemma}\label{lem:Jordan-form-properties}
The transported form $(\E^Y,\Dcal^Y)$ is strongly local. Its generator
satisfies
\[
 \mathbf 1\in\Dom(L^Y),\qquad L^Y\mathbf 1=0,
\]
and hence $T_t^Y\mathbf 1=\mathbf 1$ for every $t\geq0$.
\end{lemma}

\begin{proof}
The closed product and chain rules transfer under $U$. In particular,
let $F,G\in\Dcal$ and put
\[
 F^Y=UF,\qquad G^Y=UG.
\]
For every real $c$, $F^Y=F\circ\prX$ and
\eqref{eq:levelset-locality} give
\[
 \mathbf 1_{\{F^Y=c\}}\bigl((DF)\circ\prX\bigr)
 =\bigl(\mathbf 1_{\{F=c\}}DF\bigr)\circ\prX=0
 \quad\nu^Y\text{-almost everywhere}.
\]
Suppose that $O\subset Y$ is open and, for some $c\in\mathbb R$,
\[
 F^Y=c\quad\nu^Y\text{-a.e. on }O,
 \qquad
 G^Y=0\quad\nu^Y\text{-a.e. on }Y\setminus O.
\]
The same identity, first for $F$ and then for $G$ at level $0$, yields
\[
 \mathbf 1_O\bigl((DF)\circ\prX\bigr)=0,
 \qquad
 \mathbf 1_{Y\setminus O}\bigl((DG)\circ\prX\bigr)=0,
\]
and therefore
\[
 \E^Y(F^Y,G^Y)
 =\int_O(DF,DG)_{\WP}\circ\prX\,d\nu^Y
  +\int_{Y\setminus O}(DF,DG)_{\WP}\circ\prX\,d\nu^Y=0.
\]
Thus the form is strongly local. Since $U\mathbf 1=\mathbf 1$ and
$\E^Y(\mathbf 1,F^Y)=0$ for every $F^Y\in\Dcal^Y$,
\[
 \mathbf 1\in\Dom(L^Y),\qquad L^Y\mathbf 1=0.
\]
Consequently $T_t^Y\mathbf 1=\mathbf 1$, so the form is conservative.
\end{proof}

\begin{lemma}\label{lem:projection-quasi-homeomorphism}
There is a compact $\E^Y$-nest $(C'_m)$ with $C'_m\subset C_m$
such that $(\prX(C'_m))$ is an $\E$-nest and
\[
 \prX:C'_m\longrightarrow\prX(C'_m)
 \quad\text{is a homeomorphism for every }m\geq1.
\]
Thus $\prX$ is a quasi-homeomorphism.
\end{lemma}

\begin{proof}
For every open $O\subset X$, the unitary map $U$ preserves the
constraint in the defining variational problem, so
\[
 \begin{aligned}
 \operatorname{Cap}_1^Y(\prX^{-1}O)
 &=\inf_{\substack{F\in\Dcal\\F\geq1\ \nu\text{-a.e. on }O}}
                        \|UF\|_{\E^Y_1}^2\\
 &=\CapOne(O).
 \end{aligned}
\]
Since $\prX$ is continuous, outer approximation gives
\[
 \operatorname{Cap}_1^Y(\prX^{-1}A)\leq\CapOne(A)
                       \qquad(A\subset X).
\]
Let $\widetilde H$ be quasi-continuous on $X$. For every
$\varepsilon>0$, choose an open $O\subset X$ such that
$\CapOne(O)<\varepsilon$ and $\widetilde H|_{X\setminus O}$ is
continuous. Then
\[
 \operatorname{Cap}_1^Y(\prX^{-1}O)
 \leq\CapOne(O)<\varepsilon.
\]
Moreover, $\widetilde H\circ\prX$ is continuous on
$Y\setminus\prX^{-1}O$. Thus pullback by $\prX$ preserves
quasi-continuity.

Let $\widetilde H_{j,\ell}$ be a quasi-continuous version of
$H_{j,\ell}$ on $X$. Then $\widetilde H_{j,\ell}\circ\prX$ and
$G_{j,\ell}$ are quasi-continuous versions of the same domain function
$UH_{j,\ell}$. Define
\[
 E_{\rm qc}
 =\bigcup_{\substack{j\geq1\\\ell\in\{1,2\}}}
   \{y\in Y:G_{j,\ell}(y)\ne\widetilde H_{j,\ell}(\prX(y))\}.
\]
Since two quasi-continuous versions of the same form-domain function
agree outside a set of zero capacity,
\[
 \operatorname{Cap}_1^Y(E_{\rm qc})=0.
\]
Choose a compact $\E^Y$-nest $(\widetilde C_m)$ contained in
$Y\setminus E_{\rm qc}$ and set
\[
 C'_m=C_m\cap\widetilde C_m\qquad(m\geq1).
\]
Then $(C'_m)$ is a compact $\E^Y$-nest and $C'_m\cap E_{\rm qc}=\varnothing$.

For $y,y'\in\bigcup_mC'_m$ with $\prX(y)=\prX(y')=h$,
\begin{align*}
 (UF_j)(y)&=F_j(h)=(UF_j)(y'),\\
 \shortintertext{and}
 G_{j,\ell}(y)&=\widetilde H_{j,\ell}(h)=G_{j,\ell}(y')
 \qquad(j\geq1,\ \ell\in\{1,2\}).
\end{align*}
The separation proved in Lemma~\ref{lem:Jordan-separating-coordinates}
gives $y=y'$. Each continuous bijection
$\prX:C'_m\to\prX(C'_m)$ is therefore a homeomorphism by compactness.
The support subspaces satisfy
\[
 U^{-1}\Dcal^Y_{C'_m}\subset\Dcal_{\prX(C'_m)}.
\]
Unitarity and the nest property give
\[
 \overline{\bigcup_mU^{-1}\Dcal^Y_{C'_m}}^{\|\cdot\|_{\E_1}}
 =U^{-1}\Dcal^Y=\Dcal.
\]
It follows from the support inclusion that $(\prX(C'_m))$ is a compact
$\E$-nest. These two nests prove the quasi-homeomorphism assertion.
\end{proof}

\subsection{Final proof of Theorem~\ref{thm:full-main}}
\label{sec:proof-construction}

The quasi-regular form established above is strongly local and
conservative by Lemma~\ref{lem:Jordan-form-properties}. The form--process
correspondence \cite[Section~2]{AMR} therefore gives a properly
associated diffusion $Y^{\mathrm{form}}$, with semigroup
\[
 T_t^Y=UT_tU^{-1}.
\]
Outside a properly exceptional set, its paths are continuous in $Y$
and its lifetime is infinite. The associated process is unique up to
the usual equivalence outside properly exceptional sets.

Unitarity of $U$, symmetry of $T_t$, and $T_t^Y\mathbf 1=\mathbf 1$ give, for
$F,G\in L^2(\nu^Y)$,
\begin{align*}
 \langle T_t^YF,G\rangle_{L^2(\nu^Y)}
 &=\langle T_tU^{-1}F,U^{-1}G\rangle_{L^2(\nu)}\\
 &=\langle F,T_t^YG\rangle_{L^2(\nu^Y)},\\
 \shortintertext{and consequently}
 \int_YT_t^YF\,d\nu^Y
 &=\langle F,T_t^Y\mathbf 1\rangle_{L^2(\nu^Y)}
 =\int_YF\,d\nu^Y.
\end{align*}
Thus $\nu^Y$ is invariant and the diffusion started from $\nu^Y$ is
stationary. These identities also prove
\eqref{eq:intro-Jordan-reversibility}.

Let $Y^{\mathrm{geom}}$ be the continuous process constructed in
Proposition~\ref{prop:continuous-geometric-lift}. We first identify its
law on $C([0,\infty),Y)$ with that of $Y^{\mathrm{form}}$ started from
$\nu^Y$. It suffices to treat $0=t_0<t_1<\cdots<t_m$:
for a family starting at a positive time, insert the constant test
function $\mathbf 1$ at time zero. Take bounded measurable representatives
$\varphi_0^Y,\ldots,\varphi_m^Y$ on $Y$. Let
$\varphi_j=\varphi_j^Y\circ\iota$ on $X_{\rm nf}$, extended by zero off $X_{\rm nf}$;
this is a bounded representative of $U^{-1}\varphi_j^Y$.
The construction of $Y^{\mathrm{geom}}$ gives
$Y_t^{\rm geom}=\iota(X_t)$ almost surely at each deterministic time.
Taking the finite intersection over $t_0,\ldots,t_m$ gives
\[
 \mathbb E\prod_{j=0}^m \varphi_j^Y(Y_{t_j}^{\rm geom})
 =\mathbb E\prod_{j=0}^m \varphi_j(X_{t_j}).
\]
Because $U$ is composition with $\prX$, for bounded $F,G$ it preserves products:
\[
 U(FG)=(UF)(UG),\qquad U\mathbf 1=\mathbf 1.
\]
Writing $s_j=t_j-t_{j-1}$, the Markov property of $X$ and the identity
$T_t^Y=UT_tU^{-1}$ give
\begin{align*}
 \mathbb E\prod_{j=0}^m \varphi_j(X_{t_j})
 &=\int_X \varphi_0\,
 T_{s_1}\!\left(\varphi_1T_{s_2}\!\left(\cdots
 \varphi_{m-1}T_{s_m}\varphi_m\right)\right)d\nu
 \\
 \shortintertext{and hence}
 &=\int_Y \varphi_0^Y\,
 T_{s_1}^Y\!\left(\varphi_1^YT_{s_2}^Y\!\left(\cdots
 \varphi_{m-1}^YT_{s_m}^Y\varphi_m^Y\right)\right)d\nu^Y.
\end{align*}
The last expression is the corresponding product moment of
$Y^{\mathrm{form}}$ started from $\nu^Y$. Thus $Y^{\mathrm{geom}}$ and $Y^{\mathrm{form}}$ have the
same finite-dimensional distributions. Both laws are supported on
$C([0,\infty),Y)$, whose Borel sigma-algebra is generated by the
coordinate maps at rational times, so their path laws agree.

It remains to identify the projection for arbitrary admissible starting
points. Let $P_t^X$ and $P_t^Y$ be properly associated transition
kernels, and choose a properly exceptional set $N_X$ for the
$X$-diffusion. The capacity comparison in
Lemma~\ref{lem:projection-quasi-homeomorphism} gives
\[
 \operatorname{Cap}_1^Y(\prX^{-1}N_X)
 \leq\CapOne(N_X)=0.
\]
We include $\prX^{-1}N_X$ in the exceptional set removed below, so that
every remaining $y$ projects to an admissible starting point in $X$.
For $j\geq1$ and rational $q>0$, the semigroup identity
gives
\[
 T_q^Y(UF_j)=U(T_qF_j)\qquad\nu^Y\text{-almost everywhere}.
\]
Proper association and preservation of quasi-continuity under
pullback show that
\begin{align*}
 y&\longmapsto\int_YF_j(\prX(y'))\,P_q^Y(y,dy'),\\
 \shortintertext{and}
 y&\longmapsto\int_XF_j(x')\,P_q^X(\prX(y),dx')
\end{align*}
are quasi-continuous versions of the same $L^2(\nu^Y)$ function.
They therefore agree outside a set of zero $\operatorname{Cap}_1^Y$.
Take the countable union over $j\geq1$ and rational $q>0$, include
$\prX^{-1}N_X$ and the exceptional set of the associated diffusion,
and enclose the union in one properly exceptional set $N_Y$.
For every $y\in Y\setminus N_Y$, the resulting identities are
\[
 \int_YF_j(\prX(y'))\,P_q^Y(y,dy')
 =\int_XF_j(x')\,P_q^X(\prX(y),dx')
 \qquad(j\geq1,\ q\in\mathbb Q,\ q>0).
\]
Because the algebra enumerated by $(F_j)$ determines probability
measures on $X$, these equalities identify the pushed-forward transition
kernel at every rational time. Fix $t\geq0$ and choose rational
$q_n\downarrow t$. For each $j\geq1$, continuity of both diffusions and
bounded convergence yield
\begin{align*}
 \int_YF_j(\prX(y'))\,P_t^Y(y,dy')
 &=\lim_{n\to\infty}
   \int_YF_j(\prX(y'))\,P_{q_n}^Y(y,dy')\\
 \shortintertext{and hence}
 &=\lim_{n\to\infty}
   \int_XF_j(x')\,P_{q_n}^X(\prX(y),dx')\\
 &=\int_XF_j(x')\,P_t^X(\prX(y),dx').
\end{align*}
A second monotone-class argument therefore gives
\begin{equation}\label{eq:projection-kernel-identity}
 (\prX)_*P_t^Y(y,\cdot)=P_t^X(\prX(y),\cdot),
 \qquad t\geq0,
\end{equation}
for every $y\in Y\setminus N_Y$.

Write $\mathbb E_y^Y$ for expectation for $Y^{\rm form}$ started at
$y$, and retain $\mathbb E_h$ for the response diffusion started at $h$.
For a bounded measurable function $a$ on $X$, use the kernel notation
\[
 (P_s^Xa)(h)=\int_X a(h')\,P_s^X(h,dh'),
\]
and analogously on $Y$. Equation~\eqref{eq:projection-kernel-identity}
then says
\[
 P_s^Y(a\circ\prX)(y)=(P_s^Xa)(\prX(y))
 \qquad(y\notin N_Y,\ s\geq0).
\]
Because $N_Y$ is properly exceptional,
\[
 P_s^Y(y,Y\setminus N_Y)=1\qquad(y\notin N_Y,\ s\geq0),
\]
so this identity remains valid at every intermediate state in the
Markov iteration. For $0\leq t_1<\cdots<t_m$ and bounded measurable
$\varphi_1,\ldots,\varphi_m$ on $X$, define the backward test functions by
\[
 a_m=\varphi_m,\qquad
 a_j=\varphi_j\,P_{t_{j+1}-t_j}^Xa_{j+1}
 \quad(j=m-1,\ldots,1).
\]
At each step, on $Y\setminus N_Y$,
\[
 (\varphi_j\circ\prX)P_{t_{j+1}-t_j}^Y(a_{j+1}\circ\prX)
 =(\varphi_j P_{t_{j+1}-t_j}^Xa_{j+1})\circ\prX
 =a_j\circ\prX.
\]
Backward conditioning therefore gives
\begin{align*}
 \mathbb E_y^Y\prod_{j=1}^m\varphi_j(\prX(Y_{t_j}^{\rm form}))
 &=P_{t_1}^Y(a_1\circ\prX)(y)\\
 &=P_{t_1}^Xa_1(\prX(y))\\
 &=\mathbb E_{\prX(y)}\prod_{j=1}^m\varphi_j(X_{t_j}).
\end{align*}
These product moments determine the finite-dimensional distributions.
The two path laws are supported on continuous paths, so
\[
 \bigl(\prX(Y_t^{\rm form})\bigr)_{t\geq0}
 \stackrel{\mathrm{law}}{=}
 (X_t)_{t\geq0}
 \quad\text{started from }\prX(y),
\]
which proves \eqref{eq:intro-Jordan-projection}.

Taking $Y_t=Y_t^{\rm form}$ gives the continuous Jordan lift,
its projection law for every admissible starting point, and the
invariance and reversibility identities proved at the beginning.
This completes the proof of Theorem~\ref{thm:full-main}.

\section{Subcritical removability}
\label{sec:removable-support}

\subsection{Static H\"older regularity}

Fix $0<\kappa<4$. We combine a static H\"older estimate with
\eqref{eq:exterior-Holder-propagation} to prove
Theorem~\ref{thm:removability-main}. The boundary seminorm $[g]_\beta$
is the one used in Section~\ref{sec:Jordan-realization}.

Let $\mu_z^1$ and $\mu^0$ be the rooted and unrooted loop measures of
\cite{Zhan}, and write $\mu_{\rm root}=\mu_0^1$. For a sampled bounded
Jordan loop $\gamma$, let $m_\gamma$ be its natural-length measure and put
$L_\gamma=m_\gamma(\gamma)$. We use the full-measure set on which
$0<L_\gamma<\infty$, parametrize rooted loops by natural length, and
extend the parametrization periodically with period $L_\gamma$.
Let $q_\kappa>0$ be the deterministic exponent in
\cite[Theorem~5.2]{RS}, and set
\[
 \beta_\kappa=\min(q_\kappa,1/2).
\]

\begin{lemma}\label{lem:rooted-loop-Holder-charts}
For $\mu_{\rm root}$-almost every $\gamma$, each $x\in\gamma$ and
either component $\Omega$ of $\widehat{\mathbb C}\setminus\gamma$
admit a half-disc $Q\subset\mathbb H$, with $\operatorname{diam}Q\leq1$,
and a conformal map $F:Q\to\Omega\cap\mathbb C$ whose open straight
edge maps homeomorphically onto an open arc of $\gamma$ containing $x$.
The map extends continuously and injectively to $\overline Q$, and
\[
 |F(z)-F(w)|\leq C_Q|z-w|^{\beta_\kappa},
 \qquad z,w\in\overline Q,
\]
for some $C_Q<\infty$.
\end{lemma}

\begin{proof}
Fix a deterministic $\tau>0$. On $\{\tau<L_\gamma\}$, put
\[
 \mathcal A_\tau=\sigma(\gamma(s):0\leq s\leq\tau),
 \qquad
 \Omega_\tau=\widehat{\mathbb C}\setminus\gamma[0,\tau].
\]
By \cite[Theorem~4.1(vi)]{Zhan},
\[
 \mu_{\rm root}\{L_\gamma>\tau\}<\infty.
\]
For bounded nonnegative $\mathcal A_\tau$-measurable $A$ and bounded
nonnegative future functionals $G$, the conformal Markov property
\cite[Theorem~4.1(ii)]{Zhan} gives
\begin{align}
 &\int_{\{\tau<L_\gamma\}}A\,
       G\bigl(\gamma[\tau,L_\gamma]\bigr)\,d\mu_{\rm root}
 \notag\\
 &\quad=\int_{\{\tau<L_\gamma\}}A\,
   \mathbb E_{\Omega_\tau;\gamma(\tau),\gamma(0)}^{\operatorname{SLE}_\kappa}
       [G]\,d\mu_{\rm root}.
 \label{eq:rooted-loop-disintegration}
\end{align}
Choose, measurably from the past, the conformal map
\[
 \phi_\tau:(\Omega_\tau;\gamma(\tau),\gamma(0))
       \longrightarrow(\mathbb H;0,\infty)
\]
with $|\phi_\tau(\infty)|=1$. Denote the capacity-parametrized future
curve by $\eta_\tau$. Its continuous strictly increasing change from
natural time to capacity time satisfies
\[
 s_\tau:[\tau,L_\gamma)\longrightarrow[0,\infty),\qquad
 s_\tau(\tau)=0,\qquad \lim_{u\uparrow L_\gamma}s_\tau(u)=\infty.
\]
For integers $n\geq1$, let
\[
 \widehat f_{\tau,n}:\mathbb H\longrightarrow
             \mathbb H\setminus\eta_\tau[0,n]
\]
be the shifted inverse Loewner map. Theorem~5.2 of \cite{RS} and its
continuous boundary extension give, conditionally on the past,
\[
 \sup_{\substack{z,w\in\overline{\mathbb H}\cap\overline{B(0,m)}\\z\ne w}}
 \frac{|\widehat f_{\tau,n}(z)-\widehat f_{\tau,n}(w)|}
      {|z-w|^{q_\kappa}}<\infty
 \qquad(n,m\geq1)
\]
almost surely. Define $E_{\tau,n,m}$ to be the event that $L_\gamma>\tau$
and this bound fails. With the past held fixed, its conditional future
probability is zero. Integrating in \eqref{eq:rooted-loop-disintegration}
gives $\mu_{\rm root}(E_{\tau,n,m})=0$. Set
\[
 E^c=\bigcup_{\substack{\tau\in\mathbb Q,\ \tau>0\\n,m\geq1}}
             E_{\tau,n,m},
 \qquad E=(E^c)^c.
\]
Then
\[
 \mu_{\rm root}(E^c)
 \leq\sum_{\substack{\tau\in\mathbb Q,\ \tau>0\\n,m\geq1}}
           \mu_{\rm root}(E_{\tau,n,m})=0.
\]
Take $\gamma\in E$ and rational $0<\tau<a<b<L_\gamma$. For an integer
$n>s_\tau(b)$, the inverse time change gives
\[
 u_n=s_\tau^{-1}(n),\qquad b<u_n<L_\gamma.
\]
Simplicity of the loop gives
\[
 \operatorname{dist}\!\left(
 \gamma([a,b]),\gamma([0,\tau])\cup\gamma([u_n,L_\gamma])
 \right)>0.
\]
At either prime end above a point of
$\eta_\tau[s_\tau(a),s_\tau(b)]$, choose a half-disc $Q$ on which
$\widehat f_{\tau,n}$ extends injectively to $\overline Q$. Shrink $Q$
so that $\operatorname{diam}Q\leq1$ and its image lies in a disk $Q'$
with
\[
 \widehat f_{\tau,n}(\overline Q)\subset Q',
 \qquad
 \overline{Q'}\subset\mathbb H\setminus\{\phi_\tau(\infty)\}.
\]
The positive separation above allows a further shrinking so that
$\phi_\tau^{-1}(\widehat f_{\tau,n}(Q))$ avoids the earlier and terminal
arcs. This connected image lies in the chosen component of the full
loop complement. The exclusion of $\phi_\tau(\infty)$ also gives
\[
 \sup_{u\in Q'}|(\phi_\tau^{-1})'(u)|<\infty.
\]
For $F=\phi_\tau^{-1}\circ\widehat f_{\tau,n}$ and $z,w\in\overline Q$,
convexity of $Q'$ and the inverse-map estimate imply
\begin{align*}
 |F(z)-F(w)|
 &\leq\sup_{u\in Q'}|(\phi_\tau^{-1})'(u)|\,
       |\widehat f_{\tau,n}(z)-\widehat f_{\tau,n}(w)|\\
 &\leq C_Q|z-w|^{q_\kappa}
 \leq C_Q|z-w|^{\beta_\kappa}.
\end{align*}
The last step uses $|z-w|\leq1$ and $\beta_\kappa\leq q_\kappa$.
Compactness gives finitely many charts on each side of each subarc,
and
\[
 \bigcup_{\substack{\tau,a,b\in\mathbb Q\\0<\tau<a<b<L_\gamma}}
       \gamma((a,b))=\gamma((0,L_\gamma))
\]
covers the loop away from its root.

For deterministic $s>0$, let
\[
 (\mathcal T_s\gamma)(t)=\gamma(s+t)-\gamma(s).
\]
Theorem~4.1(iii) of \cite{Zhan} gives
$(\mathcal T_s)_*\mu_{\rm root}=\mu_{\rm root}$. Hence
\[
 E_*=E\cap\bigcap_{m\geq1}\mathcal T_{1/m}^{-1}E,
 \qquad \mu_{\rm root}(E_*^c)=0.
\]
For $\gamma\in E_*$, choose $s=1/m<L_\gamma$. Then
\[
 0<L_\gamma-s<L_\gamma,
 \qquad
 (\mathcal T_s\gamma)(L_\gamma-s)=\gamma(0)-\gamma(s).
\]
A nonroot chart of $\mathcal T_s\gamma$, translated by $\gamma(s)$,
is the required chart at the original root, with the same exponent.
\end{proof}

For a bounded Jordan loop $\gamma$, let $g_\gamma$ be its exterior
uniformization with
\[
 g_\gamma(z)=a_\gamma z+b_\gamma+O(z^{-1}),\qquad a_\gamma>0.
\]
This fixes $g_\gamma$ uniquely. The normalized-map construction and
radial boundary limits give measurable evaluations $g_\gamma(\zeta)$,
$\zeta\in\T$.

\begin{lemma}\label{lem:subcritical-static-Holder}
The exterior boundary map satisfies
\[
 [g]_{\beta_\kappa}<\infty
 \qquad\text{for }\nu^Y\text{-almost every }(h,f,g).
\]
\end{lemma}

\begin{proof}
Take a loop in the full-measure set of
Lemma~\ref{lem:rooted-loop-Holder-charts}. Choose exterior local charts
$\phi_j:Q_j\to\mathbb C$ and open arcs $I_j,I_j^+\subset\T$ such that
\[
 \overline I_j\subset I_j^+,\qquad
 \T=\bigcup_{j=1}^J I_j.
\]
Choose the arcs of angular length at most $\pi$. Each transition map
\[
 R_j=\phi_j^{-1}\circ g_\gamma
\]
is defined on an exterior collar of $I_j^+$ and has real boundary
values there. Schwarz reflection extends it across $I_j^+$, so
\[
 \sup_{\zeta\in\overline I_j}|R_j'(\zeta)|<\infty.
\]
Integration along the shorter circular arc gives
\[
 |R_j(\zeta)-R_j(\xi)|
 \leq\frac\pi2\sup_{\eta\in\overline I_j}|R_j'(\eta)|\,|\zeta-\xi|
 \qquad(\zeta,\xi\in I_j).
\]
Using the local H\"older estimate and taking the maximum over the
finitely many charts, we obtain a finite $C_\gamma$ such that
\begin{align*}
 |g_\gamma(\zeta)-g_\gamma(\xi)|
 &=|\phi_j(R_j(\zeta))-\phi_j(R_j(\xi))|\\
 &\leq C_{Q_j}|R_j(\zeta)-R_j(\xi)|^{\beta_\kappa}
 \leq C_\gamma|\zeta-\xi|^{\beta_\kappa}
 \qquad(\zeta,\xi\in I_j).
\end{align*}
Let $d_*>0$ be a Lebesgue number of the finite cover. For distinct
$\zeta,\xi\in\T$,
\[
 \frac{|g_\gamma(\zeta)-g_\gamma(\xi)|}{|\zeta-\xi|^{\beta_\kappa}}
 \leq
 \begin{cases}
 C_\gamma,&|\zeta-\xi|<d_*,\\
 d_*^{-\beta_\kappa}\operatorname{diam}g_\gamma(\T),
        &|\zeta-\xi|\geq d_*.
 \end{cases}
\]
Thus $[g_\gamma]_{\beta_\kappa}<\infty$.

Define $E_{\rm bad}=\{\gamma:[g_\gamma]_{\beta_\kappa}=\infty\}$.
Continuity gives the countable representation
\[
 [g_\gamma]_{\beta_\kappa}
 =\sup_{\substack{\zeta,\xi\in\T_{\mathbb Q}\\\zeta\ne\xi}}
       \frac{|g_\gamma(\zeta)-g_\gamma(\xi)|}{|\zeta-\xi|^{\beta_\kappa}},
\]
so $E_{\rm bad}$ is measurable. For $a=|a|e^{i\theta}\ne0$ and
$b\in\mathbb C$, uniqueness of the positive-leading-coefficient
normalization gives
\begin{align*}
 g_{a\gamma+b}(z)&=a\,g_\gamma(e^{-i\theta}z)+b,\\
 \shortintertext{and consequently}
 [g_{a\gamma+b}]_{\beta_\kappa}&=|a|[g_\gamma]_{\beta_\kappa}.
\end{align*}
Thus $E_{\rm bad}$ is invariant under affine changes and does not
depend on the root. Translation covariance now yields
\[
 \mu_z^1(E_{\rm bad})=\mu_0^1(E_{\rm bad})=0
 \qquad(z\in\mathbb C).
\]
After forgetting the root in the loop coordinate,
\cite[Theorem~4.2(i), (4.6)]{Zhan} gives
\[
 \mu^0(d\gamma)\,m_\gamma(dz)=\mu_z^1(d\gamma)\,dA(z).
\]
Tonelli's theorem gives
\[
 0=\int_{\mathbb C}\mu_z^1(E_{\rm bad})\,dA(z)
  =\int_{E_{\rm bad}}L_\gamma\,\mu^0(d\gamma).
\]
For every integer $m\geq1$,
\[
 \mu^0(E_{\rm bad}\cap\{L_\gamma\geq1/m\})
 \leq m\int_{E_{\rm bad}}L_\gamma\,\mu^0(d\gamma)=0.
\]
Since $L_\gamma>0$ almost everywhere, taking the union over $m$ gives
$\mu^0(E_{\rm bad})=0$.

For loops surrounding $0$, write $\Omega_\gamma$ for the bounded
component.
The scale and independent exterior-rotation disintegrations
\cite[Appendix~B.2.3, (B.47)--(B.48)]{Response} transfer this null set
to the normalized welding law. Explicitly, an exterior rotation by
$\vartheta$ and the normalization defining $Y$ replace $g_\gamma$ by
\[
 g(z)=\frac{g_\gamma(e^{i\vartheta}z)-b_\gamma}
            {a_\gamma e^{i\vartheta}},
 \qquad
 [g]_{\beta_\kappa}=a_\gamma^{-1}[g_\gamma]_{\beta_\kappa}.
\]
Affine invariance of $E_{\rm bad}$ and the scale interval $[0,1]$
therefore give, for a normalization constant $c_\kappa>0$,
\begin{align*}
 0&=\mu^0\!\left(E_{\rm bad}\cap
       \{0\in\Omega_\gamma,\ 0\leq\log\operatorname{crad}_{\Omega_\gamma}(0)\leq1\}\right)\\
  &=c_\kappa\,\nu^Y\{(h,f,g):[g]_{\beta_\kappa}=\infty\}.
\end{align*}
This proves the assertion with the fixed exponent $\beta_\kappa$.
\end{proof}

\subsection{The H\"older locus and removability}

For $0<\alpha\leq1$ and $M\geq1$, put
\[
 Y_{\alpha,M}=\{(h,f,g)\in Y:[g]_\alpha\leq M\}.
\]
Define
\[
 Y_H=\bigcup_{\substack{\alpha\in\mathbb Q\cap(0,1]\\M\in\mathbb N}}
                 Y_{\alpha,M},
\]
and let $W_H=\prX(Y_H)$.

\begin{lemma}\label{lem:Holder-locus}
The projected H\"older locus satisfies
\[
 W_H\in\mathcal B(X)
 \qquad\text{and}\qquad
 W_H\subset W_{\rm rem}.
\]
\end{lemma}
\begin{proof}
Take $(h,f,g)\in Y_{\alpha,M}$. Choose $z_*$ in the bounded
component of $\mathbb C\setminus g(\T)$ and put
\[
 r_*=\operatorname{dist}(z_*,g(\T))>0.
\]
Define the conformal map on $\mathbb D$, continuous on
$\overline{\mathbb D}$, by
\[
 \widetilde g(w)=
 \begin{cases}
  \bigl(g(1/w)-z_*\bigr)^{-1},&0<|w|\leq1,\\
  0,&w=0
 \end{cases}.
\]
It satisfies
\[
 \widetilde g'(0)=1,\qquad
 \|\widetilde g\|_{C(\overline{\mathbb D})}\leq r_*^{-1}.
\]
On the boundary,
\[
 |\widetilde g(\zeta)-\widetilde g(\xi)|
 \leq r_*^{-2}M|\zeta-\xi|^\alpha
 \qquad(\zeta,\xi\in\T).
\]
Set $\alpha'=\alpha/2$ and
$C_*=C_\alpha(r_*^{-1}+r_*^{-2}M)$, increasing
$C_\alpha$ as needed below. For the circular Poisson kernel
\[
 p_r(t)=\frac{1-r^2}{1-2r\cos t+r^2},
\]
the Poisson formula and kernel cancellation give
\begin{align*}
 \widetilde g(re^{i\theta})-\widetilde g(e^{i\theta})
 &=\frac1{2\pi}\int_{-\pi}^{\pi}p_r(t)
       \bigl(\widetilde g(e^{i(\theta-t)})-\widetilde g(e^{i\theta})\bigr)\,dt,\\
 \shortintertext{and, by holomorphicity,}
 e^{i\theta}\widetilde g'(re^{i\theta})
 &=\frac1{2\pi}\int_{-\pi}^{\pi}\partial_rp_r(t)
       \bigl(\widetilde g(e^{i(\theta-t)})-\widetilde g(e^{i\theta})\bigr)\,dt.
\end{align*}
Here $\int_{-\pi}^{\pi}p_r=2\pi$ and
$\int_{-\pi}^{\pi}\partial_rp_r=0$. For $1/2\leq r<1$, put $a=1-r$.
The kernel bounds are
\[
 p_r(t)\leq\frac{Ca}{a^2+t^2},
 \qquad
 |\partial_rp_r(t)|\leq\frac{C}{a^2+t^2}
 \qquad(|t|\leq\pi).
\]
Since $0<\alpha'<1$, the bound
$|e^{it}-1|^\alpha\leq2^{\alpha-\alpha'}|t|^{\alpha'}$
and the substitution $t=as$ give
\begin{align*}
 \int_{-\pi}^{\pi}p_r(t)|e^{it}-1|^\alpha\,dt
 &\leq C_\alpha a^{\alpha'}
       \int_0^\infty\frac{s^{\alpha'}}{1+s^2}\,ds
 \leq C_\alpha a^{\alpha'},\\
 \shortintertext{and}
 \int_{-\pi}^{\pi}|\partial_rp_r(t)|\,|e^{it}-1|^\alpha\,dt
 &\leq C_\alpha a^{\alpha'-1}
       \int_0^\infty\frac{s^{\alpha'}}{1+s^2}\,ds
 \leq C_\alpha a^{\alpha'-1}.
\end{align*}
The common integral is finite because
\[
 \int_0^\infty\frac{s^{\alpha'}}{1+s^2}\,ds
 \leq\int_0^1s^{\alpha'}\,ds+\int_1^\infty s^{\alpha'-2}\,ds
 =\frac1{1+\alpha'}+\frac1{1-\alpha'}.
\]
For $r\leq1/2$, boundedness and Cauchy's derivative estimate give
the same bounds after increasing $C_*$. Consequently,
\begin{align}
 |\widetilde g(re^{i\theta})-\widetilde g(e^{i\theta})|
 &\leq C_*(1-r)^{\alpha'}
 \notag\\
 \shortintertext{and}
 |\widetilde g'(re^{i\theta})|
 &\leq C_*(1-r)^{\alpha'-1}.
 \label{eq:Holder-domain-radial-gradient}
\end{align}
For distinct $z,w\in\mathbb D$, set
\[
 \delta=|z-w|,\qquad \rho_0=\min(1-|z|,1-|w|).
\]
If $\delta\leq\rho_0/2$, the segment between the points remains at
distance at least $\rho_0$ from $\T$, and
\begin{align*}
 |\widetilde g(z)-\widetilde g(w)|
 &=\left|\int_0^1\widetilde g'(z+s(w-z))(w-z)\,ds\right|\\
 &\leq C_*\delta\rho_0^{\alpha'-1}
 \leq C_*\delta^{\alpha'}.
\end{align*}
If $\rho_0/2<\delta<1/4$, then both points have modulus greater than
$1/4$. Put $z^*=z/|z|$ and $w^*=w/|w|$. We have
\[
 1-|z|\leq3\delta,\qquad 1-|w|\leq3\delta,
 \qquad |z^*-w^*|\leq8\delta.
\]
The radial bound and boundary estimate yield
\begin{align*}
 |\widetilde g(z)-\widetilde g(w)|
 &\leq |\widetilde g(z)-\widetilde g(z^*)|
       +|\widetilde g(z^*)-\widetilde g(w^*)|
       +|\widetilde g(w^*)-\widetilde g(w)|\\
 &\leq 2C_*(3\delta)^{\alpha'}+r_*^{-2}M(8\delta)^\alpha
 \leq C_\alpha C_*\delta^{\alpha'}.
\end{align*}
For $\delta\geq1/4$,
\[
 |\widetilde g(z)-\widetilde g(w)|
 \leq2\|\widetilde g\|_\infty
 \leq2\,4^{\alpha'}\|\widetilde g\|_\infty\delta^{\alpha'}.
\]
Continuity extends these bounds to $\overline{\mathbb D}$, proving
\[
 \sup_{\substack{z,w\in\overline{\mathbb D}\\z\ne w}}
 \frac{|\widetilde g(z)-\widetilde g(w)|}{|z-w|^{\alpha'}}<\infty.
\]
Thus $\widetilde g(\mathbb D)$ is a bounded H\"older domain.
Corollary~2 of \cite{JS} makes its boundary quasiconformally removable,
equivalently conformally removable in the plane
\cite[Definition~1 and the following paragraph]{JS}.
M\"obius invariance then gives removability of $g(\T)$.

For fixed $\alpha,M$, continuity of boundary evaluations gives
\[
 Y_{\alpha,M}
 =\bigcap_{\substack{\zeta,\xi\in\T_{\mathbb Q}\\\zeta\ne\xi}}
   \{(h,f,g)\in Y:
        |g(\zeta)-g(\xi)|\leq M|\zeta-\xi|^\alpha\}.
\]
Each set in the intersection is closed. Hence $Y_H$ is a countable
union of closed subsets of $Y$. Uniqueness of the normalized welding
of a removable curve makes $\prX|_{Y_H}$ injective. The complete
separable realizations of $X,Y$ and the Lusin--Souslin theorem give
\[
 W_H=\prX(Y_H)\in\mathcal B(X),\qquad W_H\subset W_{\rm rem}.
\]
\end{proof}

\subsection{Final proof of Theorem~\ref{thm:removability-main}}
\label{sec:proof-removability}

Put $T=T_*$ and choose a single rational $\beta\in(0,\beta_\kappa)$.
Start the diffusion of Theorem~\ref{thm:full-main} from $\nu^Y$, and
write $X_t=\prX(Y_t)$. Its projected path law is the stationary
response law. Stationarity and Lemma~\ref{lem:subcritical-static-Holder}
give, for every integer $\ell\geq0$,
\[
 [g_{\ell T}]_\beta
 \leq2^{\beta_\kappa-\beta}[g_{\ell T}]_{\beta_\kappa}<\infty
 \qquad\text{almost surely}.
\]
For $I_\ell=[\ell T,(\ell+1)T]$, define the extended random variable
\[
 H_\ell^{\rm Hol}=\sup_{t\in I_\ell}[g_t]_{\beta/2}\in[0,\infty].
\]
Joint continuity gives its measurable countable representation
\[
 H_\ell^{\rm Hol}
 =\sup_{q\in\mathbb Q\cap I_\ell}
   \ \sup_{\substack{\zeta,\xi\in\T_{\mathbb Q}\\\zeta\ne\xi}}
       \frac{|g_q(\zeta)-g_q(\xi)|}{|\zeta-\xi|^{\beta/2}}.
\]
Stationarity of the path law gives, in $C([0,T],Y)$,
\[
 \operatorname{Law}\bigl((Y_{\ell T+t})_{0\leq t\leq T}\bigr)
 =\operatorname{Law}\bigl((Y_t)_{0\leq t\leq T}\bigr).
\]
Thus Corollary~\ref{cor:geometric-lift-estimates} and the initial
seminorm bound imply $\mathbb P(H_\ell^{\rm Hol}=\infty)=0$ for every $\ell$.
The countable union bound yields
\[
 \mathbb P\left(\bigcup_{\ell\geq0}\{H_\ell^{\rm Hol}=\infty\}\right)
 \leq\sum_{\ell\geq0}\mathbb P(H_\ell^{\rm Hol}=\infty)=0.
\]
On the complementary event, put $N_\ell=1\vee\lceil H_\ell^{\rm Hol}\rceil$.
Since $\beta/2\in\mathbb Q\cap(0,1]$, the definition of $Y_H$ gives
\begin{align*}
 Y_t&\in Y_{\beta/2,N_\ell}\subset Y_H\qquad(t\in I_\ell),\\
 \shortintertext{and hence}
 X_t&\in W_H\qquad(t\geq0).
\end{align*}

For every compact $K\subset X\setminus W_H$, the preceding event
excludes a hit at any time. Thus
\begin{align*}
 \mathbb P(\sigma_K<\infty)
 &\leq\mathbb P\left(\bigcup_{\ell\geq0}\{H_\ell^{\rm Hol}=\infty\}\right)=0,\\
 \shortintertext{and \eqref{eq:stationary-capacity-compact} gives}
 \CapOne(K)&=\mathbb E e^{-\sigma_K}=0.
\end{align*}
The $1$-capacity is a tight Choquet capacity on $X$. Since
$W_H\in\mathcal B(X)$ by Lemma~\ref{lem:Holder-locus}, capacitability
\cite[Theorem~2.1.1 and Corollary~A.1.1]{FOT} yields
\[
 \CapOne(X\setminus W_H)
 =\sup_{\substack{K\subset X\setminus W_H\\K\ \mathrm{compact}}}
       \CapOne(K)=0.
\]
Finally,
\[
 \CapOne(X\setminus W_{\rm rem})
 \leq\CapOne(X\setminus W_H)=0.
\]
This proves \eqref{eq:subcritical-removable-capacity} and completes
the proof of Theorem~\ref{thm:removability-main}.

\begin{remark}
The continuous Jordan realization remains valid at $\kappa=4$, but
the proof above does not establish
$\CapOne(X\setminus W_{\rm rem})=0$ there. Static conformal removability
holds for SLE$_4$ \cite{KMS}; the positive H\"older exponent required
by the present propagation argument is unavailable at the critical
endpoint \cite{KMSreg}. The space $Y$ still records the actual
conformal pair without requiring its uniqueness at every time.
\end{remark}

\section{Irreducibility and equilibrium}
\label{sec:ergodicity}

For Theorem~\ref{thm:ergodicity}, the key step is to show that every
$F\in\Dcal$ with $DF=0$ equals a real constant $\nu$-almost everywhere.
Horizontal integration by parts determines coefficient pairings, with
inverse powers of the exterior derivative retaining the relative rotation.

\subsection{Intrinsic coordinates and coefficient responses}

Fix $0<\kappa\leq4$. Use the intrinsic pair
$(f_h^\circ,g_h^\circ)$ from Section~\ref{sec:welding-support}, with
\[
 f_h^\circ(0)=0,\qquad (f_h^\circ)'(0)=1,\qquad
 g_h^\circ(\infty)=\infty,
\]
and $h=(g_h^\circ)^{-1}\circ f_h^\circ$ on $\T$.
Write $f^\circ,g^\circ$ when $h$ is fixed, and set
\[
 f^\circ(z)=z+\sum_{n\geq1}a_nz^{n+1}.
\]
Retain $a_\infty=(g^\circ)'(\infty)$ and write
\begin{equation}\label{eq:intrinsic-exterior-mark}
 a_\infty=e^k\chi,\qquad k=\log|a_\infty|,\quad\chi\in S^1.
\end{equation}
The area formula and the exterior area theorem give
\[
 \pi\leq\operatorname{Area}(f^\circ(\mathbb D))
       \leq\pi|a_\infty|^2,
 \qquad k\geq0.
\]
The intrinsic pairs on $W_{\rm rem}$ agree with the measurable choice
on $X_{\rm nf}$. Their coefficient functions therefore have measurable
versions on $X$. Smooth right composition preserves $W_{\rm rem}$
\cite[Lemma~3.1]{Response}; all orbit derivatives below use these
intrinsic pairs.

Let $\nu^\#$ be the law of the normalized interior map $f^\circ$.
The marked welding parametrization is a measurable identification
modulo null sets, and, with $\chi=e^{i\alpha}$ and $\alpha\in\T$,
\begin{equation}\label{eq:intrinsic-mark-product}
 d\nu=d\nu^\#(f^\circ)\,\frac{d\alpha}{2\pi}.
\end{equation}
This is \cite[Definition~1.1]{BJwelding}, in the normalization of
\cite[Appendix~B.2.3, (B.48)--(B.49)]{Response}. In particular,
$k=k(f^\circ)$ is measurable and finite almost everywhere.

Let $\mathcal P=\mathbb C[a_1,a_2,\ldots]$, with weighted degree
$\deg a_n=n$, and let $\mathcal P_N$ be its homogeneous part of degree
$N$. Put $\mathcal P_N=\{0\}$ for $N<0$.
The full class of univalent functions on $\mathbb D$ normalized by
$f^\circ(0)=0$ and $(f^\circ)'(0)=1$ is compact in the topology of
locally uniform convergence. Coefficients are continuous on this
class, and $|a_n|\leq n+1$ makes every polynomial bounded.

Retain $c_{\mathrm L}$ and the Schwarzian $\mathcal S$.
For an integer $m$, write $[z^m]$ for Laurent coefficient extraction
at zero. For $n\geq1$, define
\[
 b_n=[z^{n-2}]\left(\frac{((f^\circ)'(z))^2}{f^\circ(z)^2}
                         -\frac1{z^2}\right),\qquad
 s_n=[z^{n-2}]\mathcal S f^\circ(z).
\]
Then $b_1=2a_1$, $s_1=0$, and the stress coefficients from
Section~\ref{sec:stress-drift} satisfy
\begin{equation}\label{eq:coefficient-stress-splitting}
 t_n=b_n+\frac{c_{\mathrm L}}{12}s_n\qquad(n\geq2).
\end{equation}
The coefficient derivations on $\mathcal P$ are
\begin{equation}\label{eq:coefficient-positive-derivation}
 \ell_n=\partial_{a_n}+\sum_{m\geq1}(m+1)a_m\partial_{a_{m+n}},
 \qquad n\geq1.
\end{equation}
Fix $0<r<1$. For $n\geq1$, the negative derivation $\ell_{-n}$
is determined, for $z$ sufficiently close to zero, by
\begin{equation}\label{eq:coefficient-negative-derivation}
 \sum_{m\geq1}(\ell_{-n}a_m)z^{m+1}
 =\frac{f^\circ(z)^2}{2\pi i}\oint_{|w|=r}
 \frac{((f^\circ)'(w))^2w^{1-n}}
 {f^\circ(w)^2(f^\circ(w)-f^\circ(z))}\,dw.
\end{equation}
Each coefficient on the right is a polynomial in finitely many $a_m$.

For a function differentiable along a source flow, use the actual
right derivative in \eqref{eq:actual-right-derivative}. Its complex
combinations are
\[
 D_n^\pm=D_{\cos(n\cdot)}\pm iD_{\sin(n\cdot)}
       =\alpha_n^{-1}(D_{n,c}\pm iD_{n,s}),\qquad n\geq2,
\]
with $\alpha_n$ from \eqref{eq:onb}. The coefficient response and score
identities \cite[Lemmas~4.2--4.3 and Corollary~5.2]{Response} give
\begin{align}
 D_n^+p&=i\ell_np,\qquad D_n^-p=i\ell_{-n}p
       \quad(p\in\mathcal P),\notag\\
 \shortintertext{and}
 \alpha_n^{-1}(\rho_{n,c}+i\rho_{n,s})
       &=i\overline{t_n},\qquad n\geq2.
 \label{eq:coefficient-horizontal-response}
\end{align}

\begin{lemma}\label{lem:exterior-coefficient-response}
For $n\geq2$,
\begin{equation}\label{eq:exterior-leading-response}
 D_n^+a_\infty=0,\qquad D_n^-a_\infty=-ib_na_\infty.
\end{equation}
For every $j\in\mathbb Z_{\geq0}$ and $p\in\mathcal P$,
\begin{align}
 D_n^+(a_\infty^{-j}p)&=ia_\infty^{-j}\ell_np,\notag\\
 \shortintertext{and}
 D_n^-(a_\infty^{-j}p)
     &=ia_\infty^{-j}(\ell_{-n}p+jb_np).
 \label{eq:inverse-power-response}
\end{align}
These functions and their first horizontal responses are bounded.
They are continuously differentiable along every finite horizontal
source flow on $W_{\rm rem}$.
\end{lemma}

\begin{proof}
Fix $0<r<1$ and give $\gamma_r=f^\circ(r\T)$ its positive orientation.
For the negative complex response, the normalized first-variation
splitting of \cite[Appendix~A.3]{Response} has jump
\[
 u(f^\circ(w))=i(f^\circ)'(w)w^{1-n}\qquad(|w|=r).
\]
Define its Cauchy integral off $\gamma_r$ by
\[
 C_u(\xi)=\frac1{2\pi i}\oint_{\gamma_r}
                   \frac{u(\eta)}{\eta-\xi}\,d\eta.
\]
Its interior and exterior branches have difference $u$ across the
contour. Subtract $C_u(0)+C_u'(0)\xi$ from both branches, where the
values at zero are taken in the interior. The resulting interior
field vanishes to second order at zero, and the exterior field is
$O(\xi)$ at infinity. These conditions uniquely determine the
normalized splitting: the difference of two such splittings is a
holomorphic vector field on the sphere, with zero value and derivative
at zero and at most linear growth at infinity. Thus
\[
 D_n^-g^\circ(z)
 =C_u(g^\circ(z))-C_u(0)-C_u'(0)g^\circ(z).
\]
For either real component of the deformation, the exterior coefficient
is extracted on any fixed circle $|z|=R>1$ by
\[
 a_\infty(t)=\frac1{2\pi i}\oint_{|z|=R}g_t^\circ(z)z^{-2}\,dz.
\]
The Ahlfors--Bers parameter derivative is locally uniform on this
circle, so it passes through the integral. Since
$C_u(\xi)=O(\xi^{-1})$ at infinity,
\begin{align*}
 \frac{D_n^-a_\infty}{a_\infty}
 &=-C_u'(0)\\
 &=-\frac{i}{2\pi i}\oint_{|w|=r}
       \frac{((f^\circ)'(w))^2w^{1-n}}{f^\circ(w)^2}\,dw
 =-ib_n.
\end{align*}
For the positive response, the jump is
$i(f^\circ)'(w)w^{n+1}$. It extends holomorphically to the interior
image and vanishes to second order at zero, so its normalized
exterior field is zero. This proves
\eqref{eq:exterior-leading-response}.

The chain rule gives \eqref{eq:inverse-power-response}. The same
parameter construction gives continuous differentiability along the
real source flows. Since $|a_\infty^{-j}|\leq1$ and all coefficient
polynomials are bounded, the two displayed responses are bounded.
Sums and differences give the bounds in each finite real horizontal
direction. Products and complex conjugates retain these properties.
\end{proof}

\subsection{Horizontal flows and integration by parts}

For a real finite horizontal field $v$, set
\[
 \Phi_t^vh=h\circ\psi_t^v,\qquad U_t^vH=H\circ\Phi_t^v,
\]
and let $K_v$ be the closed generator of $(U_t^v)_{t\in\mathbb R}$
on $L^2(\nu)$. The finite likelihood formula
\cite[Lemma~6.1 and Proposition~6.2]{Response} gives a strongly continuous
group of bounded operators, with $K_vH=D_vH$ on $\Ccyl$ and
\[
 \|U_t^v\|_{2\to2}^2\leq\exp(|t|\,\|\rho_v\|_\infty).
\]
Writing $R_t^v$ for its Radon--Nikodym density, the same formula gives
\begin{equation}\label{eq:horizontal-source-likelihood}
 R_t^v:=\frac{d(\Phi_t^v)_*\nu}{d\nu}
 =\exp\left(\int_{-t}^0\rho_v\circ\Phi_s^v\,ds\right).
\end{equation}
The bounded response of $\rho_v$ implies
\begin{equation}\label{eq:horizontal-source-density-limit}
 \left\|\frac{R_t^v-1}{t}-\rho_v\right\|_\infty\longrightarrow0
 \qquad(t\to0).
\end{equation}

Let $F\in\Dcal\cap L^\infty(\nu)$ be real, with $DF=0$.
In the next lemma, $H$ is any finite linear combination of products
of the functions $a_\infty^{-j}p$, $j\in\mathbb Z_{\geq0}$,
$p\in\mathcal P$, and their complex conjugates.

\begin{lemma}\label{lem:zero-energy-coefficient-ibp}
For every finite real horizontal field $v$,
\begin{equation}\label{eq:zero-energy-coefficient-ibp}
 \int_X F D_vH\,d\nu=\int_X F\rho_vH\,d\nu.
\end{equation}
\end{lemma}

\begin{proof}
Choose $F^{(m)}\in\Ccyl$ with
\[
 \|F^{(m)}-F\|_2+
 \|DF^{(m)}-DF\|_{L^2(\nu;\HWP)}\longrightarrow0.
\]
Then
\[
 K_vF^{(m)}=(DF^{(m)},v)_{\WP}
       \longrightarrow(DF,v)_{\WP}=0\quad\text{in }L^2(\nu).
\]
Closedness of $K_v$ gives $F\in\Dom(K_v)$ and $K_vF=0$. Hence
\begin{equation}\label{eq:zero-energy-horizontal-invariance}
 U_t^vF-F=\int_0^t U_s^vK_vF\,ds=0
 \quad\text{in }L^2(\nu),\qquad t\in\mathbb R.
\end{equation}
For each fixed $t$, change of variables gives
\begin{align*}
 \int_X F(H\circ\Phi_t^v)\,d\nu
 &=\int_X(F\circ\Phi_{-t}^v)H R_t^v\,d\nu\\
 &=\int_X FH R_t^v\,d\nu.
\end{align*}
Lemma~\ref{lem:exterior-coefficient-response} gives
\[
 \frac{H\circ\Phi_t^v-H}{t}
 =\int_0^1(D_vH)\circ\Phi_{st}^v\,ds,
 \qquad
 \left|\frac{H\circ\Phi_t^v-H}{t}\right|\leq C_{H,v}.
\]
Subtract the identity at $t=0$, divide by $t$, and use dominated
convergence on the left and
\eqref{eq:horizontal-source-density-limit} on the right. This proves
\eqref{eq:zero-energy-coefficient-ibp}.
\end{proof}

\subsection{Polynomial representations and pairings}

Write $\operatorname{Id}_{\mathcal P}$ for the identity operator on $\mathcal P$. For
$j\in\mathbb Z_{\geq0}$, set $L_n^{(j)}=\ell_n$ for $n>0$, and define
\[
 L_0^{(j)}=\sum_{m\geq1}ma_m\partial_{a_m}+(j+1)\operatorname{Id}_{\mathcal P}.
\]
For the negative modes put
\begin{equation}\label{eq:coefficient-representation}
 L_{-n}^{(j)}p
 =\ell_{-n}p+\left((j+1)b_n+\frac{c_{\mathrm L}}{12}s_n\right)p,
 \qquad n\geq1.
\end{equation}
These act on coefficient polynomials; $L$ remains the diffusion
generator. With the operator commutator $[A,B]=AB-BA$, the coefficient
formulas of \cite[Section~0.3, (0.2)--(0.8)]{AN} give
\begin{equation}\label{eq:coefficient-virasoro-relations}
 [L_m^{(j)},L_n^{(j)}]
 =(m-n)L_{m+n}^{(j)}
 +\frac{c_{\mathrm L}}{12}(m^3-m)\mathbf 1_{\{m+n=0\}}\operatorname{Id}_{\mathcal P}.
\end{equation}
Indeed, the multiplication term in that realization is the
coefficient of $z^n$ in
\[
 z^2\left((j+1)\frac{((f^\circ)'(z))^2}{f^\circ(z)^2}
               +\frac{c_{\mathrm L}}{12}\mathcal S f^\circ(z)\right),
\]
namely $(j+1)b_n+c_{\mathrm L}s_n/12$ for $n>0$, with constant term
$j+1$. For $N\geq0$, the degree convention gives
\[
 L_n^{(j)}\mathcal P_N\subset\mathcal P_{N-n}
 \qquad(n\in\mathbb Z).
\]
The constant polynomial satisfies
\[
 L_n^{(j)}1=0\quad(n>0),\qquad L_0^{(j)}1=(j+1)1.
\]
In particular,
\begin{align*}
 [L_3^{(j)},L_{-2}^{(j)}]&=5L_1^{(j)},\\
 [L_2^{(j)},L_{-3}^{(j)}]&=5L_{-1}^{(j)},\\
 \shortintertext{and}
 [L_2^{(j)},L_{-2}^{(j)}]&=4L_0^{(j)}+\frac{c_{\mathrm L}}2\operatorname{Id}_{\mathcal P}.
\end{align*}

The highest-weight vector $1$ generates $\mathcal P$.
To verify this, put $b_\kappa=\sqrt\kappa/2$, so
$c_{\mathrm L}=1+6(b_\kappa+b_\kappa^{-1})^2\geq25$.
The singular-vector criterion \cite[(1)--(5)]{Kent}, with its
parameter equal to $-b_\kappa^2$, gives, for positive integers $r,s$, the possible highest weights
\begin{equation}\label{eq:positive-weight-nonresonance}
 h_{r,s}=\frac{(b_\kappa+b_\kappa^{-1})^2-(rb_\kappa+sb_\kappa^{-1})^2}{4}\leq0,
 \qquad r,s\geq1.
\end{equation}
Thus the Verma module $V(c_{\mathrm L},j+1)$ has no singular vector.
A nonzero proper submodule would have a nonzero vector of least
positive degree, annihilated by all positive modes; hence the module
is irreducible. The homomorphism
\[
 V(c_{\mathrm L},j+1)\longrightarrow\mathcal P
\]
that sends its highest-weight vector to $1$ is nonzero and therefore
injective. At level $N$, both spaces have dimension equal to the
number of partitions of $N$, so the homomorphism is surjective at
every level. Consequently,
\begin{equation}\label{eq:coefficient-cyclicity}
 \mathcal P=\operatorname{span}
 \{L_{-n_1}^{(j)}\cdots L_{-n_r}^{(j)}1:
                  r\geq0,\ n_1,\ldots,n_r\geq1\}.
\end{equation}

Fix $j,\ell\in\mathbb Z_{\geq0}$ and a sesquilinear map
$B:\mathcal P\times\mathcal P\to\mathbb C$, linear in its first
variable. Assume that
\begin{equation}\label{eq:horizontal-coefficient-pairing}
 B(L_n^{(j)}p,q)=B(p,L_{-n}^{(\ell)}q)
 \qquad(|n|\geq2,\ p,q\in\mathcal P).
\end{equation}

\begin{lemma}\label{lem:coefficient-pairing-uniqueness}
The pairing satisfies
\[
 B(L_n^{(j)}p,q)=B(p,L_{-n}^{(\ell)}q)
 \qquad(n\in\mathbb Z,\ p,q\in\mathcal P).
\]
If $j\ne\ell$, then $B=0$. If $j=\ell$, the value $B(1,1)$ uniquely
determines $B$.
\end{lemma}

\begin{proof}
For $|m|,|n|\geq2$, two applications of
\eqref{eq:horizontal-coefficient-pairing} give
\begin{align*}
 B(L_m^{(j)}L_n^{(j)}p,q)
 &=B(L_n^{(j)}p,L_{-m}^{(\ell)}q)\\
 &=B(p,L_{-n}^{(\ell)}L_{-m}^{(\ell)}q).
\end{align*}
Subtracting the identity with $m,n$ exchanged yields
\[
 B([L_m^{(j)},L_n^{(j)}]p,q)
 =B(p,[L_{-n}^{(\ell)},L_{-m}^{(\ell)}]q).
\]
The pairs $(m,n)=(3,-2),(2,-3)$ give the relations for $1,-1$.
For $(m,n)=(2,-2)$, the common real central term cancels, giving
the relation for $0$. In particular,
\[
 (j-\ell)B(1,1)=0.
\]
By \eqref{eq:coefficient-cyclicity}, a polynomial of positive degree
is a sum of negative-mode descendants. For $n\geq1$ and $q_0\in\mathcal P$,
\[
 B(1,L_{-n}^{(\ell)}q_0)=B(L_n^{(j)}1,q_0)=0.
\]
For a descendant in the first variable, repeated transfer gives
\[
 B(L_{-n_1}^{(j)}\cdots L_{-n_r}^{(j)}1,q)
 =B(1,L_{n_r}^{(\ell)}\cdots L_{n_1}^{(\ell)}q).
\]
If $q\in\mathcal P_M$, the polynomial on the right lies in
$\mathcal P_{M-\sum_i n_i}$. A positive degree gives zero by the
preceding identity, degree zero gives a scalar multiple of $B(1,1)$,
and a negative degree gives zero. Cyclicity proves both conclusions.
\end{proof}

\subsection{Final proof of Theorem~\ref{thm:ergodicity}}
\label{sec:proof-ergodicity}

Let $F\in\Dcal\cap L^\infty(\nu)$ be real, with $DF=0$. Choose a
bounded measurable representative and put $c_F=\int_XF\,d\nu$.
For $j,\ell\in\mathbb Z_{\geq0}$, define
\begin{equation}\label{eq:marked-coefficient-pairing}
 B_{j,\ell}^F(p,q)=\int_X
 F a_\infty^{-j}\overline{a_\infty}^{-\ell}p\overline q\,d\nu,
 \qquad p,q\in\mathcal P.
\end{equation}
These integrals are finite. Apply
\eqref{eq:zero-energy-coefficient-ibp} in the cosine and sine directions
to $H=a_\infty^{-j}p\,\overline{a_\infty^{-\ell}q}$.
Equation~\eqref{eq:inverse-power-response} gives
\begin{align*}
 D_n^+H
 &=ia_\infty^{-j}(\ell_np)\,
                 \overline{a_\infty^{-\ell}q}\\
 &\quad-ia_\infty^{-j}p\,
          \overline{a_\infty^{-\ell}(\ell_{-n}q+\ell b_nq)}.
\end{align*}
By \eqref{eq:coefficient-horizontal-response}, its integral against
$F\,d\nu$ is $i\int_XFH\overline{t_n}\,d\nu$. Therefore, for $n\geq2$,
\begin{align}
 B_{j,\ell}^F(\ell_np,q)
 &=B_{j,\ell}^F\bigl(p,(\ell_{-n}+\ell b_n+t_n)q\bigr)\notag\\
 &=B_{j,\ell}^F(p,L_{-n}^{(\ell)}q).
 \label{eq:marked-positive-pairing}
\end{align}
Reality of $F$ gives
\[
 \overline{B_{\ell,j}^F(q,p)}=B_{j,\ell}^F(p,q).
\]
Applying \eqref{eq:marked-positive-pairing} with the indices exchanged,
\begin{align*}
 B_{j,\ell}^F(L_{-n}^{(j)}p,q)
 &=\overline{B_{\ell,j}^F(q,L_{-n}^{(j)}p)}\\
 &=\overline{B_{\ell,j}^F(\ell_nq,p)}
 =B_{j,\ell}^F(p,\ell_nq).
\end{align*}
Thus both signs in \eqref{eq:horizontal-coefficient-pairing} hold.
Lemma~\ref{lem:coefficient-pairing-uniqueness} implies
\begin{equation}\label{eq:marked-cross-pairing-zero}
 B_{j,\ell}^F(p,q)=0\qquad(j\ne\ell).
\end{equation}
For $j=\ell=0$, comparison with $F=\mathbf 1$ gives
\begin{equation}\label{eq:marked-zero-pairing}
 B_{0,0}^F(p,q)=c_FB_{0,0}^{\mathbf 1}(p,q).
\end{equation}

Use the product coordinates \eqref{eq:intrinsic-mark-product}. For
$m\in\mathbb Z$, define the Fourier coefficients of $F$ in the mark by
\[
 \widehat F_m(f^\circ)=\frac1{2\pi}\int_0^{2\pi}
        F(f^\circ,e^{i\alpha})e^{im\alpha}\,d\alpha.
\]
They are measurable and bounded by $\|F\|_\infty$.
Equation~\eqref{eq:marked-zero-pairing} gives
\[
 \int(\widehat F_0-c_F)p\overline q\,d\nu^\#=0
 \qquad(p,q\in\mathcal P).
\]
The products $p\overline q$ span a conjugation-stable, point-separating
algebra containing the constants on the full compact normalized
schlicht class. Regard $(\widehat F_0-c_F)\nu^\#$ as a finite signed
measure on that class. Stone--Weierstrass and uniqueness of finite
measures give
\[
 \widehat F_0=c_F\quad\nu^\#\text{-almost everywhere}.
\]
For $m\geq1$, use $(j,\ell)=(0,m)$ in
\eqref{eq:marked-cross-pairing-zero}. Since
$\overline{a_\infty}^{-m}=e^{-mk}e^{im\alpha}$, we obtain
\[
 \int e^{-mk(f^\circ)}\widehat F_m(f^\circ)p\overline q\,
           d\nu^\#(f^\circ)=0\qquad(p,q\in\mathcal P).
\]
The finite complex measure $e^{-mk}\widehat F_m\nu^\#$ has total
variation at most $\|F\|_\infty$. The same measure-determining argument
gives
\[
 e^{-mk}\widehat F_m=0\quad\nu^\#\text{-almost everywhere}.
\]
As $k<\infty$ almost everywhere, $e^{-mk}>0$ and $\widehat F_m=0$.
Reality gives $\widehat F_{-m}=\overline{\widehat F_m}=0$.
Take a common full-measure set for all integer frequencies. On each
remaining circle fibre, $F(f^\circ,\cdot)\in L^2(\T)$, so Fourier
completeness yields
\begin{equation}\label{eq:bounded-zero-energy-rigidity}
 F=c_F\quad\nu\text{-almost everywhere}.
\end{equation}

For an arbitrary real $F\in\Dcal$ with $DF=0$, set
$F_M=(-M)\vee(F\wedge M)$. The contraction property gives
\[
 \|DF_M\|_2^2=\E(F_M,F_M)\leq\E(F,F)=0.
\]
By \eqref{eq:bounded-zero-energy-rigidity}, each $F_M$ is constant, and
\[
 \left\|F-\int_XF\,d\nu\right\|_2
 \leq2\|F-F_M\|_2\longrightarrow0\qquad(M\to\infty).
\]
If $F\in\Dom(L)$ and $LF=0$, then
\[
 \|DF\|_2^2=\E(F,F)=-\langle LF,F\rangle_2=0.
\]
Conversely, $\mathbf 1\in\Dom(L)$ and $L\mathbf 1=0$ by conservativeness. Thus
\[
 \{F\in\Dcal:DF=0\}=\ker L
 =\{F\in L^2(\nu):F=c\ \nu\text{-a.e. for some }c\in\mathbb R\},
\]
which proves \eqref{eq:intro-zero-energy-kernel}. For a measurable
semigroup-invariant set $A\subset X$,
\[
 T_t\mathbf 1_A=\mathbf 1_A\ (t\geq0)
 \quad\Longrightarrow\quad
 \mathbf 1_A\in\ker L
 \quad\Longrightarrow\quad\nu(A)\in\{0,1\}.
\]
Hence the form is irreducible.

Work under the stationary law $\mathbb P$. For $s\geq0$, define the
shift on continuous paths by $(S_s\omega)(t)=\omega(s+t)$.
Let $A$ be an event in the canonical continuous-path sigma-algebra with
$\mathbb P(A\mathbin{\triangle}S_s^{-1}A)=0$ for every $s\geq0$.
Choose a properly exceptional set $N_{\rm erg}$ outside which the Markov
laws $\mathbb P_h$ are fixed. Put $p_A(h)=\mathbb P_h(A)$ for
$h\notin N_{\rm erg}$, and extend $p_A$ by zero on $N_{\rm erg}$.
Then $p_A$ is measurable, $0\leq p_A\leq1$, and the Markov property
and shift invariance give, for each fixed $s\geq0$,
\[
 T_sp_A(h)=\mathbb P_h(S_s^{-1}A)=p_A(h)
 \quad\text{for }\nu\text{-almost every }h.
\]
Thus $p_A\in\Dom(L)$, $Lp_A=0$, and
$p_A=\mathbb P(A)$ almost everywhere. For each fixed $t\geq0$,
\begin{align*}
 \mathbb E[\mathbf 1_A\mid\mathcal G_t]
 &=\mathbb E[\mathbf 1_{S_t^{-1}A}\mid\mathcal G_t]\\
 &=p_A(X_t)=\mathbb P(A)\quad\text{almost surely}.
\end{align*}
The continuous-path sigma-algebra is contained in
$\mathcal G_\infty=\sigma(\bigcup_{N\geq1}\mathcal G_N)$.
Martingale convergence at integer times therefore gives
\[
 \mathbf 1_A=\lim_{N\to\infty}
 \mathbb E[\mathbf 1_A\mid\mathcal G_N]
 =\mathbb P(A)\quad\text{almost surely}.
\]
Hence $\mathbb P(A)\in\{0,1\}$, proving stationary ergodicity.

Finally, let $F\in L^2(\nu)$, set $F^\circ=F-\int_XF\,d\nu$, and let
$\mu_{F^\circ}^{\rm sp}$ be its spectral measure for $-L$.
The kernel identity gives $\mu_{F^\circ}^{\rm sp}(\{0\})=0$, so
\[
 \|T_tF^\circ\|_2^2
 =\int_{(0,\infty)}e^{-2t\lambda}\,d\mu_{F^\circ}^{\rm sp}(\lambda)
 \longrightarrow0\qquad(t\to\infty)
\]
by dominated convergence. This proves \eqref{eq:ergodic-L2} and
completes the proof of Theorem~\ref{thm:ergodicity}.

\subsection{Stationary correlations and time averages}

All expectations below use the stationary law $\mathbb P$.

\begin{corollary}\label{cor:ergodic-consequences}
For $F,G\in L^2(\nu)$,
\[
 \lim_{t\to\infty}\mathbb E[F(X_0)G(X_t)]
 =\left(\int_XF\,d\nu\right)\left(\int_XG\,d\nu\right).
\]
For $F\in L^1(\nu)$,
\[
 \frac1T\int_0^T F(X_t)\,dt\longrightarrow\int_XF\,d\nu
 \quad\mathbb P\text{-almost surely as }T\to\infty.
\]
An invariant probability with a density in $L^2(\nu)$ equals $\nu$.
\end{corollary}

\begin{proof}
Stationarity and the Markov property give
\[
 \left|\mathbb E[F(X_0)G(X_t)]
             -\int_XF\,d\nu\int_XG\,d\nu\right|
 \leq\|F\|_2\left\|T_tG-\int_XG\,d\nu\right\|_2
 \longrightarrow0.
\]
The stationary ergodic theorem gives the time-average limit.
If $\varrho\in L^2(\nu)$ is the density of an invariant probability,
then, for every bounded measurable $G$, symmetry gives
\[
 \langle T_t\varrho,G\rangle_2
 =\langle\varrho,T_tG\rangle_2
 =\langle\varrho,G\rangle_2.
\]
Thus $T_t\varrho=\varrho$, and \eqref{eq:ergodic-L2} implies
$\varrho=\int_X\varrho\,d\nu=1$.
\end{proof}

\end{document}